\documentclass[a4paper,english,12pt]{amsbook}
\pdfoutput=1
\usepackage{pstricks}
\usepackage[protrusion=true,expansion=true]{microtype}
\usepackage{graphicx}
\usepackage[all]{xy}
\usepackage{hyperref}
\usepackage[utf8]{inputenc}
\usepackage{eucal}
\usepackage{mathrsfs}
\usepackage{amssymb}
\usepackage{amsxtra}
\usepackage{enumerate}
\hypersetup{colorlinks=true,citecolor=black,filecolor=black,linkcolor=blue,urlcolor=black} 
 
\newtheorem{theorem}{Theorem}[chapter]
\newtheorem{proposition}[theorem]{ Proposition} 

\newtheorem{corollary}[theorem]{Corollary}

\newtheorem{definition}[theorem]{Definition}

\theoremstyle{remark}

\newtheorem{example}[theorem]{\it Example}
 
 \def \1{\mathbb {1}}
\def \R{\mathbb {R}}
\def \RM{\mathbb {R}}
\def \NM{\mathbb{N}}
\def \CM{\mathbb{C}}

\def \Hom {{\rm Hom}}
\def \Map {{\rm Map}}

\def \Der {{\rm Der\,}}

\def \p {{\rm exp\,}}
\def \Id {{\rm Id\,}}

\def \d{\partial}
\def\dt{\delta} 
\def\a{\alpha}
\def\b{\beta}
\def\e{\varepsilon}  
\def\g{\gamma}
\def\m{\mu}

\def\l{\lambda}
\def\L{\Lambda}
\def\p{\varphi}

\def\G{\Gamma}   
\def\D{\Delta}
\def \s{\sigma}

\def \to{\longrightarrow} 

\def \alg{\mathfrak{g}}

\def\del{\nabla}
\def \< {{\langle }}
\def \> {{\rangle }}
\def \( {\left( }
\def \) {\right) }

\newcommand{\Bt}{{\mathcal B}}

\newcommand{\Dt}{{\mathcal D}}
\newcommand{\Et}{{\mathcal E}}
\newcommand{\Ft}{{\mathcal F}}
\newcommand{\Gt}{{\mathcal G}}
\newcommand{\Ht}{{\mathcal H}}
\newcommand{\It}{{\mathcal I}}

\newcommand{\Lt}{{\mathcal L}}
\newcommand{\Mt}{{\mathcal M}}

\newcommand{\Ot}{{\mathcal O}}
\newcommand{\Pt}{{\mathcal P}}

\newcommand{\Tt}{{\mathcal T}}

\newcommand{\Xt}{{\mathcal X}}

\newcommand{\lra}{\longrightarrow}

\title[Th\'eorie KAM]{{\sc  Th\'eorie KAM}}
\author{ Mauricio  Garay }

\makeindex
 
\begin{document}
 \thispagestyle{empty}
 \begin{center}{\LARGE  \bf KAM THEORY\\} 
  \vskip0.8cm { \large \sc M. Garay and D. van Straten}\\
  \vskip0.5cm
  \rule{0.3\textwidth}{0.5pt}\\
    \vskip1cm
    
 { \large \sc  \textbf{II \\ \vskip0.5cm KOLMOGOROV SPACES}}
 
 \vskip1cm
  \begin{figure}[ht]
  \begin{minipage}[b]{0.6\linewidth}   
  \includegraphics[height=0.6\linewidth,width=0.75\linewidth]{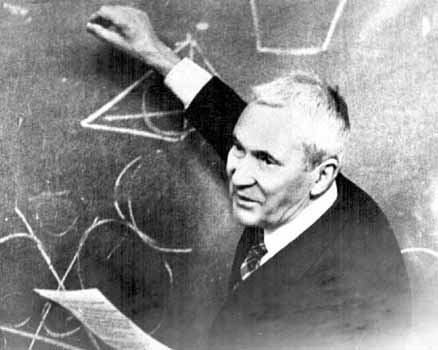}
 
  \end{minipage}
   \begin{minipage}[b]{0.60\linewidth}
     \includegraphics[height=0.42\linewidth,width=0.75\linewidth]{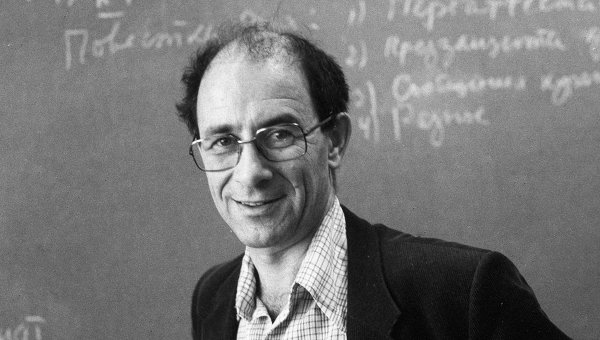}

             \end{minipage}\hfill
         \begin{minipage}[b]{0.40\linewidth}
           \includegraphics[height=0.9\linewidth,width=0.6\linewidth]{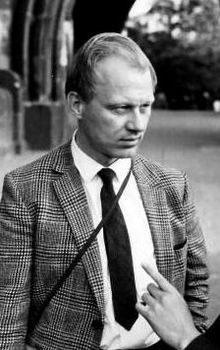}
     \end{minipage}
  \end{figure}

  \end{center}

\newpage
\ \\
\thispagestyle{empty}

\clearpage
%
%
%
\tableofcontents
\setcounter{part}{1}
\part{Kolmogorov spaces}
\chapter{Functorial analysis}
%
In the analysis of KAM iteration schemes, one often deals with spaces of functions defined 
over smaller and smaller subsets. In order to keep track of this 
{\em functional analysis over shrinking sets}, we need a convenient framework.
We set up a particular formalism of {\em relative Banach spaces} that is inspired 
by the theory of vector bundles over a base $B$ and their associated sheaves of sections 
over subsets of $B$. It turns out that the estimates we need in KAM-theory aquire a
 meaning in terms of the geometry of the base $B$.
The main notion is that of a {\em Kolmogorov space}, which is a certain system  
of Banach spaces, parametrized by an ordered set. At a superficial level the 
formalism  bears some resemblance to the classical notion of a Banach scale, \index{Banach scale} 
but on closer look it has quite a different flavour. 
Many notions valid for Banach spaces, such as operator calculus or fixed point theory, 
extend, properly understood, to  the category of Kolmogorov spaces, 
whereas within a na\"ive formalism based on Banach scales or Fr\'echet spaces they do not. 
We try a systematic treatment and pay particular attention to the various types of morphisms 
that occur naturally and which appear to be a key feature of the theory. 
\section{Relative sets}
Let $B$ be an arbitrary set. 
If for each $b \in B$ a set $X_b$ is given, we speak of a
{\em set over $B$} or simply a $B$-set. One may form the 
disjoint union
\[ X:=\bigsqcup_{b \in B} X_b \]
of the set of sets. There is a canonical projection
$$p:X \to B ,$$
which send each element $x \in X_b$ to its base-point $b \in B$ and
we recover the set $X_b$ as the {\em fibre} $p^{-1}(b)$ over $b \in B$.
Relative sets form a category whose morphisms are {\em maps over $B$}, i.e.
maps $f: X \to Y$ that sit in a commutative diagram
$$\xymatrix{ X \ar[rr]^f \ar[rd]& & Y \ar[ld] \\
 & B }$$
We denote the set of such maps over $B$ by
\[ \Map_B(X,Y) .\]
 
If $p:X \lra B$ is a set over $B$ and $A \subset B$ a subset, then
{\em a section over $A$} is a map 
$$s:A \lra X$$ such that 
\[p \circ s=Id_A .\]
It consists of the choice of an element  $x_a \in X_a$ for each 
$a \in A$. We denote the set of all sections over $A$ by
\[ \Gamma(A,X) .\]

If $X$ and $Y$ are $B$-sets, one can form a new $B$-set $\mathcal{M}ap(X,Y)$,
whose fibre over $b$ is the set $Map(X_b,Y_b)$ of all maps $f_b:X_b \to Y_b$.
Note the relation
\[ \Map_B(X,Y)=\Gamma(B,\mathcal{M}ap(X,Y)) .\]

If $p:X \to B$ is a set over $B$ and $\phi:A \to B$ another map, we can form a set over $A$
\[ q: \phi^*(X) \to A\]  
called the  {\em pull-back}\index{pull back} over $\phi$. The
fibre over $a \in A$ is $X_{\phi(a)}$, so we have a {\em pull-back diagram}
$$\xymatrix{\phi^*(X) \ar[r] \ar[d]_q& X \ar[d]^p \\
 A \ar[r]^\phi& B }
$$ 
If $i: A \to B$ is the inclusion of a subset $A \subset B$, then one clearly 
has the equality
\[ \Gamma(A,X)=\Gamma(A,i^*X) .\]

\section{Relative vector spaces}
Let us now fix a field $K$, which in our applications always will be 
$\RM$ or $\CM$ and a vector space will always be a $K$-vector space. 
If $B$ is a set, and for each $b \in B$ we are given a vector space 
$E_b$, we can form the $B$-set $p:E \lra B$ where
\[ E:=\bigsqcup_{b \in B} E_b.\]
All fibres $E_b=p^{-1}(b)$ of this $B$-set now have the structure of 
a $K$-vector space, hence we call it a {\em vector space over $B$}. 
The elements of $E$ are called {\em vectors}, but of course, one can add such 
vectors only if they belong to the same vector space $E_b$. 
It is useful to have the geometrical picture of a fibre bundle in mind, but 
clearly we are dealing here with an almost trivial caricature of it.\\

\begin{center}
\begin{figure}[htb]
\includegraphics[width=11cm]{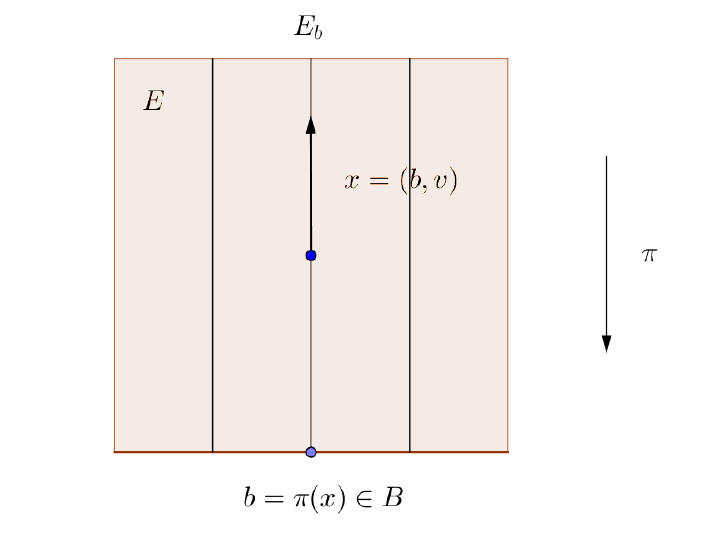}
\end{figure}
\end{center}
We will denote elements of $E$ often by {\em pairs} $(b, v)$, where
$b \in B$ and $v \in E_b$. This is in fact redundant, as a vector $v \in E_b$
'knows' its $b \in B$ over which it lives. But this notation adds clarity 
and avoids confusion.\\

The vector space operations may be used pointwise to define
the structure of a vector space on the set 
\[\Gamma(A,E)\]
of all sections of $E$ over $A$. Of course,
\[\Gamma(A,E) = \prod_{a \in A} E_a\]
is just the {\em direct product} of all vector spaces $E_a, a \in A$.
If we define the {\em support} of section $s:A \lra E$ as the 
set of $a \in A$ for which $s(a) \neq 0$, then one can identify the
direct sum
\[ \bigoplus_{a \in A} E_a \subset \prod_{a \in A} E_a\]
as the subspace of sections with finite support.\\

\begin{example} The projection on the first factor
$$p:\RM^2 \to \RM $$
defines a vector space over $B=\RM$.\footnote{Note by our terminology a slight
ambiguity arises : the statement {\em '$E$ is vector space over $\RM$'} means here that we are dealing with a relative Banach space over the base $\RM$ and not that $E$ is a vector space over the field $\RM$! In practice it will not cause any problems.}  The fibre $E_b$ over $b \in \RM$ is a copy of $\RM$. A section over a subset $A \subset B$ is a map of the form
$$A \to \RM^2, b\mapsto  (b,f(b))$$
where $f$ is an arbitrary function on $A$.
More generally any vector bundle in the sense of topology $E \to B$ is of 
course a vector space over $B$.
\end{example}
\begin{example} Consider the vector space $E$ over $B=\NM$ with fibre $E_n$ at $n \in \NM$ 
equal to the space $C^n([0,1],\RM)$ of $C^n$ functions on the unit interval. 
A section over $B$ consists of a sequence of functions $(f_n)$ with $f_n \in C^n([0,1],\RM)$.
\end{example}

Vector spaces over a basis $B$ form a category $\text{\bf \em Vect}_B$ in an obvious way. The morphisms 
are {\em linear maps over $B$}, so give rise to a 
commutative diagram 
$$\xymatrix{ E \ar[rr]^u \ar[rd]& & F \ar[ld] \\
 & B }$$
meaning that for each $b \in B$ the restriction of $u$ to $E_b$ is a linear 
map $E_b \to F_b$. We denote the set of morphism from $E$ to $F$ by
\[ \Hom_B(E,F).\]

The usual operations on vector spaces, like direct sum $\oplus$, $\otimes$, $Hom(-,-)$ extend to vector spaces over $B$. 
As the {\em internal Hom} is of importance, let us spell its definition. For two vector spaces $E,F$ over $B$ 
we let $\mathcal{H}om(E,F) \lra B$ be the vector space over $B$ with fibre at $b \in B$ equal to $\Hom(E_b,F_b)$. We have the relation
\[ \Hom_B(X,Y)=\Gamma(B,\mathcal{H}om(E,F)).\]

We note furthermore that the formation of section spaces is 'functorial'. 
If $f:E \to F$ is a  morphism of linear space over $B$ and
$A \subset B$, there is an induced linear map of section spaces:
\[  \Gamma(A,E) \to \Gamma(A,F), (s:A \to E) \mapsto (f \circ s: A \to F)\]

Like for vector bundles, we can pull-back a vector space over $B$ 
$$p:E \lra B$$
by a mapping of sets:
$$\phi:A \lra B$$
to get a vector space $\phi^*(E)$ over $A$, with $E_{\phi(a)}$ as fibre over $a \in A$.


\section{Relative Banach spaces}
Similarly, if the fibres of a map 
$$p:E \lra B$$ have the structure
of a Banach space, we speak of a {\em Banach space over $B$}.  
These form the objects of an obvious category $\text{\bf \em Ban}_B$. 
The set $\Hom_B(E,F)$ of morphisms in the category are given by 
{\em continuous linear maps} over $B$: any $u \in \Hom_B(E,F)$ sits in a 
commutative diagram 
$$\xymatrix{ E \ar[rr]^u \ar[rd]& & F \ar[ld] \\
 & B }$$
and for each $b \in B$ the restriction of $u$ to $E_b$ is a continuous 
linear map $u_b:E_b \lra F_b$. The set $\Hom_B(E,F)$ has a natural structure of 
a vector space, where the operations are defined pointwise.
We also introduce for two Banach spaces $E$ and $F$ over $B$ an
{\em internal Hom}, which defined to be the Banach space $\mathcal{H}om$ over $B$ whose fibre over $b$ is the Banach space $Hom(E_b,F_b)$ of continuous (=bounded) linear
maps $E_b \to F_b$, with the operator norm as norm. Again one has:
\[ \Hom_B(E,F)=\Gamma(B,{\Ht}om(E,F)) .\]  

When $p:E \lra B$ is a Banach space over $B$, each Banach space $E_b$ is 
equipped with a specific norm $|-|_b$ and the precise choice of the norm 
is an important part of the data.
Each vector  $x=(b,v) \in E$ has a norm
\[ |x|=|v|_b \in \RM_{\ge 0} \]
and so there is a map
\[ |-|: E \to \RM_{\ge 0},\;\;x \mapsto |x|.\]
Together with the projection $p: E \to B$ to the base we obtain the {\em norm map}\index{norm map}
$$\nu:E \to B \times \RM_{\ge 0},\ x=(b,v) \mapsto (b,|x|)=(b, |v|_b) . $$
It plays a role analogous to that of the norm for ordinary Banach spaces, 
which is the special case where $B$ is a point.
 
The {\em unit ball} \index{unit ball of a relative Banach space}  $B_E$ is the $B$-set whose fibre at $b$ is the unit 
ball of $E_b$. Note that this is just the preimage of $B \times [0,1]$ under the norm map $\nu$:
$$B_E:=\nu^{-1}( B \times [0,1]).$$ 
(The ball of radius $r$ will be denoted by $rB_E=\nu^{-1}(B \times [0,r])$.)
If $s \in \Gamma(A,E)$ is a section, then for each $a \in A$ we can
consider the norm $|s(a)|$, hence we obtain a {\em norm function}\index{norm function}
\[ |s|:=|-|\circ s: A \lra \RM_{\ge 0}, \;\;a \mapsto |s(a)|\]
of the section.\\
Using the norm of sections, there are different ways to define a topology on
the section spaces $\Gamma(A,E)$. One of the simplest ones is that of
{pointwise convergence in norm}, but as we only want to use Banach spaces, we 
in fact will not make much use of this topology.


\section{Bounded sections and bounded homomorphisms}
If $E \lra B$ is a Banach space over $B$, we can use the norm to define a 
subset of $\Gamma(A,E)$, consisting of those sections, for which the norm function is 
{\em bounded}.

\begin{definition}
\[ \Gamma^{b}(A,E):=\{ s \in \Gamma(A,E)\;\;|\;\;\sup_{a \in A} |s(a)| < \infty\} .\]
\end{definition}

For $s \in \Gamma^{b}(A,E)$ we put
\[ |s|_A:=\sup_{a \in A} |s(a)| .\]

\begin{proposition}
If $p:E \lra B$ is a Banach space over $B$ then $\Gamma^b(A,E)$ with $|-|_A$ is a Banach space.
\end{proposition}
\begin{proof} We write  $|-|$ for $|-|_A$. Clearly, one has:
\begin{align*}
|s_1+s_2| &\le |s_1|+|s_2|,\\
|\lambda\cdot s|&=|\lambda||s| \end{align*}
and so $\Gamma^b(A,E)$ is a normed vector space.
Now consider a Cauchy sequence of sections $s_n$ in $\Gamma^b(A,E)$, so for any $\epsilon >0$
we can find $N \in \NM$ such that
\[ |s_n-s_m| \le \epsilon \;\;\;\textup{for}\;\;n,m \ge N .\]
As for all $a \in A$ one has
\[ |s_n(a)-s_m(a)| \le |s_n-s_m|,\]
it follows immediately that for any $a \in A$ the sequences $s_n(a)$ are Cauchy sequences 
in $E_a$ and thus converge in $E_a$ to an element $s(a)$ which defines a section
$$s \in \G(A,E) .$$ 
As 
\[||s_n|-|s_m|| \le |s_n-s_m| \]
the sequence of norms $(| s_n |)$ is  Cauchy sequence in $\RM$ and is therefore bounded. As a consequence, 
the norm function $|s|$ is also bounded and thus the limit $s$ is again an element in $\Gamma^b(A,E)$.
 \end{proof}

As mentioned above, if $ p:E \to B$ is a Banach space over $B$, then the 
precise choice of the norms $|-|_b, b\in B$ is an important part of the 
data. If
\[ \lambda: B \to \RM_{>0}\]
is a strictly positive function, we can {\em rescale} the norm on $E_b$ by
 a factor $\lambda(b)$. Such a rescaling does not change the Banach spaces 
$E_b$, but defines a {\em different} Banach space over $B$.  We will need 
a special notation for this rescaled space.

\begin{definition} If $E \to B$ is a Banach space over $B$ and 
$\lambda: B \to \RM_{>0}$ a rescaling function, we define the {\em rescaled 
Banach space}\index{rescaled Banach space} over $B$ to be
\[ E(\lambda) \to B \]
with 
\[ E(\lambda)_b =E_b, \textup{with norm}\;\;\; |-|^{(\lambda)}:=\lambda(b) |-|_b .\]
\end{definition}

If the scaling function $\lambda$ is not bounded (from above and below)
the spaces of bounded sections of $E$ and $E(\lambda)$ will be different.
This phenomenon is a bit similar to the existence of different vector bundles 
over the same base and with isomorphic fibres.\\
\begin{example}
Consider again the vector space $E=\RM^2$ over $B=\RM$ defined by the projection on the first factor
$$p: \RM^2 \to \RM,\;\; (b,x) \mapsto b .$$
We consider the norm given by the absolute value of the second coordinate
$$\RM^2 \to \RM,\;\; (b,x) \mapsto |x|. $$
This defines a Banach space structure on $E$ over $B$.
A section
$$\RM \to \RM^2,\;\; b \mapsto (b,f(b))  $$
is bounded if and only if $f$ is bounded. Now rescale the norm by the function
$$\l(b)=b^k. $$
Then the section of $E(\l)$ is bounded if $b \mapsto b^k f(b)$ is bounded. For instance if $k>0$, the function $f$ must go to zero as $b$ goes to infinity.
\end{example}

For morphisms in $\text{\bf \em Ban}_B$ one can also define a notion of boundedness

\begin{definition} A morphism $u \in \Hom_B(E,F)$ \index{bounded morphism}
$$\xymatrix{ E \ar[rr]^u \ar[rd]& & F \ar[ld] \\
 & B }$$
 is {\em bounded}  if 
$$\sup_{x \in E}\frac{| u(x) |}{|x |} <+\infty. $$
\end{definition}

Clearly, this is equivalent to saying the that the operator norm $|u|_b:=|u_b|$,
$u_b: E_b \to F_b$, is bounded as a function on $B$. 

We denote the set of bounded morphisms by
\[ \Hom^b_B(E,F) \subset \Hom_B(E,F) .\]
Using Banach spaces over $B$ as objects and $\Hom^b_B(E,F)$ as morphisms, we 
obtain another category  $\text{\bf \em Ban}^b_B$.

As an example, consider the map
\[i: E \to E(\lambda),\;\;(b,x) \to (b,x) .\]
The map $i$ is bounded precisely when the rescaling function $\lambda$ is 
bounded from above. If $\lambda$  and $\lambda^{-1}$ are {\em both} bounded 
from above, the the map $i$  is an isomorphism in the category
 $\text{\bf \em Ban}^b_B$.

\begin{proposition} If $f: E \to F$ is an element of $\Hom^b_B(E,F)$ and $A \subset B$, then $f$
maps bounded sections to bounded sections, i.e. the induced inear map
\[\Gamma(A,E) \to \Gamma(A,F)\]
maps $\Gamma^{b}(A,E)$ to  $\Gamma^{b}(A,F)$.
\end{proposition}

As a result, Banach spaces over $B$ that are isomorphic as objects in $\text{\bf \em Ban}^b_B$
have isomorphic spaces of bounded sections. Note however, that the map 
\[j: E \to E(\lambda), \;\;(b,x) \to (b,x/\l(b))\]
has norm equal to $1$, so $E$ and $E(\lambda)$ {\em are} isomorphic in the category 
$\text{\bf \em Ban}^b_B$ via this map!

\section{Vector spaces and Banach spaces over categories}
In the above constructions the vector spaces $E_b$ for different $b$'s were
completely independent of each other and the above constructions are 
essentially empty.

In interesting situations the fibres $E_b$ are connected 
to each other and we have to take this into account. In the theory of vector 
bundles one can introduce a {\em connection} on $E \to B$ to compare fibres 
over  different points, but here we shall proceed in a different way. 
We will assume that the base $B$ is a {\em small category}. 
This means we are given a set $B$ and furthermore, for any two $a,b \in B$ 
there is a set $Mor(a,b)$ of {\em morphisms} from $a$ to $b$, for which the 
usual axioms are fulfilled.  So we can compose a morphism from $a$ to $b$
with one from $b$ to $c$ to obtain a morphism from $a$ to $c$ and this composition of morphisms is associative.

\begin{figure}[htb!]
\includegraphics[width=8cm]{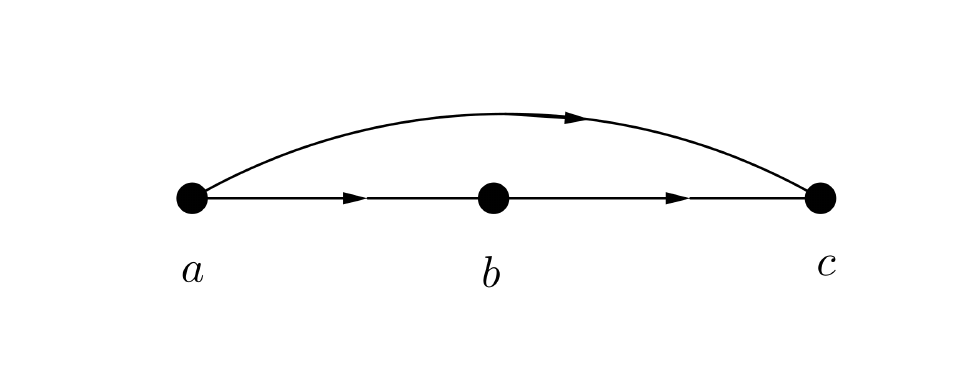}
\end{figure}
 
\begin{definition}A {\em vector space over a small category $B$} is a
functor 

\[e:B \lra \text{\bf  \em Vect}\] 
from the category $B$ to the category of $\text{\bf  \em Vect}$ vector spaces.\\
 
A {\em Banach space over a small category $B$} is a
functor 
\[ e:B \lra \text{\bf \em Ban}\]
from the category $B$ to the category $\text{\bf \em Ban}$ of Banach spaces.

\end{definition} 

So, as before, for each $b \in B$, we are given a vector space $e(b)=:E_b$, but
for each morphism $\alpha \in Mor(a,b)$ we are also given linear maps
\[ e(\alpha): E_a \lra E_b\]
with the condition
\[ e(\beta \circ \alpha) =e(\beta) \circ e(\alpha)\]
whenever it makes sense. These maps between fibres over different points
serve to 'connect' these different fibers.

\begin{figure}[htb!]
\includegraphics[width=8cm]{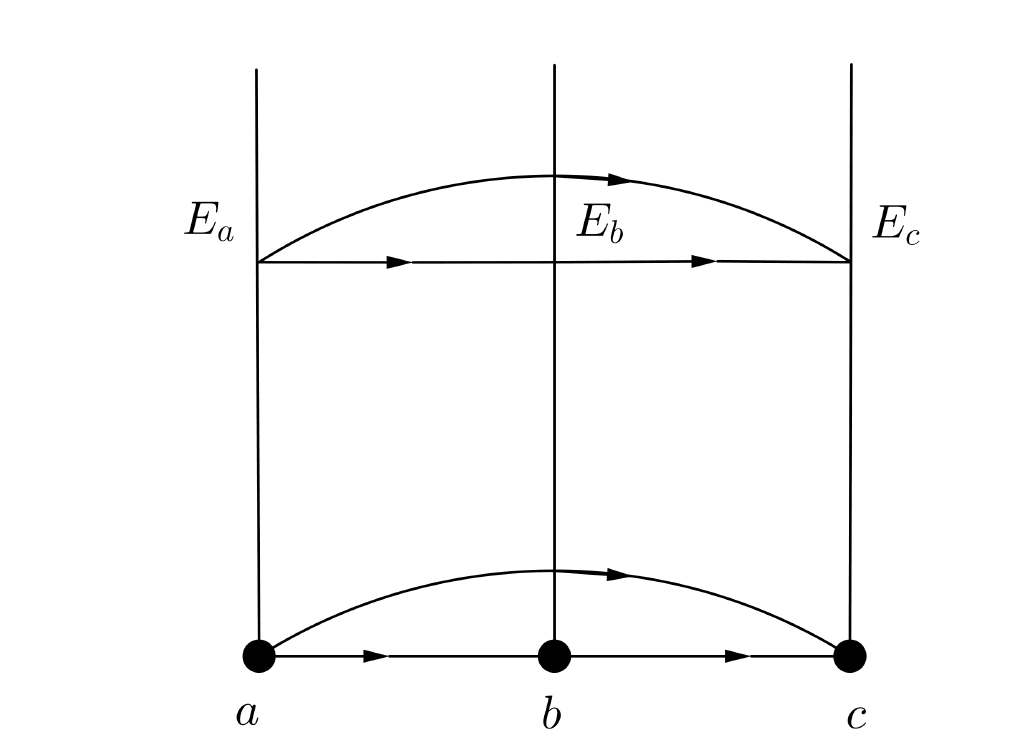}
\end{figure}
 
If we have just a set $B$, we can turn it into a trivial category by setting
$Mor(a,b)=\emptyset$ for $a \neq b$ and $Mor(b,b)=\{Id_b\}$, and then
we recover the notions of vector spaces and Banach spaces over the set $B$
as defined before.

\begin{example}\label{E::example}
Consider the small category $B=(\NM,\geq)$ where objects are integers and
the morphism space $Mor(i,j)$ consists of one element if $i \geq j$. The axioms for morphisms of a category encode the order relation in $\NM$.
We put
$$E_k=C^k([0,1],\RM)$$
and obtain a Banach space $E= C^\bullet([0,1],\RM)$ over the set $\NM$. As a $C^n$-function can be 
considered as a $C^k$-function  if $n \ge k$, we have natural inclusion maps $E_n \to E_k $. Thus we obtain a Banach-space over the ordered set $(\NM, \geq)$.
\end{example}

If $E$ and $F$ are Banach spaces over a small category $B$, a morphism
$u: E \to F$ is defined by continuous linear maps $u_b:E_b \to F_b$ that
are 'compatible' with the morphisms of the category $B$. This corresponds 
precisely to the notion of a {\em natural transformation} of the corresponding 
functors 
$B \to \textup{\bf \em Ban}$. Spelled out this means that for each 
$\alpha \in Mor(a,b)$ one has a  commutative diagram
$$\xymatrix{E_a \ar[r]^{u_a} \ar[d]_{e(\alpha)}& F_a \ar[d]^{f(\alpha)} \\
 E_b \ar[r]^{u_b}& F_b}
$$
We write again $\Hom_B(E,F)$ for this set, which has an obvious structure of
a linear space. It reduces to our earlier definition in case $B$ is
a set condidered as trivial small category. In this way we can speak of the 
category $\text{\bf \em Ban}_B$ of Banach spaces over the small category $B$.

In the language of categories and functors, the notion of pull-back is just the notion of {\em composition of functors}: if $C \to B$ is a functor between small 
categories and $e: B \to \text{\bf \em Ban }$ a Banach space over $B$, we obtain by
composition a functor $C \to \text{\bf \em Ban}$, i.e. a Banach space over $C$.

\section{Banach spaces over ordered sets}

\begin{definition}
By an {\em ordered set} we mean\footnote{We use the french convention. In most english texts, these are called {\em partially} ordered set,
sometimes abbreviated to {\em poset}.} a set $B$, together with a transitive, anti-symmetric and reflexive relation $\ge$.
\begin{align*}
 a \ge b,\;b \ge c &\Longrightarrow a \ge c,\\
 a \ge b, b \ge a &\Longrightarrow a=b,\\
 a \ge a .
\end{align*}
\end{definition}

The following example will be relevant later. We consider the set $\RM^n$, with the following order relation: 
for $x=(x_1,x_2,\ldots,x_n),y=(y_1,y_2,\ldots,y_n) \in \RM^n$ we define 
$$x \ge y$$ 
to mean
\[ x_i \ge y_i,\;\;\;\textup{for } i=1,2,\ldots,n . \]
One also can take an arbitrary subset $B \subset \RM^n$, with the order relation induced from the above one on $\RM^n$.

\begin{figure}[htb!]
\includegraphics[width=10cm]{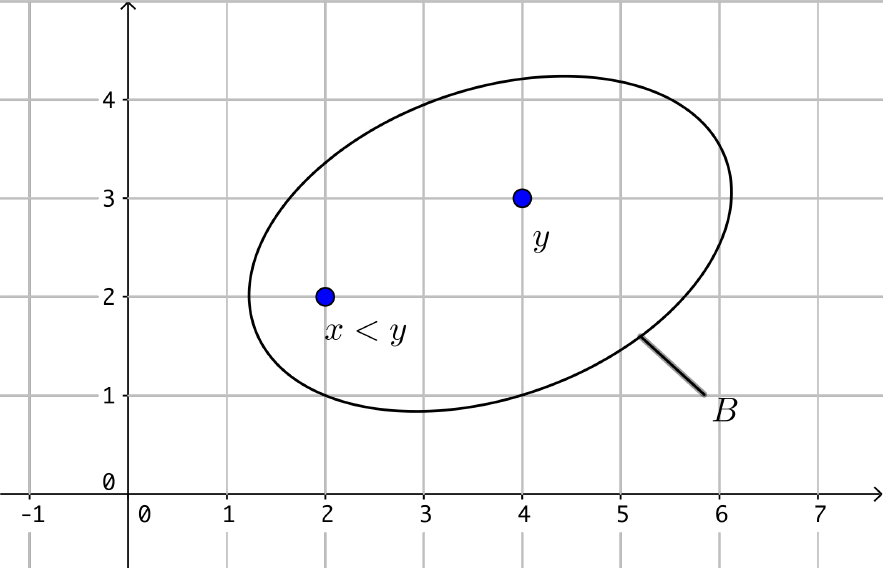}
\end{figure}

Any such an ordered set $(B,\ge)$ can be considered as
a small category with $B$ as set of objects and the set of 
morphisms from $a$ to $b$ consisting of a single element if $a \ge b$, and 
is empty otherwise.

A {\em vector space over an ordered set $B$}\index{vector space over an ordered set} is a covariant functor from $B$, considered as category,  to the category {\bf \em Vect} of  vector spaces. In concrete terms this means the following: 
we are given vector spaces
$E_b$ for each $b \in B$ and for each $a \ge b$ linear maps
\[e_{ba} :E_a \to E_b,\]
such that for all $c \le b \le a$ the maps $e_{ba}$, $e_{bc}$ and $e_{ca}$ satisfy
\[e_{cb}e_{ba}=e_{ca},\;\;\;\;e_{aa}=Id_{E_a} .\]
The mappings $e_{ba}$ will  be called the {\em restriction mappings}\index{restriction mapping} of the vector space over $B$.

\begin{example}
Consider the map $\RM^2\to \RM=B,\;(b,x) \mapsto b$. If we consider $B$ as an ordered set and $\phi:B \to \RM^*$ 
is a nowhere vanishing function, then the maps
\[e_{st}:=\phi(s) \phi(t)^{-1} : E_t \to E_s\]
trivially satisfy $e_{st}e_{tu}=e_{su}$. In this way we obtain a vector space over the ordered set $\RM$.  
\end{example}

\begin{example}
Let $X$ be a topological space with topology $\mathcal{T}$, consisting of the open subsets of $X$ ordered by inclusion. If $U \in \mathcal{T}$ is an open set, we let 
$$E_U:=C(U,\RM)$$
be the vector space of continuous functions on $U$. If $U \supseteq V$ there is a restriction mapping 
$$f_{VU}:C(U,\RM) \to C(V,\RM)$$ so we obtain a natural vector space of continuous functions over $\mathcal{T}$:
$$E \stackrel{\pi}{\to} \mathcal{T} $$
with fibre $\pi^{-1}(U)=E_U$.

More generally, a vector space over $\mathcal{T}$ is just the same as a presheaf (of vector spaces) on $X$.\\
\end{example}

Similarly,  a {\em Banach space over an ordered set $B$} is a covariant functor from $B$, considered as category, 
to the category $\text{\bf \em Ban}$ of Banach spaces:
$$e:B \to \text{\bf \em Ban}.$$
As the morphisms in the category $\text{\bf \em Ban}$ are continuous linear mappings, this implies that the restriction 
mappings are continuous.\\
 
There are several simple special notions related to ordered sets that we will need.\\

{\bf \em An ordered map}\index{ordered map} $f:A \to B$ between ordered sets is a map respecting 
the order relations on $A$ and $B$: 
if $a \ge a'$ holds in $A$, then $f(a) \ge f(a')$ holds in $B$.
In other words, we are dealing with monotonous maps and these are exactly 
the functors from $A$ to $B$, if we consider the ordered sets as small 
categories. Hence we can pull-back vector spaces and Banach spaces over 
ordered sets using ordered maps.\\

If $B$ is an ordered set, and $A \subset B$ a {\em subset} of $B$, then 
we can restrict the order relation of $B$ to $A$ and hence make $A$ into an
ordered set as well. The resulting inclusion map is an ordered map
and $A$ is, in the sense of category theory, a full sub-category.\\

{\bf \em Downsets and Upsets:} In an ordered set $B$ there are special subsets that are 'saturated' with respect to the order relation.
A {\em downset}\index{downset} $A \subset B$ is a set with the property
\[ a \in A,\;\; a \ge b \Longrightarrow b \in A\] 
and a {\em upset}\index{upset} $A \subset B$ is a set with the property that
\[ a \in A,\;\; b \ge a  \Longrightarrow b \in A\] 

If $A$ and $B$ are intervals in $\R$, the downsets of a point $(a,b)$ 
in $A \times B$ looks as follows:\\
\vskip0.3cm
\begin{figure}[htb!]
      \includegraphics[width=8cm]{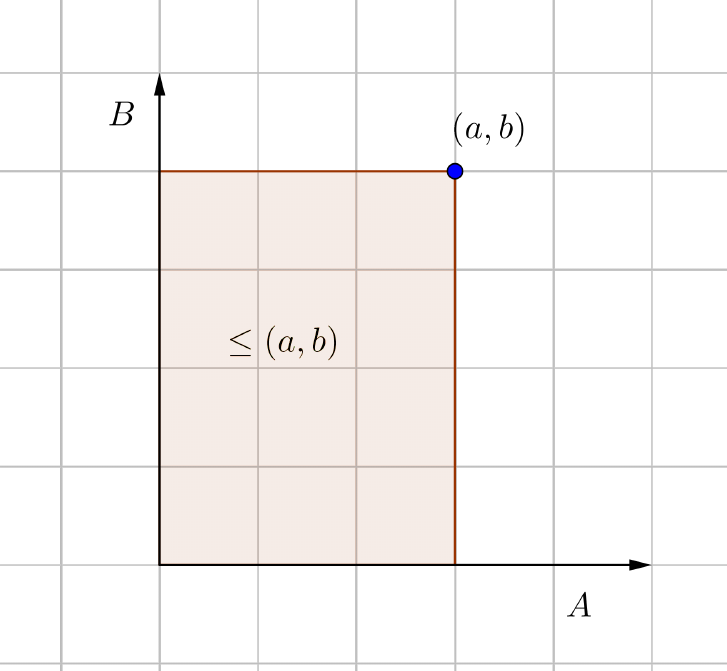}
\end{figure}
\vskip0.3cm
Arbitrary intersections and arbitrary unions of downsets are downsets. 
As a consequence, the intersection of all downsets containing any arbitrary set $A \subset B$ is the smallest downset containing $A$, the {\em downset hull} $A^d$ of $A$. Of particular importance are the {\em principal downsets}, determined by a
single element:
\[ B[\tau]:=\{ b \in B: b \le \tau\} =\{\tau\}^d .\]

Of course, mutatis mutandis, all the same for upsets. If $f:A \to B$ is 
an ordered map and $U \subset B$ a downset, then 
the preimage $f^{-1}(U)$ is a downset in $A$. So ordered maps are continuous
for the topology which has the downsets as open sets.\\

{\bf \em Directed sets and limit:} In many of our examples the ordered set $B$ is also directed. Recall that
a {\em directed set}\index{directed set} is a set $B$ 
together with a partial order relation $\ge$, with the additional property 
that for any two elements
$a, b \in B$ there exists $c \in B$ with $a \ge c$ and $b \ge c$. 
 
For a linear space $E \to B$ over a directed sets, we have a {\em direct limit functor}\index{direct limit} 
\[ E \mapsto \lim_{\longrightarrow} E\]
This vector space is functorially defined as the quotient of vector space
\[ \bigoplus_{b \in B} E_b\]
obtained by identifying vectors from $E_a$ and $E_b$ if 
they map to the same vector in $E_b$, if $a \ge c$, $b \ge c$. 

\begin{example}
Let $X$ be any topological space and consider as before the vector space $E$ over the
ordered set $\mathcal{T}$ of all open subset of $X$ whose fibre over $U$ is
$E_U=C(U,\RM)$. We can pull-back $E$ to the directed set $\mathcal{T}_p$ consisting 
of open sets in $X$ that contain $p$ by the inclusion
 $$\mathcal{T}_p \to \mathcal{T} $$
The direct limit can be identified with the vector space of 
{\em germs} of continuous functions at $p$.
\end{example}

\section{Horizontal sections}
If $ E\to B$ is a vector space over a small category $B$, then over each
morphism $\alpha \in Mor(a,b)$ in $B$ there is an associated linear
map
\[ e(\alpha): E_a \to E_b\]
between the fibres. As these maps allow us to relate fibres over 
different points, it can be seen as a set-theoretic type of connection.

\begin{definition}\index{horizontal section}
A section $s \in \Gamma(B,E)$ is called a {\em horizontal section} if it
is compatible with all these linear maps: 
\[e(\alpha) ( s(a) ) = s(b),\;\;\;\forall a, b \in B, \alpha \in Mor(a,b) .\]
We denote by 
\[ \Gamma^h(B,E) \subset \Gamma(B,E)\]
the linear subspace of horizontal sections.
\end{definition}
\begin{example}
Consider the $\RM$-vector space defined by the projection
$$\RM^2 \to \RM,\ (b,x) \mapsto b $$
with restriction mappings
$$(b,x) \mapsto (a,x). $$
Horizontal sections
$$ \RM \to \RM^2,\ b \mapsto (b,f(b))$$
 are defined  by constant functions: $f(b)=c$ for some $ c \in \RM $.
\end{example}
\begin{example} Consider again the vector space
 $C^\bullet([0,1],\RM),$  over 
the ordered set $(\NM, \geq)$. A section is a sequence of 
functions $(f_n)$, but horizontality now tells that $f_{n+1}$ should 
map to $f_n$ via the inclusion $$C^{n+1}([0,1],\RM) \to C^n([0,1],\RM) .$$
This means that choosing a global horizontal section is the same as the choice  of a
$C^\infty$-function $f$. The function 
$f_n$ is the same as $f$ but considered as an element of $ C^n([0,1],\RM)$.
\end{example}

In the case of vector spaces $E$ over an ordered set $B$, one   
has the following simple construction:  for any element $x=(b,v)$ we 
can construct a horizontal section over the downset $B[b]$ of $b$ 
by putting for $a \le b$
$$s_x(a)=e_{ab}(v), $$
Thus to an element $x \in E$, there is an associated horizontal section 
\[ s_x \in \Gamma(B[b],E) .\]

\begin{proposition}
\label{P::horizontal}
If $E$ is a Banach space over a base $B$, then the space  of horizontal bounded 
sections
\[ \G^h(B,E) \cap \Gamma^{b}(B,E) \]
is a Banach space.
\end{proposition}
\begin{proof}
Clearly, the equations
\[e(\alpha) ( s(a) ) = s(b),\;\;\;\forall a, b \in B, \alpha \in Mor(a,b)\]
defining horizontality determine a closed linear subspace of $\Gamma^{b}(B,E)$.
\end{proof}

\begin{definition} The Banach space of horizontal bounded sections is 
\[ \G^{\infty}(A,E):=\G^h(A,E) \cap \G^{b}(A,E) .\]
\end{definition}
\begin{example}
Consider again the vector space over the ordered set $B=\RM$ defined by the 
projection
$$\RM^2 \to \RM,\ (b,x) \mapsto b, $$
with restriction maps $e_{ab} =\phi(a)\phi(b)^{-1}: E_b \to E_a$ determined by a nowhere
vanishing function $\phi:B \to \RM^*$. We make it into a Banach space over $B$ by putting
$|1|_b=1$.
If $\phi(b)=1$ for all $b$, then the restriction maps are the identity and all horizontal sections 
are constant and therefore bounded, so
we have
$$\G^\infty(\RM,\RM^2) \approx \RM $$
For $\phi(b):=\exp(b)$ alll non-zero horizontal sections are unbounded, so
$$\G^\infty(\RM,\RM^2) = 0 $$
\end{example}

\begin{example}\label{E::example2} For the Banach space $E=C^\bullet([0,1],\RM)$ over $(\NM,\ge)$ of \ref{E::example}, the space $\G^b(\NM,E) $  is isomorphic to the space of $C^\infty$-functions  with uniformly bounded derivatives:
 $$\sup_{n \in \NM,x \in [0,1]} | f^{(n)}(x)|<+\infty. $$   
 So in this case there is a strict inclusion $\G^\infty(\NM,E) \subset \G^h(\NM,E)$. However over any finite set $A \subset \NM$,
the spaces $\G^\infty(A,E)$ and  $\G^h(A,E)$ are identical.
\end{example}

\section{The concept of Kolmogorov space}
\label{SS::definition}
We now arrive at our first important concept:

\begin{definition} A Kolmogorov space\footnote{In set-theoretical topology the term Kolmogorov space is used for topological
spaces for which the $T_0$-separation axiom holds. We will never use the notion in this sense, although the topology defined by downsets in an ordered set
provide natural examples such $T_0$ spaces.}\index{Kolmogorov space} is a Banach space
$$E \to B$$ 
over an ordered set $B$, such that the restriction morphisms
$$e_{st} : E_t \to E_s,\ s<t  $$
have norm at most one.
\end{definition}

An $N$-Kolmogorov space\index{N-Kolmogorov space}, abbreviated a $KN$-space, is a Banach space
$$E \to B$$ 
over an ordered subset $B \subset \RM^N$.

In practice, we will be mostly working with K1 and K2 spaces and in this
text K3-spaces will not play a role.\\
\begin{figure}[htb!]
\label{F::K2}
\includegraphics[height=5cm]{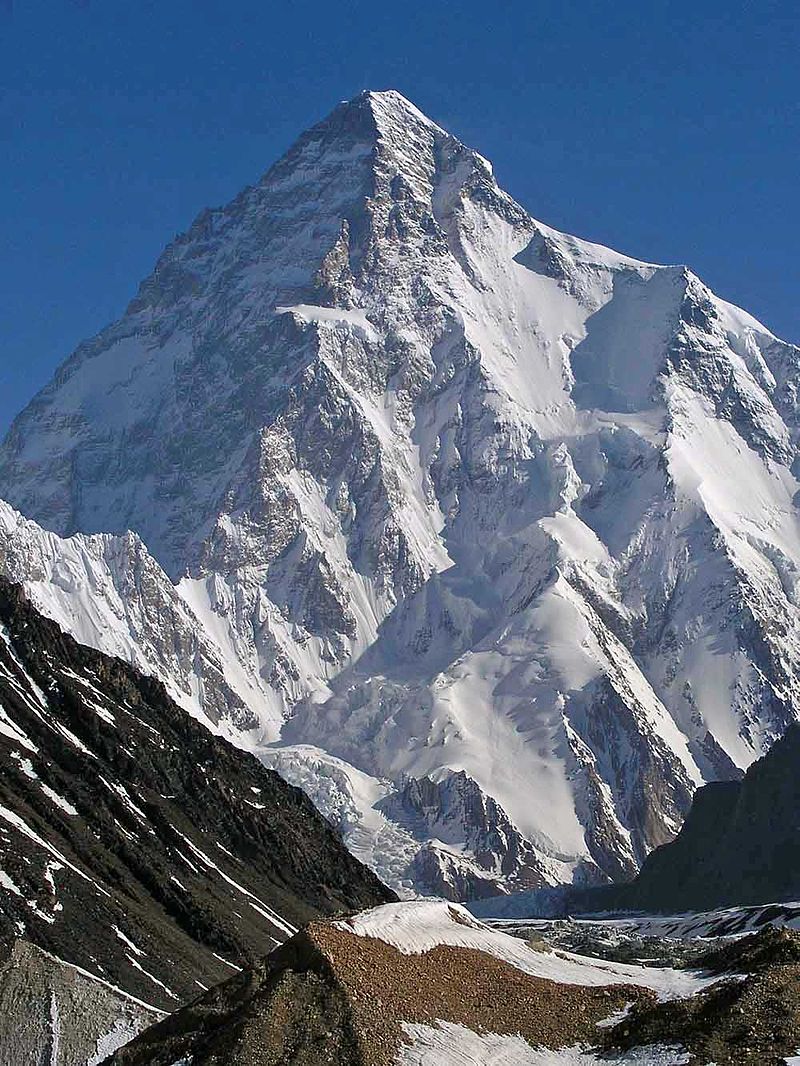}
\end{figure}
\vskip0.3cm
Kolmogorov spaces over the interval $]0,S] \subset \RM_{>0}$ are of particular importance and will be called {\em $S$-Kolmogorov spaces}\index{S-Kolmogorov space}.

So, in a Kolmogorov space $E$, the restriction maps  
 $$e_{st}:E_t \to E_s,\ t \ge s, $$
are continuous mappings of Banach spaces with operator norm $| e_{st} | \leq 1$, i.e. one 
has an estimate
\[ |e_{st}(x)|_s \le |x|_t . \]

$S$-Kolmogorov spaces have the following characteristic property:

\begin{proposition} \label{P::tauto} For an $S$-Komogorov space $E$ and $\tau \in ]0,S]$, there is a natural isometry of Banach spaces
\[ E_{\tau} = \Gamma^{\infty}(]0,\tau],E) .\]
\end{proposition}

\begin{proof}
A vector $v \in E_{\tau}$ determines a unique horizontal section $s: t \mapsto e_{t\tau}(v) \in E_t$.
As $|e_{t\tau}(v)| \le |v|_{\tau}$, the norm function $|s(t)|$ attains its maximum at $t=\tau$.  
\end{proof}

So in an $S$-Kolmogorov space, there is is no difference between vectors and
horizontal sections defined over sub-intervals $]0,t]$: each vector $(t,v)$
determines a unique horizontal section over $]0,t]$. This is no longer
true for $K2$-spaces. The downsets $B[b]$ determined a single $b \in B$ are 
very special subsets and usually we will have to consider horizontal sections 
over more general subsets.\\

A Kolmogorov space 
$$p:E \to B$$
carries a {\em canonical topology}\index{canonical topology}.
As mentioned before, the upsets define a topology on $B$ with open sets:
\[ U[b]:=\{ b' \in B\;|\;\;b' \ge b\}\]
with $b \in B$.
Now let $E_{U[b]}=\bigsqcup_{b' \geq b}E_{b'} \to U[b]$ be the pull-back via the inclusion map 
$$U[b] \hookrightarrow B.$$
The restriction maps of $E$ determine a canonical map
\[\pi_b: E_{U[b]} \to E_b,\;\;(b',v) \mapsto (b,e_{bb'}(v)).\] 
We obtain a topology on $E$ by declaring the sets
\[ \pi_b^{-1}(V),\;\; V \subset E_b\;\;\textup{open}\]
to form the basis. This is the {\em coarsest} topology for which the maps $\pi_b$ are continuous.

With this canonical topology the structure map
$p:E \to B$ and all horizontal sections $s \in \G^h(B,E)$, seen as
maps $s: B \to E$, are continuous. 

It is readily seen that a sequence $(b_n,x_n) \subset E$ converges to $(b,x)$ in the canonical topology
iff $(b_n)$ converges to $b$ for the topology of upsets and
$$|e_{bc_n}(x_n)-x| \xrightarrow[n \to +\infty ]{} 0   .$$
This is the way in which we use the topological structure.

If $E$ is a Kolmogorov space over $B$ and $\tau \in B$, then one can obtain a Kolmogorov space by {\em truncation}\index{truncation of Kolmogorov spaces}: if 
$$i_{\tau}:B[\tau] \to B,\ B[\tau]=\{b \in B\;|\;\; b \leq \tau \}$$ is the inclusion, then $E[\tau]$ is just the pull-back
of $E$ via this map: 
\[ E[\tau]:=i_{\tau}^* E .\] 
For many questions only the properties of the Kolmogorov space $E[\tau]$ for arbitrary small $\tau$ are important, so one is willing to replace $E$ by $E[\tau]$ if needed.\\

Kolmogorov spaces over a fixed ordered base $B$ form a 
full sub-category ${\text{\bf \em Kol}_B}$ of the category $\text{\bf \em Ban}_B$ of Banach spaces over $B$: a morphism $f: E \to F$ in ${\text{\bf \em Kol}_B}$ 
over $B$ that is just a  fibrewise continuous linear map.

If we want to compare Kolmogorov spaces over different bases, one has to
use pull-back by ordered maps: if $p:A \to B$ is an ordered map, we obtain
a functor
\[p^*: {\text{\bf \em Kol}_B} \to {\text{\bf \em Kol}_A}, 
(E \to B) \mapsto (p^*E \to A) .\]

In fact we are in the context of fibred categories and can consider
a $2$-category $\text{\bf \em Kol}$ fibres over the category of ordered sets,
with fibre over $B$ the category $\text{\bf \em Kol}_B$, etc.
Here we will not develop this point of view further.

\section{Kolmogorofication}

There are many situations where we obtain Banach spaces over an ordered set $B$ where the Kolmogorov space property
does not hold. For instance, if $E \to B$ is a Kolmogorov space over $B$ and 
\[ \lambda: B \to \RM_{>0}\]
a rescaling function, then the rescaled Banach space  $E(\l)$ over $B$ is 
only guaranteed to be a Kolmogorov space if the rescaling is monotonously
increasing:
\[ t \ge s  \Leftrightarrow \l(t) \ge \l(s) . \]

If $E$ is a Banach space over an ordered set $B$ and $U \subset B$ a subset, we can form the  Banach space
\[ \G^\infty(U,E)\]
of bounded horizontal sections over $U$. If $U \supset V$, then there is a natural restriction map
\[ \G^\infty(U, E) \to \G^\infty(V, E)\]
and as the norms on these spaces are defined as supremum over the corresponding sets,  these restriction maps trivially satisfy
\[ |s|_{U} \ge |s|_{V} .\]

Moreover for any $U$ and $V$ we have the following general sheaf-like gluing property:
\begin{proposition}
There is an exact sequence 
\[ 0 \lra \G^\infty(U \cup V),E \lra \G^\infty(U,E)\bigoplus\G^\infty(V,E) \lra \G^\infty(U \cap V,E)\]
\end{proposition}
\begin{proof}
If we have two bounded horizontal sections over two sets $U$ and $V$ that coincide on the intersection 
$U \cap V$, then they define a unique section over the union $U \cup V$, bounded by the maximum of the 
norms of the two sections.
\end{proof}

The subsets of $B$ are ordered by the inclusion and therefore $\G^\infty(-,E)$
defines a Kolmogorov space $\mathcal{E}$ over the power set $\Pt(B)$ of all subsets of $B$.

\begin{definition} If $E \to B$ is a Banach space over an ordered set $B$ 
the {\em secondary Kolmogorov space}\index{secondary Kolmogorov space} is the Kolmogorov space
$\mathcal{E} \to \mathcal{P}(B)$ with fibre over $U \subset B$ equal to  $\G^\infty(U,E)$.\\
\end{definition}

Recall that 
\[ B[b]:=\{b' \in N\;|\;b' \le b\}\]
is the principal downset determined by $b \in B$. Kolmogorov spaces have the
folowing simple key property:\\

\begin{proposition}\label{P::Kproperty} Let $E \to B$ be a Kolmogorov space and $b \in B$
then the restriction to the fibre over $b$ 
\[ \G^\infty(B[b],E) \to E_b, \;\;s \to s(b) \]
is an isometry.
\end{proposition}
\begin{proof}
If $v \in E_b$ is any vector, we can form the horizontal section over $B[b]$ by setting
\[ s(a):=e_{ab}(v)\]
As for the restriction maps one has $|e_{st}| \le 1$, it follows that the supremum of the
norm is attained at $a=b$, so
\[ |s|_{B[b]}=|s|_b=|s(b)| .\]
\end{proof}
 
This is the basis for the following  'miraculous' formal procedure that changes any Banach space over $B$ into a Kolmogorov space.

\begin{definition} Let $E \to B$ be a Banach space over an ordered set $B$. The 
{\em Kolmogorification} $EK \to B$ of a $E$ is the pull-back via 
\[\eta: B \to \Pt(B), b \mapsto B[b]:=\{b' \in N\;|\;b' \le b\}\]
of the secondary Kolmogorov space $\mathcal{E} \to \Pt(B)$:
\[ EK:=\eta^*(\mathcal{E}),\]
i.e. the fibre of EK at $b \in B$ is defined as
\[ EK_b:=\G^\infty(B[b],E) .\] 
\end{definition}

Proposition \ref{P::Kproperty} states that if $E$ itself is a Kolmogorov space, then $EK=E$. Let us now consider a few examples.\\
\begin{example}
Let us consider the vector space $E=\RM^2$ over the base $B=\RM_{>0}$:
$$ \RM^2 \to \RM_{>0}, (b,x) \mapsto b  $$ with the norm
$$|x|_b=b^{-k}|x|. $$
For $k>0$, this is not a Kolmogorov space. There are no non-zero bounded 
sections for this norm and the Kolmogorofication is the trivial vector space $EK=\{ 0 \} .$
\end{example}
\begin{example} \label{E::example}
A slightly more interesting example is given by the vector space $E=C^\bullet([0,1],\RM)$ over the ordered set $B=(\NM,\ge)$ with the norm
$$|f|_n=n^{-k}\sup_{x \in [0,1],j \leq n}|f^{(j)}(x)|. $$
For $k>0$, this is again not a Kolmogorov space. The Kolmogorofication has fibres
$$EK_n=C^n([0,1],\RM) $$
with norms
$$\sup_{x \in [0,1],j\leq n }|j^{-k}f^{(j)}(x)| . $$
This example shows that starting from a norm which behaves badly,
we get a systematic way to construct the adapted norms.
\end{example}

The category ${\text{\bf \em Kol}_B}$ of Kolmogorov spaces over $B$ 
is a full subcategory of the category ${\text{\bf \em Ban}_B}$ of Banach spaces
over $B$. The process of Kolmogorification provides us with a functor
\[ {\text{\bf \em Ban}_B} \to {\text{\bf \em Kol}_B},\]
restricting to the identity on the ${\text{\bf \em Kol}_B}$. Note the similarity to the process of sheafification of a presheaf.

For this reason we will sometimes use an alternative terminology:

\begin{definition} {\em A pre-Kolmogorov space over $B$}\index{pre-Kolmogorov space} is a 
{\em Banach space over an ordered set $B$}.
\end{definition}

Another important construction is that of the {\em direct image}\index{direct image} in the category of Kolmogorov spaces. If $p:A \to B$ is an ordered map, and $U \subset B$ is a downset, 
then the preimage $p^{-1}(U) \subset A$ is also a downset. 
If $E$ is a Banach space over $A$, then the direct image $p_*(E)$ is
a Kolmogorov space over $B$, defined by setting  
\[  p_*(E)_b:=\Gamma^\infty(p^{-1}(B[b]),E) .\]





\section{The order filtration\index{order filtration} }

If $\s$ is a horizontal section of an $S$-Kolmogorov space $E$, then the associated
norm-function $t \mapsto |\s(t)|$ is a non-decreasing function. It makes sense to
look at the speed of vanishing of $|\s(t)|$ as $t$ drops to $0$. 
We can do this as follows: by rescaling $E$ by the function
$$\l_k:]0,S] \to \RM, s \mapsto s^{-k},\ k \in \RM. $$
we get a Banach space $E(\l_k)$ over $]0,S]$. 

For $k < 0$ the function $\l_k$ is increasing and therefore, the spaces $E(\l_k)$ are naturally Kolmogorov spaces.
\begin{definition}\label{D::order k-part}\index{order $k$-part}
The Kolmogorofication  of $E(\l_k)$ is called the {\em order $k$-part of $E$}\index{order $k$-part $E^(k)$}. 
We denote it by $E^{(k)}$.
\end{definition}
Let us spell out what this space is. We start with a  vector $v_t \in E_t$ and construct the associated horizontal
section $v$  which consists of the vectors 
$$v_s:=e_{st}(v_t) \in E_s.$$ 

The section $v$ is bounded if and only if
$$\sup_{0 <s\le t} s^{-k}|v|_s<+\infty$$
So Kolmogorofication of the Banach space $E(\l_k)$ over $]0,S]$ leads to the norm
\[ |v|_t^{(k)}:=\sup_{0 <s\le t} s^{-k}|v|_s\]
As expected, the restriction mappings 
have norm $\le 1$, since the $\sup$ over a smaller set becomes smaller.

In this way we obtain a decreasing filtration 
\[  E^{(k)} \supset E^{(l)}\;\;  \textup{for}\;\; k \le l \]
by Kolmogorov spaces, although one should always keep in mind that for
each $E^{(k)}$ we have to use a different norm. We will call this
filtration the {\em order filtration} of $E$.
  
 \begin{proposition}
\label{P::inclusion}
Let $E$ be an $S$-Kolmogorov space.
Let $B_k \subset E^{(k)}$ be the unit ball. For any $\e>0$, and any $R>0$ we have
$$RB_{k+\e}[\tau] \subset  B_k$$
for $\tau=R^{-1/\e}$.
\end{proposition}
\begin{proof}
 Take a bounded horizontal section
 $$\s \in \G^\infty(]0,t],RB_{k+\e}[\tau])$$
 We have
 $$|\s(s)| \leq |\s|^{(k+\e)} s^{k+\e} \leq R s^{k+\e}$$ 
 for $s \leq t$.
 Thus
 $$|\s|^{(k)} \leq Rt^{\e}<1.$$

\end{proof}

\section{External Hom}
If $E$ and $F$ are pre-Kolmogorov spaces over an ordered set $B$, then
the vector space of homomorphisms $u:E \to F$ in the category $\textup{\bf \em Ban}_B$
was denoted by
\[\Hom_B(E,F)\]
and consisted of continuous linear maps $u_b:E_b \to F_b$, with the additional property that
for all $b \le a $ one has a commutative diagram
$$\xymatrix{E_a \ar[r]^{u_a} \ar[d]_{e_{ba}}& F_a \ar[d]^{f_{ba}} \\
 E_b \ar[r]^{u_b}& F_b}
$$
We can also form the Banach-space $\mathcal{H}om_B(E,F)$ over the {\em set} $B$ with fibre
over $b \in B$ the vector space $Hom(E_b,F_b)$ of continuous linear operators, and clearly one has 
\[ Hom_B(E,F) \subset \G(B, \mathcal{H}om_B(E,F)).\]
But note that  $\mathcal{H}om_B(E,F)$ is not a Banach space over the {\em ordered set $B$}, as {\em there are no natural restriction maps} 
\[ Hom(E_a,F_a) \to Hom(E_b,F_b)\]
for $b \le a$. Consequently, one can not speak about horizontal sections of  
$\mathcal{H}om_B(E,F)$, which we would like to give back the space $\Hom_B(E,F)$. This fundamental problem can be circumvented by an important construction 
that we will describe now.\\

Consider more generally Banach spaces $E \to A$ and $F \to B$ over sets 
$A$ and $B$.  We can form for each $a \in A$ and $b \in B$ the Banach space
\[ Hom(E_a,F_b)\]
of bounded linear transformations with the operator norm. We denote by
\[ \mathcal{H}om(E,F) \to A \times B\]
the resulting Banach space over $A \times B$. 
The fibre above $(a,b) \in A \times B$ is the vector space $Hom(E_a,F_b)$.
If $E$ and $F$ are Banach spaces over {\em ordered} sets  $A$ and $B$, then 
for $a' \geq a$, $b \geq b'$ we have a restriction mappings in 
$E_{a'} \to E_a$ and $F_{b} \to F_{b'}$ making a diagram
$$\xymatrix{ E_{a'} \ar[r] & E_a \ar[ld] \\
 F_{b} \ar[r] & F_{b'}}$$
and thus we obtain a natural map
$$Hom(E_{a},F_{b}) \to Hom(E_{a'},F_{b'}) ,$$
hence a Banach space over the ordered set $A^{op} \times B$:
$$\mathcal{\Ht}om(E,F) \to A^{op} \times B. $$

We used the notation let $A^\text{op}$ to denote the same set,
but with {\em opposite} order relation. So in $A^{op}\times B$ we have:
\[ (a,b) \geq (a',b') \textup{ if and only if }\;\;a'\leq a\;\; \textup{in $A$ and }  b \geq b'\;\;\textup{in $B$} \]

In $A^{op}\times B$ the downset of a point now looks as follows:\\
\vskip0.3cm
\begin{figure}[htb!]
      \includegraphics[width=10cm]{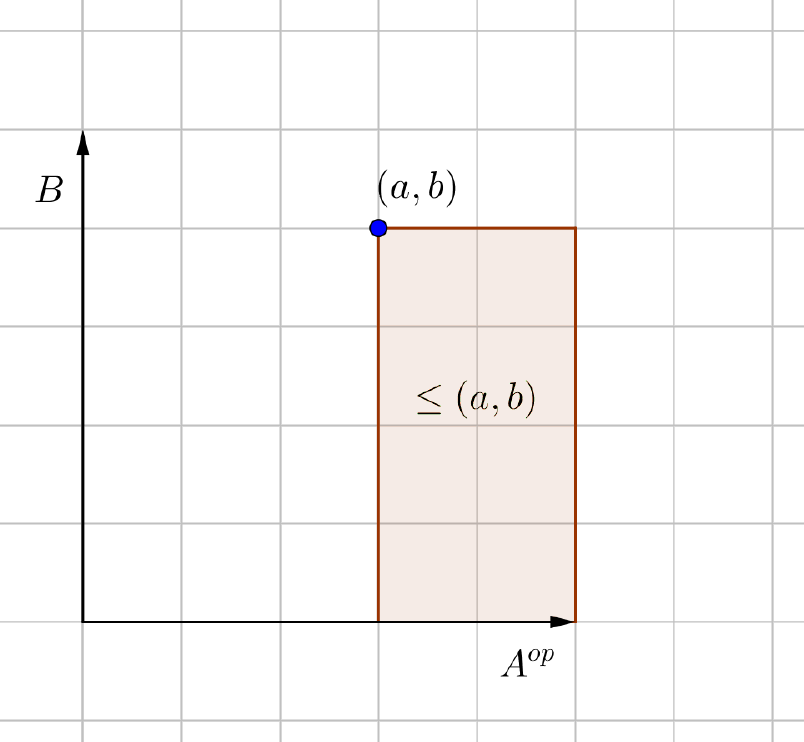}
\end{figure}
\vskip0.3cm

In case $E \to A$ and $F \to B$ are Kolmogorov spaces, we see from the above 
composition diagram that the restriction mappings in
$\mathcal{H}om(E,F)$ also have norm $\le 1$, so $\mathcal{H}om(E,F)$ is a 
Kolmogorov space over $A^\text{op} \times B$.

\begin{definition} The {\em external Hom-space}\index{external Hom-space}
\label{D::external Hom-space} of 
Kolmogorov spaces $E \to A$ and $F \to B$ is the above described 
Kolmogorov space
\[\mathcal{H}om(E,F) \to A^{op} \times B\]
with the Banach space 
\[ Hom(E_a, F_b) \]
as fibre over $(a,b) \in A^{op} \times B$.
\end{definition}

The fact that we have to reverse the order relation in $A$ is of course
due to the contravariant nature of the $Hom$-functor in the first variable. 
If $E$ and $F$ are Kolmogorov spaces over $]0,S]$, then the external
Hom-space $\mathcal{H}om(E,F)$ is a Kolmogorov space over $]0,S]^{op} \times ]0,S]$. 
This is a natural source of $K2$-spaces.\\

The external Hom can be used to solve the above mentioned problem concerning
the categorical $\Hom_B(E,F)$.

\begin{definition} For a partial ordered set $B$, we denote by
\[ \overline{\D}:=\{ (s,t) \in B^{op} \times B\;|\;\;s \le t\}\]
the closed subdiagonal\index{closed subdiagonal}.
\end{definition}

\begin{proposition}
\label{P::complete_morphisms}
Restriction to the diagonal defines an isomorphism
\[  \G^h(\overline{\D},\mathcal{H}om(E,F)) \to \Hom_{B}(E,F)\] 
of vector spaces and an isometry of Banach spaces
\[  \G^\infty(\overline{\D},\mathcal{H}om(E,F)) \to \Hom_{B}^b(E,F)\] 
\end{proposition}
\begin{proof}
We have to construct an inverse for the restriction map.
If $u:E \to F$ is a Kolmogorov space morphism over $B$, then for each $t \in B$ 
we have a map $u_t:E_t \to F_t$, so we obtain a section of $\mathcal{H}om(E,F)$
over the diagonal $\{(t,t)\}$. This section extends naturally to a horizontal
section over $\overline{\D}$. Indeed, the compatibility with the restriction 
maps in $E$ und $F$ tells us
\[ f_{st} u_t = u_s e_{st},\;\;\; \textup{for all}\;\;\;s \le t .\]
This means that by setting
\[ u_{st}:=f_{st} u_t = u_s e_{st}:E_t \to F_s\]
we obtain an horizontal extension of this section over $\overline{\D}$.
If we start with an element of $\Hom^b_B(E,F)$, it represents  
a bounded section of $\mathcal{H}om(E,F)$ over the diagonal $\{(t,t)\}$,
which extends in exactly the same way to a bounded horizontal 
section of $\mathcal{H}om$ over $\overline{\Delta}$. The fact that we
obtain an isometry follows from the Kolmogorov property: for horzontal
sections over $\overline{\D}$ the norm function takes its supremum on 
the diagonal.
\end{proof}

We should always make a careful distinction between the following
objects (here $A,B$ are ordered sets):
\begin{enumerate}
\item $Hom(E,F)$ is the ordinary Banach space of continuous linear maps between Banach spaces $E$ and  $F$ provided with the operator norm.\\

\item $\Hom_B(E,F)$ is the categorical Hom between two Banach spaces $E,F \to B$ over an ordered set $B$.  It is a vector space.\\ 

\item  $\Hom_B^b(E,F)$ consists of those elements of  $\Hom_B(E,F)$ whose operator norm function is bounded on $B$. It is a Banach space.\\

\item $\mathcal{H}om_B(E,F)$ is a Banach space over the {\em set} $B$, with fibre $Hom(E_b,F_b)$; we will not make much use of this object.\\ 

\item  $\mathcal{H}om(E,F)$ is the most important  and more general object. It is a Banach space over $A^{op} \times B$. In the case that $A=B$ the (resp. bounded) horizontal sections of  $\mathcal{H}om(E,F)$ over the closed subdiagonal gives back  $\Hom_B(E,F)$ (resp. $\Hom_B^b(E,F))$.\\

\end{enumerate}


\chapter{Kolmogorov spaces and sheaves}
In the previous chapter we introduced the formal framework of Kolmogorov spaces
that will be used to handle the functional analytic parts of KAM-theory.
We have seen that it is easy to construct examples and in this and later 
chapters we will be mainly interested in specific examples related to spaces 
of {holomorphic objects}. Due to the identity theorem for holomorphic 
functions, in these examples the restriction mappings 
are {injective} on domains and over a directed set the Kolmogorov space 
structure leads to a {descending filtration} on the limit vector space 
${\displaystyle \lim_{\to} E}$ by Banach spaces $E_s$. However, it is always 
important to keep also track of the norms involved.   

\section{The sheaf of holomorphic functions}
For an open set  $U$  of $ \CM^n$ or more generally of a fixed analytic space $X$, we denote by $ \Ot_X(U)$ or simply $\Ot(U)$ the set of all holomorphic functions on $U$. As obviously the pointwise sum and product of holomorphic functions is again holomorphic,
the set $\Ot(U)$ has the natural structure of a $\CM$-algebra.
For an inclusion $V \subset U$ there are natural {\em restriction maps}\index{restriction map}
$\Ot(U) \to \Ot(V), f\mapsto f_{|V}$, making the assignment $U \mapsto \Ot(U)$ into a presheaf on $\CM^n$ or equivalently we have a vector space $\Ot$ over 
the ordered set $\Tt$ of open subsets. In fact, this presheaf is a sheaf $\Ot$ on $\CM^n$: if any $U$ is covered by open sets $U_i, i \in I$ and $f^{(i)} \in \Ot(U_i)$ have the property that
\[f^{(i)}_{|U_i \cap U_j} =f^{(j)}_{|U_i \cap U_j}\]
then there exist a unique $f \in \Ot(U)$ with 
\[f_{|U_i}=f^{(i)} .\]

%
%
The space of sections $\Ot(K)$ of $\Ot$ over a compact set $K$ can be identified with the space of holomorphic functions that are defined on an open subset 
containing $K$, where two such functions are identified when they coincide on a 
common open subset containing $K$. That is, we are dealing with {\em germs} 
\index{germ of a function} of holomorphic functions along the set $K$. 
In the special case where $K$ consists of a single point, we have the stalk 
$\Ot_p$ of $\Ot$ at the point in question, which can, after choice of coordinates $x_1,x_2,\ldots,x_n$, be identified with the algebra of convergent 
power series
\[ \Ot_0 =\CM\{x_1,x_2,\ldots,x_n\} .\]

For an open subset $U \subset X$, the vector spaces $\Ot(U)$ has 
a natural structure of a {\em Fr\'echet-space} defined by the family of 
semi-norms:
\[ |f|_K =\sup_{ z \in K} |f(z)|,\;\; K \subset U\;\;\textup{compact}\]
 But we will not really use this topology, as in the formalism
of Kolmogorov spaces we developed, we work only with Banach spaces.

\section{Kolmogorov spaces $\Ot^b$, $\Ot^c$, $\Ot^h$}
So the question arises which Banach spaces of holomorphic functions one can 
naturally associate to (certain) open subsets $U$ of an analytic space $X$.
We will briefly discuss three different natural choices.\\
First, one can consider the vector space 
\[ \Ot^b(U):=\{f \in \Ot(U)\;|\;f\;\;\textup{is bounded}\},\]   
consisting of the {\em bounded holomorphic functions} on $U$.
We can make $\Ot^b(U)$ into a normed vector space using the supremum norm: 
\[ |f|_U:=\sup_{p \in {U}} |f(p)| .\]
and in fact with this norm $\Ot^b(U)$ is a Banach space: a Cauchy sequence
$(f_n)$ for $|-|_U$ converges pointwise to a bounded continuous function and
by the Cauchy integral theorem that function is also holomorphic.

There are obvious restriction mappings, making the $\Ot^b(U)$ into a Banach 
space
$\Ot^b$ over $\mathcal{T}$. If $U \supseteq V$, the restriction mapping
\[ f_{VU}: \Ot^b(U) \to  \Ot^b(V)\]
has operator norm at most $1$, as the supremum $|f|_V$ over the smaller set
$V$ clearly is $\le |f|_U$. Hence 
$$\Ot^b \to \Tt$$
is in fact a Kolmogorov space.
Note however that $U \mapsto \Ot^b(U)$ is no longer a sheaf. For two
open subsets $U,V$ the kernel of the obvious map
\[ \Ot^b(U) \oplus \Ot^b(V) \to \Ot^b(U \cap V),\;\;\;(f,g) \mapsto f_{| U \cap V}-g_{| U \cap V}\]
is precisely $\Ot^b(U \cup V)$, so bounded holomorphic functions can be
glued, but the sheaf axiom for infinite unions does not hold in general.\\
If we restrict to the sub-directed system $\mathcal{T}_p$ of open sets 
containing a point $p$,  we obtain as limit the vector space of germs of 
holomorphic functions at $p$,  which can in turn be identified with the 
space of convergent power series.\\

Second, for any open  $U \subset X$ we can consider the vector space 
\[ \Ot^c(U):=C(\overline{U}) \cap \mathcal{O}(U) .\]   
consisting of holomorphic functions on $U$ that have a {\em continuous} 
extension to the closure $\overline{U}$.
It is again a vector space over $\mathcal{T}$:
$$\Ot^c \to \mathcal{T} $$
and in fact $\Ot^c$ is a sheaf on $X$. But $\Ot^c(U)$ is in general
not a Banach space, but if $U$ is relatively compact (i.e. if $\overline{U}$
is compact), then $\Ot^c(U)$ equipped with  the $\sup$-norm {\em is} a Banach 
space.  So if we restrict $\Ot^c$ to the system $\Tt_c$ of relatively compact 
open sets we obtain a Kolomogorov space. We can also glue $\Ot^c$-functions: 
the kernel of the map
\[ \Ot^c(U) \oplus \Ot^c(V) \to \Ot^c(U \cap V),\]
is precisely $\Ot^c(U \cup V)$. As the arbitrary union of relative compact 
open sets  need not be relative compact, we again do not quite obtain a sheaf 
in the ordinary sense of the word.\\

Clearly, for $U \in \mathcal{T}_c$ there is a natural inclusion
\[ \Ot^c(U) \to \Ot^b(U)\]
compatible with restrictions. So these combine to a natural morphism $I$ of 
Kolmogorov spaces over $\mathcal{T}_c$: 
 $$\xymatrix{\Ot^c \ar[rd]  \ar@{^{(}->}[rr]^I & & \Ot^b \ar[ld] \\
  & \mathcal{T}_c &}
 $$
In other words: $I \in \Hom_{\mathcal{T}_c}(\Ot^c,\Ot^b)$.\\

Third, if on the analytic manifold $X$ a measure $\mu$ is given, e.g. $X=\CM^n$ 
with the Lebesgue measure, one also can consider the Banach space\index{structure sheaves $\Ot^{(p)}$, $\Ot^h$} 
$$\Ot_X^{(p)} \to \mathcal{T} $$
with fibre
\[ \Ot^{(p)}(U):=L^p(U,\mu) \cap \Ot_X(U)\]
of $p$-integrable holomorphic functions on $U$, with norm
\[  |f|_U =(\int_U |f(z)|^p d\mu)^{1/p} .\]
If $V \subset U$, then there are natural restriction mappings 
\[  \Ot^{(p)}(U) \to \Ot^{(p)}(V)\]
of norm $\le 1$, so we obtain Kolmogorov spaces. 

One can consider $\Ot^{(p)}$ as an $L^p$-variant of the structure sheaf.
Again, they are not sheaves in the usual sense, but again the kernel of 
the map
\[ \Ot^{(p)}(U)\oplus \Ot^{(p)}(V) \to  \Ot^{(p)}(U \cap V)  \]
is  $\Ot^{(p)}(U \cup V)$. 

In the special case $p=2$ we obtain a Hilbert space, for which we will
choose the particular notation
\[  \Ot^h(U):=\Ot^{(2)}(U) .\]
This norm comes from a Hermitean scalar product on $\Ot^h(U)$
given by
\[ \langle f, g \rangle_U:=\int _{U} f(z) \overline{g(z)} d\mu .\]

Such a space might, for obvious reasons,  be called a {\em Kolmogorov-Hilbert}
space\index{Kolmogorov-Hilbert space}.

\section{Restriction to polydiscs}

In practice it is convenient to work with systems of vector spaces over 
smaller ordered sets, like intervals in $\RM$, with their natural 
ordering. We will often use the function spaces introduced above
over systems of (open or closed) {\em polydiscs}.\index{polydisc}
For a {\em polyradius}\index{polyradius}
\[ \rho:=(\rho_1,\rho_2,\ldots,\rho_n) \in \R_{\ge 0}^n\]
we write 
\[D_{\rho}:=D_{\rho_1} \times D_{\rho_2} \times \ldots \times D_{\rho_n}:=\{(x_1,x_2,\ldots,x_n)\;|\;\;|x_i| < \rho_i \}\]
for the {\em open polydisc} of polyradius $\rho$. The open
polydisc is an open set over the orderered set $\RM_{\ge 0}^n$ of radii.

We will consider frequently the family $D=(D_s)$ of polydiscs
$$D_s:=\{ z \in \CM^n:|z_i|< s \}=D_{s,s,\ldots,s} .$$
We will call it the {\em unit polydisc}\index{unit polydisc}; it is a set 
over $\RM_{>0}$: $$D \to \RM_{>0} .$$
Using the map of ordered sets
\[ i: \RM_{>0} \to \mathcal{T}_c,\;\;s \mapsto D_s ,\]
we can pullback the Kolmogorov spaces $\Ot^b, \Ot^c$ and $\Ot^h$ to
obtain Kolmogorov spaces 
\[ \Ot^b(D):=i^*\Ot^b,\;\;\Ot^c(D):=i^*\Ot^c,\;\;\Ot^h(D) :=i^*\Ot^h\]
over $\RM_{>0}$, with fibres 
\[ \Ot^b(D)_t:=\Ot^b(D_t) .\]
\[ \Ot^c(D)_t:=\Ot^c(D_t) = C(\overline{D_t}) \cap \mathcal{O}(D_t) .\]
\[ \Ot^h(D)_t:=\Ot^h(D_t) = L^2(\overline{D_t}) \cap \mathcal{O}(D_t) .\]


\section{The order filtration on $\Ot^c$}
Let us take a closer look to the Kolmogorov space $E:=\Ot^c(D)$ over
$\RM_{>0}$ of holomorphic functions on the unit polydics, with continuous 
extension to the boundary.   The order $k$-part $E^{(k)}$ was defined in 
\ref{D::order k-part} as the Kolmogorification of $E(\lambda_k)$, the Banach 
space over $\RM$ obtained  by rescaling the norm with the function
\[\lambda_k: s \mapsto s^{-k} .\]
It consists of those $f \in E_t=\Ot^c(D_t)$ for which 
$$| f |_t^{(k)}:=\sup_{s \leq t,|z| \leq s} s^{-k}|f(z) |_s. $$
is finite. Clearly, if $k\le 0$, then $E^{(k)}=E$ and furthermore for
$k \ge 0$ and $f \in E^{(k)}_t$, we have the estimate  
\[ |f|_s \le R s^k,\;\;\textup{for all}\;s \le t,\;\;\;\textup{with}\;\;R =|f|_t.\]
That is, for any $z \in D_s$, with $s:=|z|$:
\[ |f(z)| \le R |z|^k.\]
So we see that for $k \in \NM$ the space $E^{(k)}_t$ consist of 
those functions whose derivatives at the origin vanish up to order 
$k-1$. When we write
$$f(z)=z^kf_k(z) $$
then by the maximum principle
$$s^{-k}|f(z) |_s=|f_k|_s $$
and therefore
$$| f |_t^{(k)}=|f_k|_t. $$

The direct limit functor sends  $\Ot^c(D)$ to the space of germs of holomorphic functions at $0$, which can be identified with convergent analytic series  $\CM\{ z \}$. 
The direct limit of $E^{(k)}$  is equal to to the $k$-th power $\Mt^k$ of the maximal ideal  
$$\Mt:=\{ f \in \CM\{ z \}: f(0)=0 \}. $$ 
So the filtration by Kolmogorov-spaces
$$ E=E^{(0)} \supset E^{(1)} \supset E^{(2)} \supset E^{(3)} \supset \ldots$$
corresponds {\em exactly} to the $\Mt$-adic filtration via the direct limit functor:
$$ \CM\{ z \} \supset \Mt \supset \Mt^2 \supset \Mt^3 \supset  \ldots$$
This generalises, of course, to an arbitrary number of variables.

The Kolmogorov space norms $|-|^{(k)}$ are 'adapted' to the filtration: for instance 
according to Proposition~\ref{P::inclusion},  the unit ball $B \subset \Ot^c(D)$ contains any ball of $(\Ot^c(D))^{(1)}[\tau]$ for $\tau$ small enough.  

%

\section{The Kolmogorov spaces  $\ell^2(\lambda)$}
We will see that the  Kolmogorov-Hilbert space  $\Ot^h(D)$ can be seen  a 
special case of the {\em sequence Kolmogorov spaces} $\ell^2(\lambda)$,  
which we now define. These simple spaces depend on an additional datum, 
that we will call a {\em weight sequence}\index{weight sequence}.

\begin{definition} 
A {\em weight sequence} $\l=(\l_n),\ n \in \NM$ over $]0,S]$ is a sequence 
of positive increasing  functions indexed by $\NM$: 
$$\l_n: I \to \RM_{>0},\;\;\; \l_n(t) \geq \l_n(s) \text{ for } t \ge s.$$
\end{definition}
In applications we will also consider weight sequences that are parametrised by
an arbitrary countable index set rather than $\NM$.
 
To describe the space $\ell^p(\l)$, we first consider the sequence space
\[ \CM^{\infty}\]
consisting of arbitrary sequences
\[ x=(x_0,x_1,x_2,\ldots,x_n,\ldots), \; x_i \in \CM .\]
The elements
\[ e_n:=(0,0,\ldots,0,1,0,\ldots) \in \CM^{\infty}\]
with $1$ at spot $n$ form the {\em standard basis elements} of $\CM^{\infty}$
and we can write
\[x =\sum_{n \geq 0} x_n e_n .\]
If $\l=(\l_n)$ is a weight sequence and  $x \in \CM^{\infty}$, we set
\[ |x|_s:=\left(  \sum_{n \geq 0 } |x_n|^p \l_n^{p}(s) \right)^{1/p} \in \RM_{\ge 0} \cup \{\infty\},\]
where we suppressed the dependence of the norm on $p$ from the notation.

\begin{definition} The Banach space $\ell^p(\l)$ over the interval $]0,S]$
is defined by the Banach spaces
\[ \ell^p(\l)_s:=\{ x \in \CM^{\infty} : \; |x|_s <+\infty \}, \;\;\; s \in \RM, \]
with $|-|_s$ as norm.
\end{definition}

\begin{proposition}
$\ell^p(\l)$ is an $S$-Kolmogorov space.
\end{proposition}
\begin{proof}
As $\l_n(t) \ge \l_n(s)$ for all $t \ge s$, one clearly has
\[ t \ge s \Longrightarrow |x|_t \ge |x_s| \]
and therefore we have inclusions
\[ \ell^p(\l)_t \hookrightarrow \ell^p(\l)_s .\]
These inclusions have clearly norm $\le 1$, so the Banach spaces $\ell^p(\l)_t$,
with these inclusions as restriction maps, form a Kolmogorov space on the interval $]0,S]$. 
\end{proof}

As an example, for the constant weight system  $\l_n=1$ for all $n \in \NM$, we get a ``constant'' Kolmogorov 
space: $\ell^p(\l)_t=\ell^p(\NM)$ for all $t$'s and the restriction maps are all equal to the identity. 
More interesting are the spaces defined by the weight sequence $$ \l_n(s)=s^n.$$ 
In this case the restriction maps 
\[ \ell(\l)^p_t \to \ell(\l)^p_s\]
are limits of finite rank mappings
$$\rho_k:\ell^p(\l)_t \to \ell^p(\l)_s,\ (x_n) \mapsto (x_0,x_1,\dots,x_k,0,0\dots) $$ and are therefore compact operators.\\

For the particular weight sequence $\l_n=s^n$ it is very simple to determine 
the order filtration of the Kolmogorov spaces $\ell^p(\l)$. 

\begin{proposition}
$\ell^p(\l)^{(k)}$ consists of sequences with vanishing first $k$-terms. 
\end{proposition}
\begin{proof}
Indeed if
 $$|x|_s=\left( \sum_{n \geq k} |a_n|^p s^{pn} \right)^{1/p}=s^k\left( \sum_{n \geq 0} |a_{n+k}|^p s^{pn} \right)^{1/p} $$
 thus $x \in E^{(k)}$. Morever if $a_k \neq 0$ then
 $$| x|_s \geq |a_k|s^k $$
 and $x \notin E^{(k+\e)}$ for any $\e >0$. So the condition is necessary and sufficient.  
\end{proof}

\begin{proposition}
Let $D$ be the open unit disc over $\R_{>0}$.
Then the map
$$\Ot^h(D)_s \to \ell^2(\l)_s, f \mapsto (\< f, z^I \rangle_s ) $$
defines an isomorphism between the Kolmogorov spaces 
$\Ot^h(D)$ and $\ell^2(\l)$, where the weight sequence $\l$
is given by
$$\l_I(s)=\sqrt{C(I)} s^{n+|I|}. $$
and 
\[ C(I):=\prod_{k=1}^n \frac{\pi}{i_k+1} .\]
\end{proposition}
\begin{proof} 
The monomials 
\[ z^I:=z_1^{i_1}z_2^{i_2}\ldots z_n^{i_n}\] 
form an orthogonal Hilbert basis of $\Ot^h(D)$. One computes
$$\< z^I,z^J \rangle_s =\int_{D_s}z^I \bar z^J d \mu = C(I)s^{2n+|I|+|J|} \dt_{IJ} $$
where
\[ C(I) =\prod_{k=1}^n \frac{\pi}{i_k+1} . \]
In particular:
$$|z^I|_s=C(I)^{1/2}s^{n+|I|} .$$
\end{proof}

\section{Kolmogorov spaces and coherent sheaves}
\label{S::privilege}
In the previous sections, we introduced the Kolmogorov space versions $\Ot^b$, $\Ot^c$ and
 $\Ot^h$ of the sheaf of functions $\Ot$. In applications one not only encounters
functions, but also vector fields, differential forms and further tensor fields.
These objects can be seen as sections of vector bundles and form locally free
sheaves. Given a free sheaf $\Ft$ there is no difficulty in defining its Kolmogorov
versions $\Ft^a$, $a \in \{b,c,h\}$, as for sufficiently small $U$ one has
\[\Ft(U) = \bigoplus_{i=1}^p \Ot(U) e_i\]
for some choice of generators $e_i$, and one can simply put
\[\Ft^a(U):=\bigoplus_{i=1}^p \Ot^a(U) e_i,\;\;\textup{for}\;\;\;a \in \{b,c,h\} .\]
In this way one can speak for instance about Kolmogorov spaces of vector fields $\Theta_X^c$ over a smooth complex manifold and more generally of arbitrary tensors. Of course, this depends on the choice of generators, but this dependence
is usually not important.\\

However, often one has to deal with other types of sheaves, for example ideal sheaves 
$\It \subset \Ot$, or functions defined on a submanifold like $\Ot/\It$ and such sheaves 
are usually not locally free. The question arises if it is possible to define Kolmogorov  versions $\Ft^a$, $a \in \{b,c,h\}$ for a general coherent sheaf.
Recall the sections of a coherent sheaf $\Ft$ over sufficiently small open subsets $U$
can be be described by a {\em presentation}, i.e. an exact sequence of the form
\[ \Ot(U)^q \stackrel{A}{\to}  \Ot(U)^p \to  \Ft(U) \to 0\]
where 
\[ A \in Mat(p \times q, \Ot(U))\]
is a matrix of functions, holomorphic on $U$, so $\Ft(U)=Coker(A)$. 
Given such a presentation, one is tempted to use the same matrix $A$ 
(maybe on a slightly smaller open set) to define $\Ft^a$-versions by setting
\[ \Ft^a(U):=Coker(\Ot^a(U)^q \stackrel{A}{\to}  \Ot^a(U)^p)\]
for $a \in \{c,b,h\}$.

A complication arises from the fact that the map $A$ may not have a closed
image, and as a result the cokernel will not have the structure of a Banach
space. It was {\sc H. Cartan} who realised that in  most situations one may 
circumvent this difficulty. We will use the following terminology:

\begin{definition} Let $\Ft$ be a coherent sheaf on $\CM^n$. An open subset $U \subset \CM^n$ is
called  $\Ft$-open,~\index{open set for a coherent sheaf} if $\Ft^a(U)$ is a
 Banach space for $a \in \{b,c,h\}$.
(These are more commonly called {\em privileged neighbourhood} for the sheaf $\Ft$.)
\end{definition}

\begin{theorem}
For any coherent sheaf $\Ft$ there exists open subsets $\Omega \subset GL(n,\CM)$ and
$\D \subset \RM^n_+ $ containing $0$ in its closure, such that $AD_r$ is $\Ft$-open for any $A \in \Omega$, $r \in \D$.
\end{theorem}

Note that $\Ft$-open sets can be glued: the union of two $\Ft$-open sets is again an $\Ft$-open set. Therefore the theorem implies the existence of privileged neighbourhoods over any compact sets. A proof of the theorem is given in the appendix. 

The category of coherent sheaves admits version of kernels and cokernels in the 
Kolmogorov space setting:

\begin{corollary}
 Consider a morphism of coherent sheaves on $\CM^n$.
 $$\Ft \stackrel{u}{\to} \Gt. $$
 For each point $ p \in \CM^n$ there exists a fundamental system of $S$-open sets $U_s$, $s \in ]0,S]$,
such that the kernels, images  and Cokernels of the maps
\[ \Ft^a(U) \stackrel{u}{\to} \Gt^a(U),\;\;\;a \in \{b,c,h\}\]
are $S$-Kolmogorov spaces.
\end{corollary}

\chapter{Partial morphisms of Kolmogorov spaces}
\section{An example}
A homomorphism  $u:E \to F$ between Kolmogorov-spaces over $B$ defines in
particular for each $b \in B$ a linear continuous map
\[ u_b: E_b \to F_b .\]
In many geometrical situations one has to deal with mappings of a different 
kind. One does not have linear maps $E_b \to F_b$ for $b \in B$, 
but rather on is working with maps $E_a \to F_b$ for certain {\em pairs} 
$a$ and $b$.\\

Let us give a nice, crazy example. Consider the set $D=(D_t) $ over $B=\RM_{>0}$ 
defined by
$$D_t:=\{ x+iy \in \CM: x^2+ty^2 < t^2 .\} $$
For $t=1$, any rotation 
$$R(\theta):z \mapsto e^{i\theta}z$$ induces a morphism
$$R(\theta)^*:\Ot^c(D)_1 \to \Ot^c(D)_1,\ f(z) \mapsto f(e^{i\theta}z).  $$
Due to the fact that the set $D_t$ is a round disc only for $t=1$, 
there are no obvious maps
\[ \Ot^c(D)_t \to \Ot^c(D)_t\]
that extend the map $R(\theta)^*$ one has for $t=1$.\\
\vskip0.3cm
\begin{figure}[htb!]
\includegraphics[width=7cm]{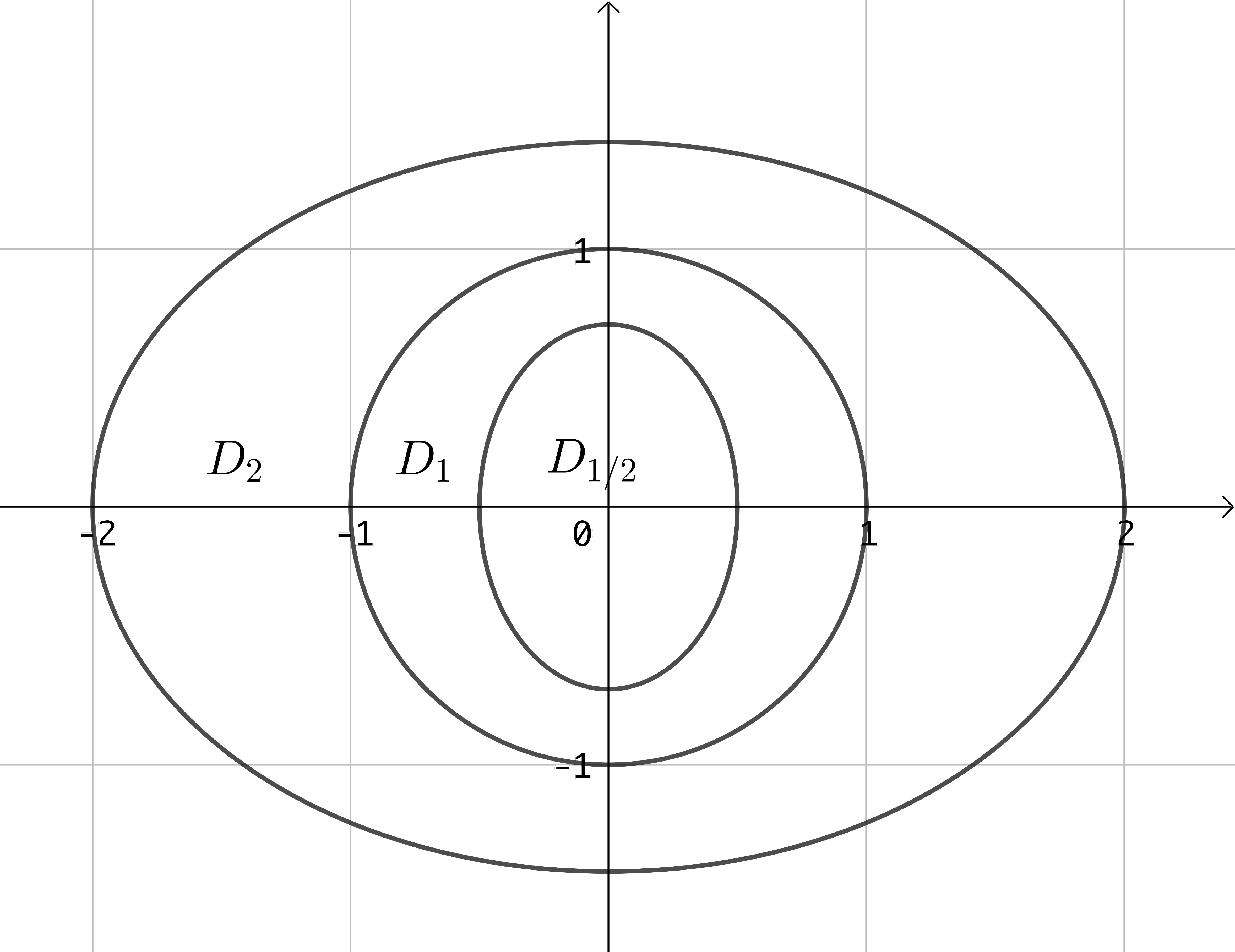}
\end{figure}
\vskip0.3cm
To be specific, let us look at the rotation $r:=R(\pi/2)$ over $\theta=\pi/2$.
The ellipse $D_t$ contains the rotated ellipse $r(D_s)$ if
and only if $s \le \sqrt{t}$ and $\sqrt{s} \le t$, i.e. iff
$$s<\min(\sqrt{t},t^2) $$
If these inequalities holds, the map $r$ maps $D_s$ into $D_t$:
\[ r: D_s \to D_t, (x,y) \to (-y,x)\]
and we obtain induced maps
\[ r_{st}:=r^*: \Ot^c(D)_t \to \Ot^c(D)_s\]
for all pairs $(t,s)$ in the following region.\\

\begin{center}
\includegraphics[width=12cm]{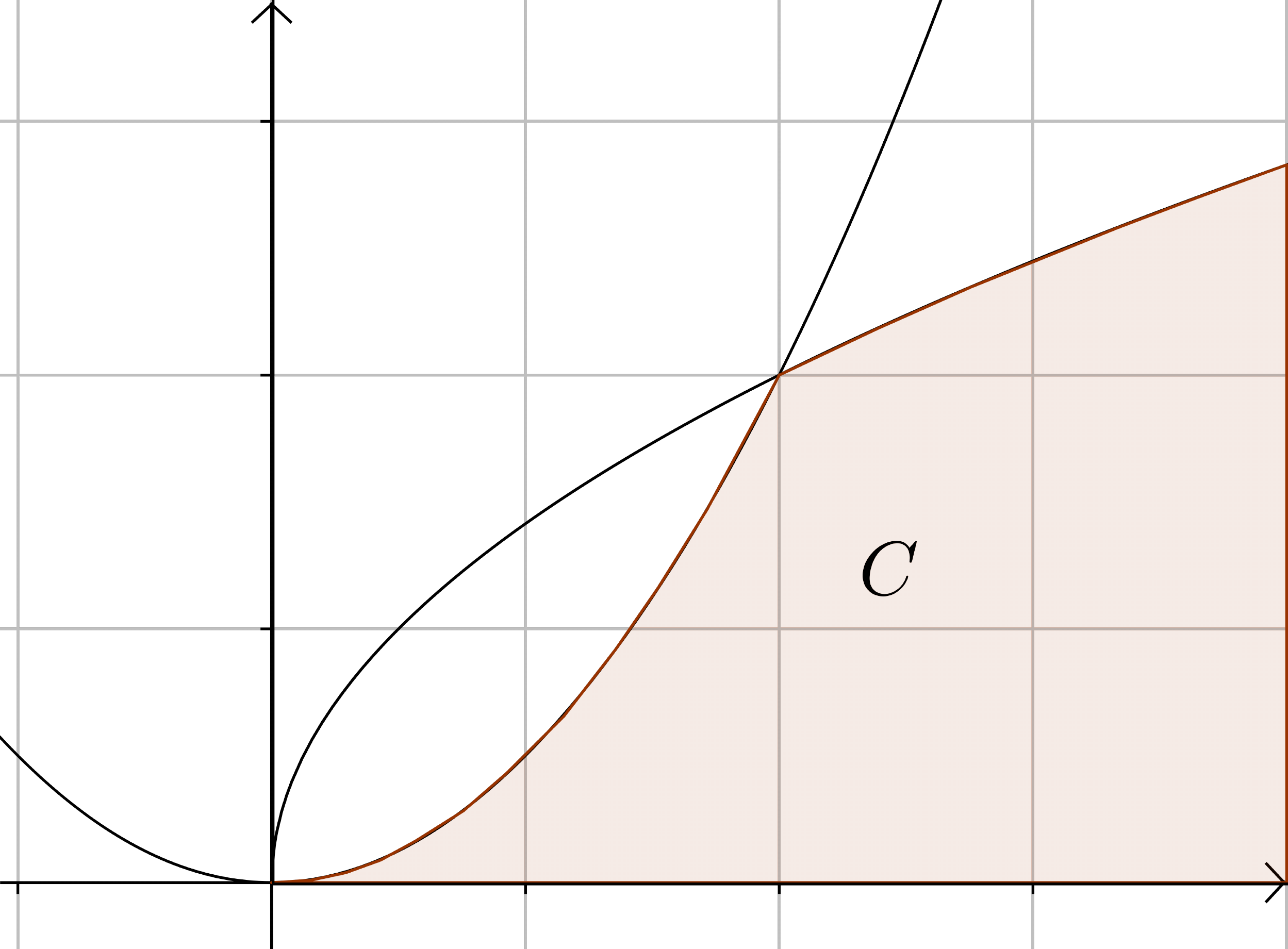}
\end{center}

Note that these mappings $r_{st}$ are not arbitrary, but rather are
compatible with restrictions: if $s' \le s$ and $t \le t'$, then one
has the relation
\[  r_{s't'} =e_{s's} \circ r_{st} \circ e_{tt'},\]
where the $e$'s denote the restriction mappings in the Kolmogorov 
space $\Ot^c(D)$.\\

We will call such thing a {\em partial morphism}\index{partial morphism} 
of $\Ot^c(D)$ to $\Ot^c(D)$. 
From this example it should already be clear that the appearance of such
partial morphisms is the rule rather than the exception.
 
 \section{Partial morphisms}
 \begin{definition}\index{partial morphism} 
Let $E,F$ be Banach spaces over ordered sets $A,B$.
A {\em partial morphism} from $E$ to $F$ over $C \subset A^{op} \times B$ is 
a horizontal section of $\mathcal{H}om(E,F)$ over $C$. We use the notation
\[ \Hom_C(E,F):=\G^h(C,\mathcal{H}om(E,F))\]
to denote the vector space of all partial morphisms over $C$.
We will give it the topology of pointwise convergence, but this topology will play no big role here.
\end{definition}

We will often write partial morphisms just as ordinary morphisms
\[  u: E \to F\]
meaning that for $(t,s) \in C \subset A \times B$ there is a corresponding map
\[ u_{st}: E_t \to F_s . \]
We call $C$ the {\em definition set}\index{definition set} of the partial morphism and write
$$C=\Dt(u) . $$

{\bf \em Restriction and Extension}\\
Obviously, a partial morphism can be restricted to a smaller set: 
if $C \subset D$, there is a natural restriction 
map
\[ \Hom_D(E,F) \to \Hom_C(E,F),\;\;u \mapsto u_{|C} .\]

Conversely, if $u \in \Hom_C(E,F)$ is a partial morphism with definition 
set $C$, one may try to extend $u$ to a larger set $D \supset C$.

For example, if we consider a {\em single map} $u_{ba}:E_a \to E_b$ as a partial 
morphism defined at the point $C=\{(a,b)\}$, we may extend it to the downset
\[ C^d:=\{(t,s) \in A^{op} \times B\;|\; t \le a, \,s \le b\}\]
using the restriction maps $e$ resp $f$ of $E$ resp. $F$, by defining
\[ u_{st} = f_{sb} u_{ba} e_{at}:E_t \to F_s .\]
Recall that $t \le a$ in $A^{op}$ means concretely that $t \ge a$ in $A$.

Nevertheless given $u \in \Hom_C(E,F)$ and two points $p=(a,b), q=(c,d) \in C$, 
the extension of $u$ on the downset of $p=(a,b)$ and the extension on the
downset of $q=(c,d)$ need not to coincide on their intersection.
But this is guaranteed if 
$$p,q \in C \implies \inf (p,q)=(c,b) \in C.$$

\begin{center}
\includegraphics[width=11cm]{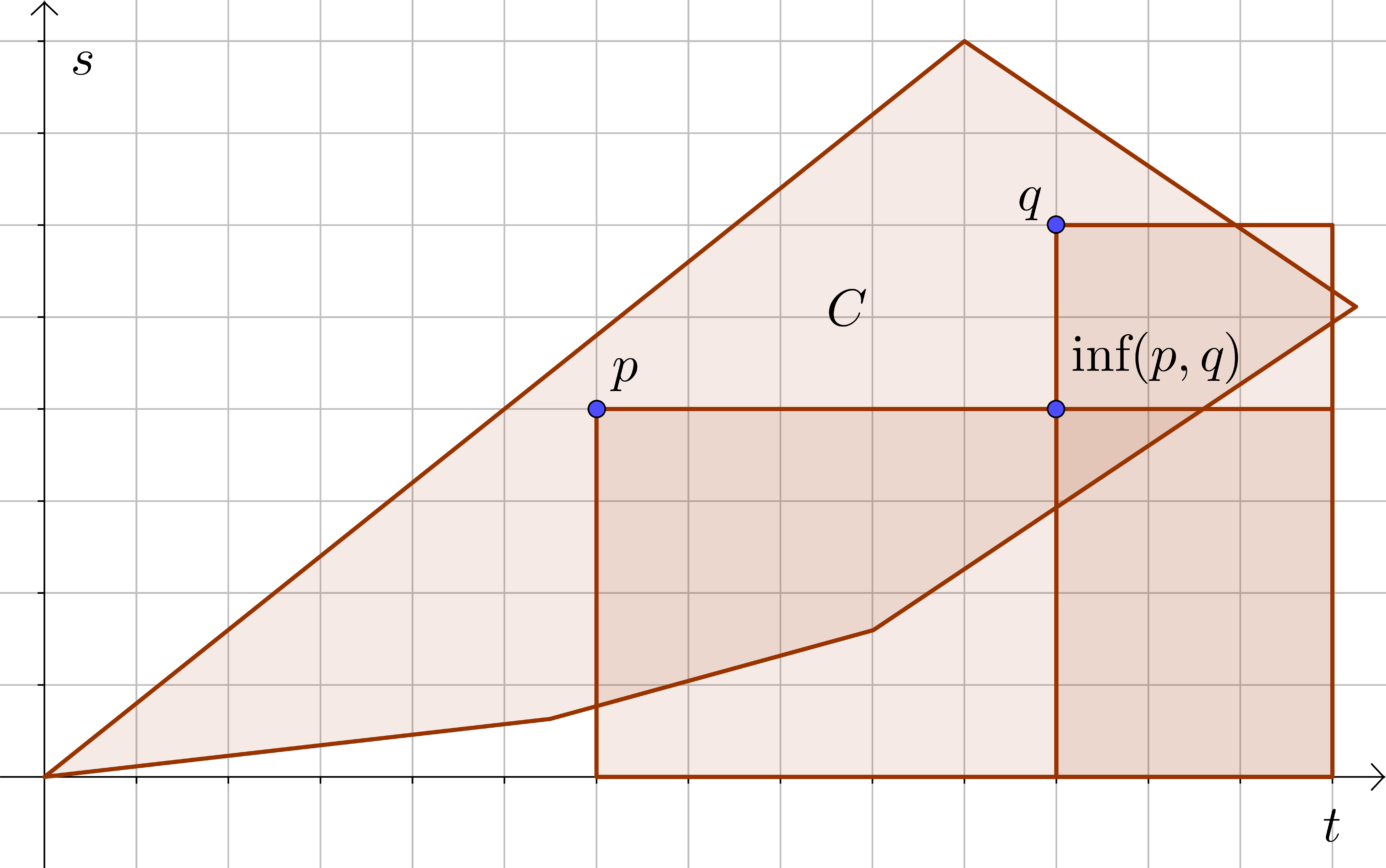}
\end{center}

\begin{definition} For a partial ordered set $B$, denote by
\[ \overline{\D}:=\{ (s,t) \in B^{op} \times B\;|\;\;s \le t\}\]
the closed subdiagonal\index{closed subdiagonal}.
If $E$ and $F$ are Kolmogorov spaces over $B$, we call  
\[\Hom_{\overline{\D}}(E,F)\]
the space of {\em complete morphisms}\index{complete morphism}.
\end{definition}

The restriction map
\[  \Hom_{\overline{\D}}(E,F) \to \Hom_{\delta}(E,F),\]
to the diagonal
\[\delta:=\{(b,b) \in B^{opp}\times B \} \]
plays a special role. An element $u$ of $\Hom_{\delta}(E,F)$ is an 
{\em arbitrary collection} of maps 
\[ u_b:E_b \lra F_b ,\]
whereas the image of the above restriction map consist of those
which commute with the restriction maps in $E$ and $F$:
\[ f_{ab}u_b=u_ae_{ab}\]
which is expressing the fact that the $u_b$ define a morphism in
the category, i.e. an element of $\Hom_B(E,F)$.
We have seen in \ref{P::complete_morphisms} 
\[  \Hom_{\overline{\D}}(E,F) \to \Hom_{B}(E,F)\] 
is an isomorphism of vector spaces, which shows that complete morphism
can be identified with morphisms in the category of Kolmogorov spaces.\\

\section{Basic examples}
The general context of the example used in the introduction is the following.
Let $\p: X \to X$ be a holomorphic mapping of an analytic space $X$ to
itself. If $U$ and $V$ are relative compact open subsets of $X$ with the
property that 
\[ \p(U) \subset V,\]
there is an associated map of Banach spaces
\[ \p^*_{UV}: \Ot^c(V) \to \Ot^c(U); f \mapsto f \circ \p .\]
These maps $\p^*_{UV}$, for various open sets $U$ and $V$, are compatible 
with restriction mappings, so we can say that this construction defines a
horizontal section $\p^*$ of some Kolmogorov space, more precisely, an
element $\p^c:=(\p^*)^c$ of
\[ \Hom_C(\Ot^c,\Ot^c):=\G^h(C,\mathcal{H}om(\Ot^c,\Ot^c)),\]
where we consider $\Ot^c$ as a Kolmogorov space over $\Tt^c$
and 
\[C:=\{(U,V) \in (\Tt^c)^{op} \times \Tt^c\;|\;\p(U) \subset V\} .\]

When we consider  $S$-Kolmogorov spaces like $\Ot^c(D)$, where $D$ is a 
relative compact open set over $]0,S]$, we are dealing with a pull-back of 
$\Ot^c$ via a map $$]0,S] \to \Tt^c,\;\; s \mapsto D_s$$ 
and can do the analogous construction. The map $\p$ induces partial 
morphism $\p^c$ over the set
\[A:=\{ (t,s) \in ]0,S] \times ]0,S]\;|\;\p(D_s) \subset D_t\},\]
i.e:
\[\p^c \in \Hom_A(\Ot^c(D),\Ot^c(D))=\G^h(A,\mathcal{H}om(\Ot^c(D),\Ot^c(D))\]

It is instructive to look at some simple examples.\\

If $\a \in \CM^*$ is a fixed complex number, then the scaling map
\[ \p: \CM \to \CM, z \mapsto \a z\]
induces a a ring homomorphism
$$\p^*:\CM\{ z \} \to \CM\{ z \},\ f(z) \mapsto f(\a\,z) .$$
on the power-series ring. Let us try to see what partial morphism
on the Kolmogorov space $\Ot^c(D)$ is induced by $\p$, 
where $D=(D_t)$ is the unit disc over $B=\RM_{>0}$.
Clearly, $\p(D_s) \subset D_t$ precisely if $|\a|s \leq t$.
The above procedure then produces a partial morphism 
\[ \p^c \in \Hom_{A_\a}(\Ot^c(D),\Ot^c(D)),\;\;\;\]
where the definition set is the cone
\[ A_\a=\{ (s,t) \in \RM^2: |\a|\,s \leq t \}. \]
\begin{center}
 \includegraphics[height=6cm]{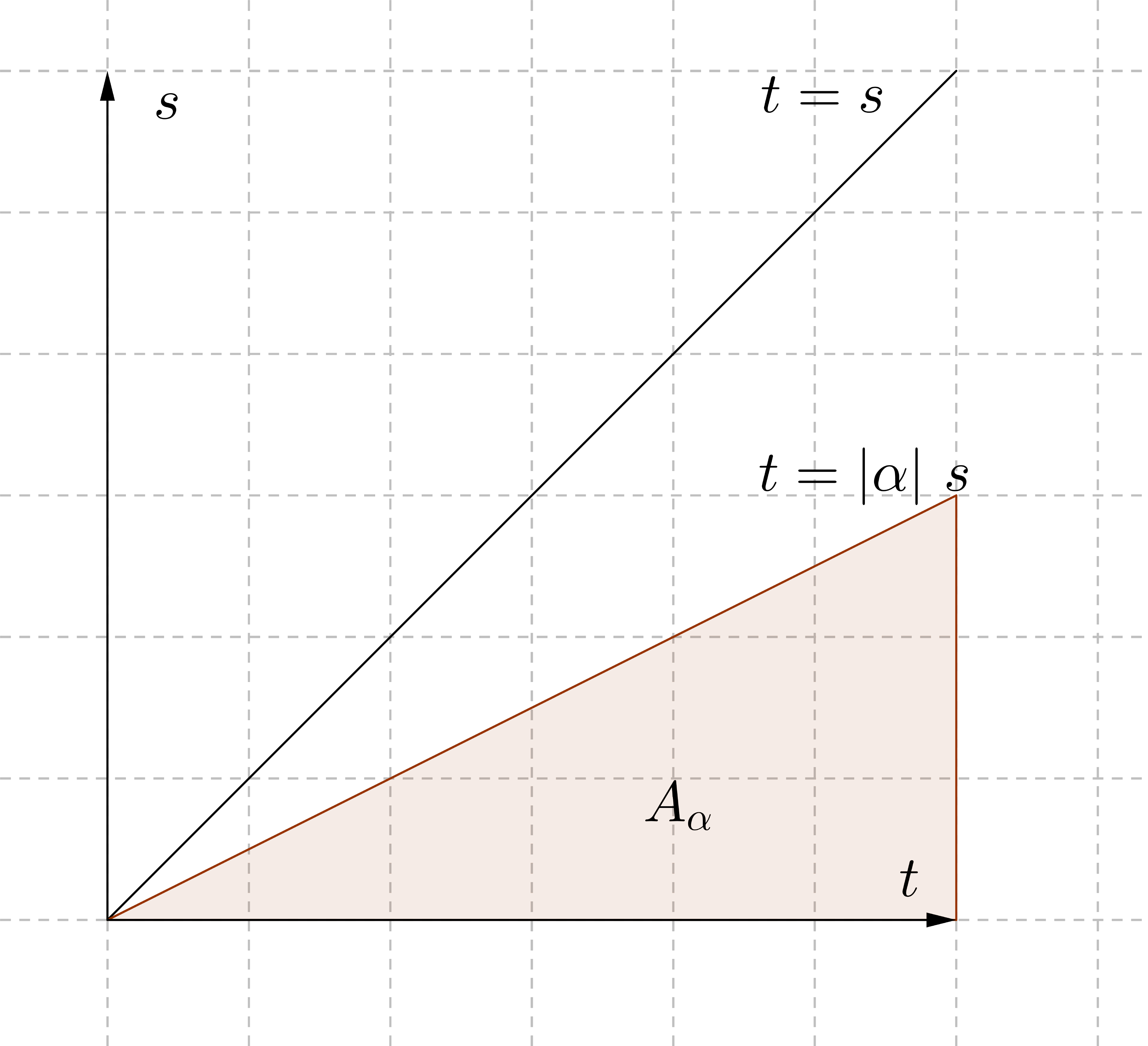}
\end{center}
What happens with maps which are not one-to-one? Say we consider the map
$$\p:\CM \to \CM,\ z \mapsto  z^2 .$$
Composition with $\p$ defines again a ring homomorphism:
$$\p^*:\CM\{ z \} \to \CM\{ z \},\ f(z) \mapsto f(z^2) .$$
If $f$ is holomorphic inside the disk of radius $t$, then the function $f(z^2)$
is defined for all $z$ with $|z^2|\leq t $.
This means that $\p$ induces a partial morphism 
\[ \p^c \in \Hom_{A}(\Ot^c(D),\Ot^c(D)),\;\;\;\]
with definition set
\[ A:=\{ (s,t) \in \RM^2: s^2 \leq t \}. \]
\begin{center}
 \includegraphics[height=6cm]{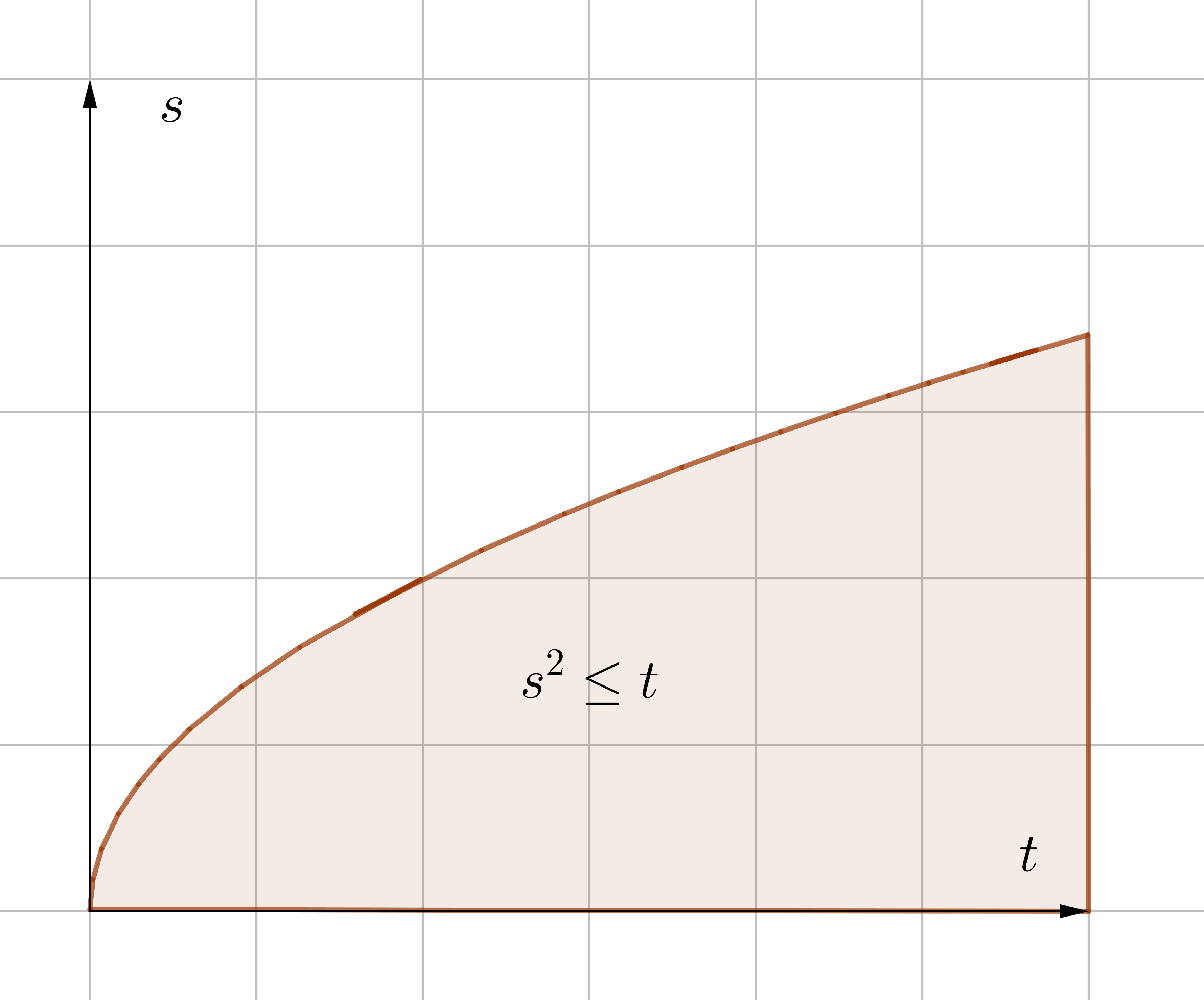}
\end{center}
So from the geometry of the definition set, we may already distinguish maps 
which do not correspond to biholomorphic germs: their boundary have a vertical 
tangent at the origin.\\ 
Let us now investigate what happens for maps which do  not fix the origin. 
Consider for instance the translation by a vector $\lambda \in \CM$:
$$\tau:\CM \to \CM,\ z \mapsto z+\lambda .$$ 
Composition with $\tau$ defines 
a morphism
$$\tau^*:\Ot_{\CM,0} \to \Ot_{\CM,\l},\   f(z) \mapsto f(z+\lambda)$$
which maps a germ a the origin to a germ at $\l$.

It induces a partial morphism
\[\tau^c \in \Hom_A (\Ot^c(D) , \Ot^c(D))\]
over the set 
\[A:=\{(s,t) \in \RM^2 : s <t-|\lambda|\}\]
Geometrically the base $A$  of the vector space $\Hom_A(\Ot^c(D),\Ot^c(D))$
does not contain the origin in its closure and this means precisely that our original map $\tau$ does not fix the origin.\\
\begin{center}
\includegraphics[height=5cm]{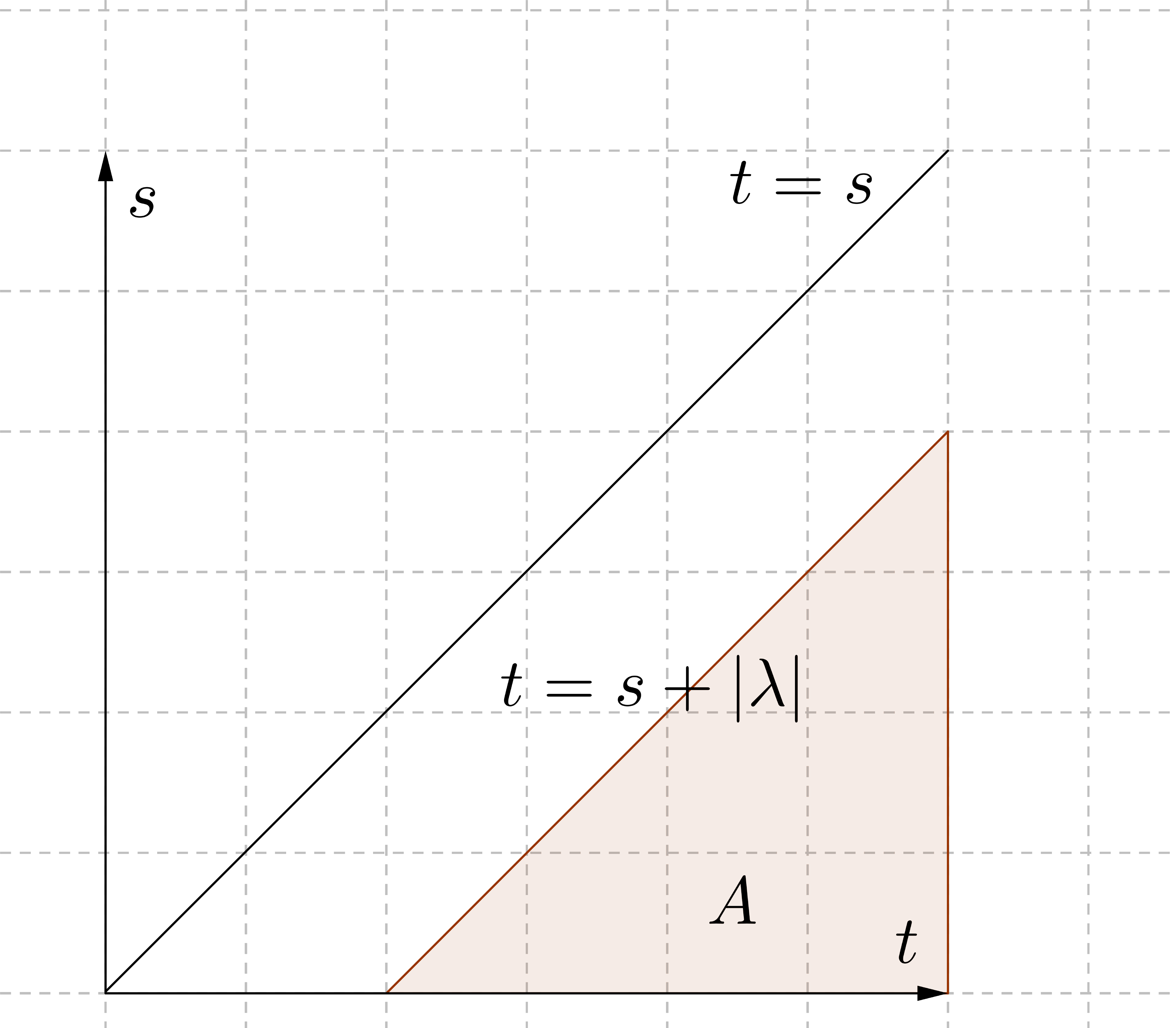}
\end{center}

From these very simple examples, we already can see that definition sets 
will play an important role in functorial analysis. We will see that much
many technical points in KAM theory reduce to simple statements
in the geometry of definition sets.
\section{Composition of partial morphisms}
How to compose such partial morphisms? 
If we are given partial morphisms $u: E \to F$ and $v: F \to G$ we are 
given maps
\[ u_{st}: E_t \to F_s,\;\;(t,s) \in A=\Dt(u) \]
and
\[ v_{rs}: F_s \to G_r,\;\;(s,r) \in B:=\Dt(v) \]
So if we have pairs $(t,s) \in A$ and $(s,r) \in B$, we
can define the composition
\[ v_{rs} \circ u_{st} : E_t \to G_r\]
This composition may depend on the choice of $s$, but if $s' \le s$,
then we have, using the horizontality of $u$ and $v$:
\[v_{rs'} \circ u_{s't}=v_{rs'}\circ f_{s's} \circ u_{st}=v_{rs}\circ u_{st} .\]
So if $s$ and $s'$ are comparable in the order, we obtain the same map
as composition. So if the index sets are {\em totally ordered},
e.g in the case of morphisms between $K1$-spaces, there can not arise
ambiguity.\\ 

This motivates the following two definitions:\\

\begin{definition}\label{D::convolution}\index{convolution of sets}  
The {\em convolution} of $A_1 \subset B_1 \times B_2$ and 
$A_2 \subset B_2 \times B_3$ is the set
\[A_1 \star A_2:=\{(r,t) \in B_1 \times B_3\;|\;\;\exists s \in B_2\;\;\textup{such that} (t,r) \in A_1, (r,s) \in A_2\}\] 
\end{definition}

The above given definition of composition then results in the following

\begin{proposition} Let $E_i \to B_i,\ i=1,\dots,3$ be K1-spaces.
The composition $v \circ u$ of two partial morphisms $u \in \Hom_{A_1}(E_1,E_2)$ and  $v \in \Hom_{A_2}(E_2,E_3)$ 
is a well-defined element of
\[\Hom_{A_1\star A_2}(E_1,E_3),\]
The map
\[\Hom_{A_1}(E_1,E_2) \times \Hom_{A_2}(E_2,E_3) \to \Hom_{A_1 \star A_2}(E_1,E_3)\]
is bilinear in both variables.
\end{proposition}
Of course, it may very well happen that $A_1 \star A_2$ is the empty set.
\section{Convolution of sets}
The most important case arises when we compose partial morphisms between
$S$-Kolmogorov space. Let us take a closer look at this case. We put 
$$C:=]0,S].$$
The convolution of sets defined in \ref{D::convolution} gives an associative mapping
$$\Pt(C^\text{op} \times C) \times \Pt(C^\text{op} \times C) \to \Pt(C^\text{op} \times C),\
(A,B) \mapsto A \star B $$

It is useful to look at some simple examples of this operation.
We earlier encountered for $\a \in \RM_{>0}$ the set
\[ A_\a=\{ (s,t) \in ]0,S] \times ]0,S]: \a\,s \leq t \}. \]
Obviously,
$$A_{\a} \star A_{\b}=A_{\a\b} .$$
More generally, if $A$ and $B$ are defined via increasing functions $f,g$ with
$$A=\{ (t,s): s < f(t) \},\  B=\{ (t,s): s < g(t) \}$$
then
$$A \star B= \{ (t,s)| s < g \circ f(t) \}.$$
Moreover if
\begin{align*}
f(t)=\a t+o(t)+\dots,\\
g(t)=\b t+o(t) .
\end{align*}
then $A \star B$ is bounded by a curve admitting a tangent at the origin with slope $\a\b$.
\begin{center}
\includegraphics[height=6cm]{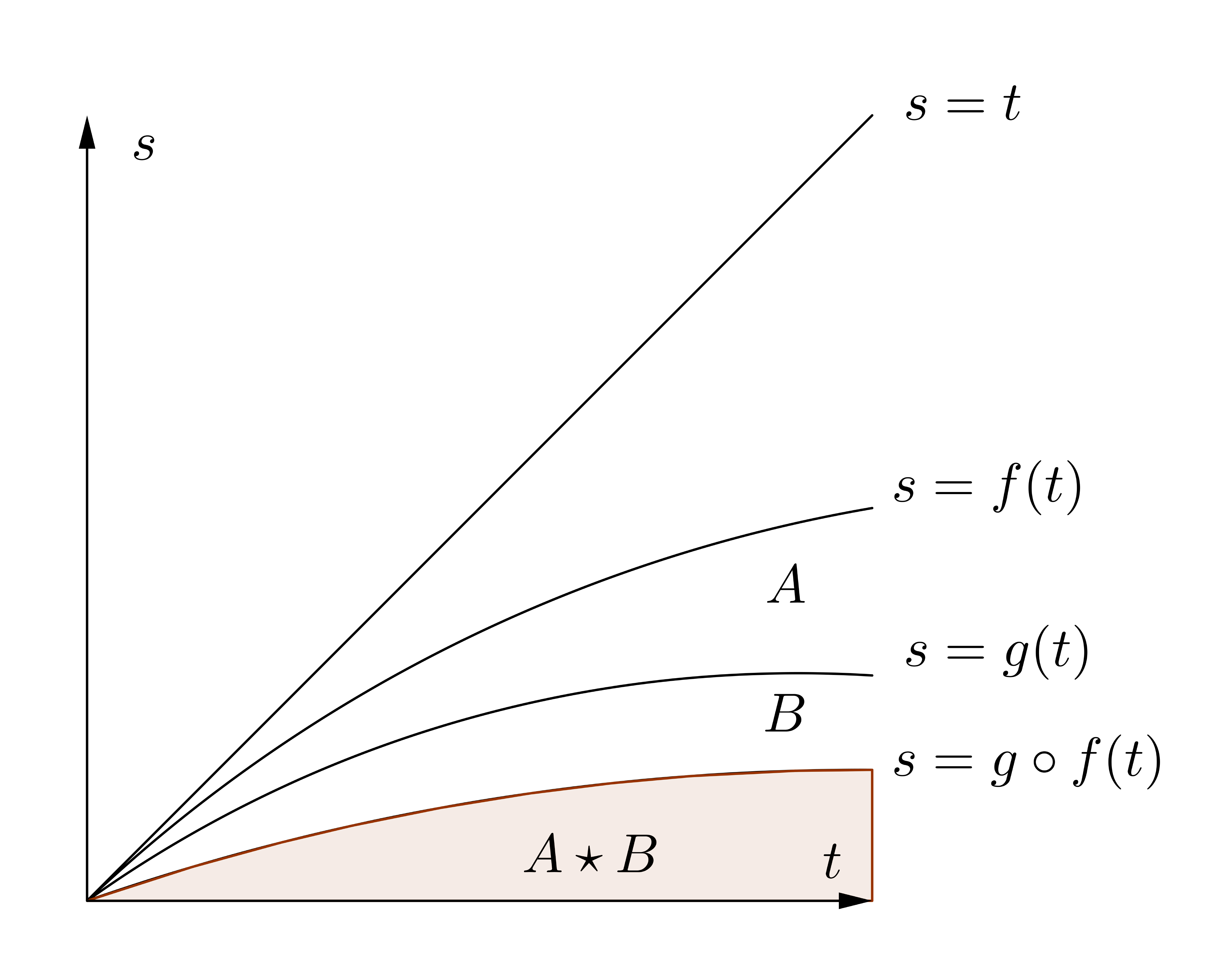}
\end{center}

A set $A \subset C^{op} \times C$ is called a {\em idempotent} if
\[ A= A \star A\]
Suchs sets are of some importance, as the space of sections over $A$ has 
the structure of an {\em algebra} under composition:
\[ \Hom_A(E,E) \times \Hom_A(E,E) \to \Hom_{A\star A}(E,E)=\Hom_A(E,E)\]
The closed subdiagonal
\[ \overline{\D}:=\{ (s,t) \in   C^\text{op} \times C: s \leq t\}\]
is obviously an {idempotent}; the algebra $\Hom_{\overline{\D}}(E,E)$
is the algebra of complete morphisms.

A downset $A$ satisfies the equalities
$$A \star \overline{\D}=\overline{\D} \star A=A  $$
In particular if we form the set
$$\overline{\D} \star A\star \overline{\D}  $$
we get the downset hull of $A$.\\

Clearly, also the {\em open diagonal $\D$}\index{open subdiagonal}
\[ {\D}:=\{ (s,t) \in   C^\text{op} \times C: s < t\}\]
is an idempotent. The algebra 
\[\Hom_{\D}(E,E)=\G^h(\D,\mathcal{H}om(E,E))\]
is called the {\em algebra of almost complete morphisms}\index{almost complete
morphism}.
More generally, for a point $(s,t) \in C^\text{op} \times C$ we can consider the
closed and open {\em upset-trangles} of the point $(s,t)$:
\[\overline{\D}(s,t):=\{(s',t') \in \overline{\D}:  s' \ge s, t' \ge t\}\]
\[\D(s,t):=\{(s',t')\in \D:  s' \ge s, t' \ge t \]
are idempotents. As the order on the first factor is reversed for $C=\RM_{>0}$, we get the following picture:\\

\begin{center}
\includegraphics[height=6cm]{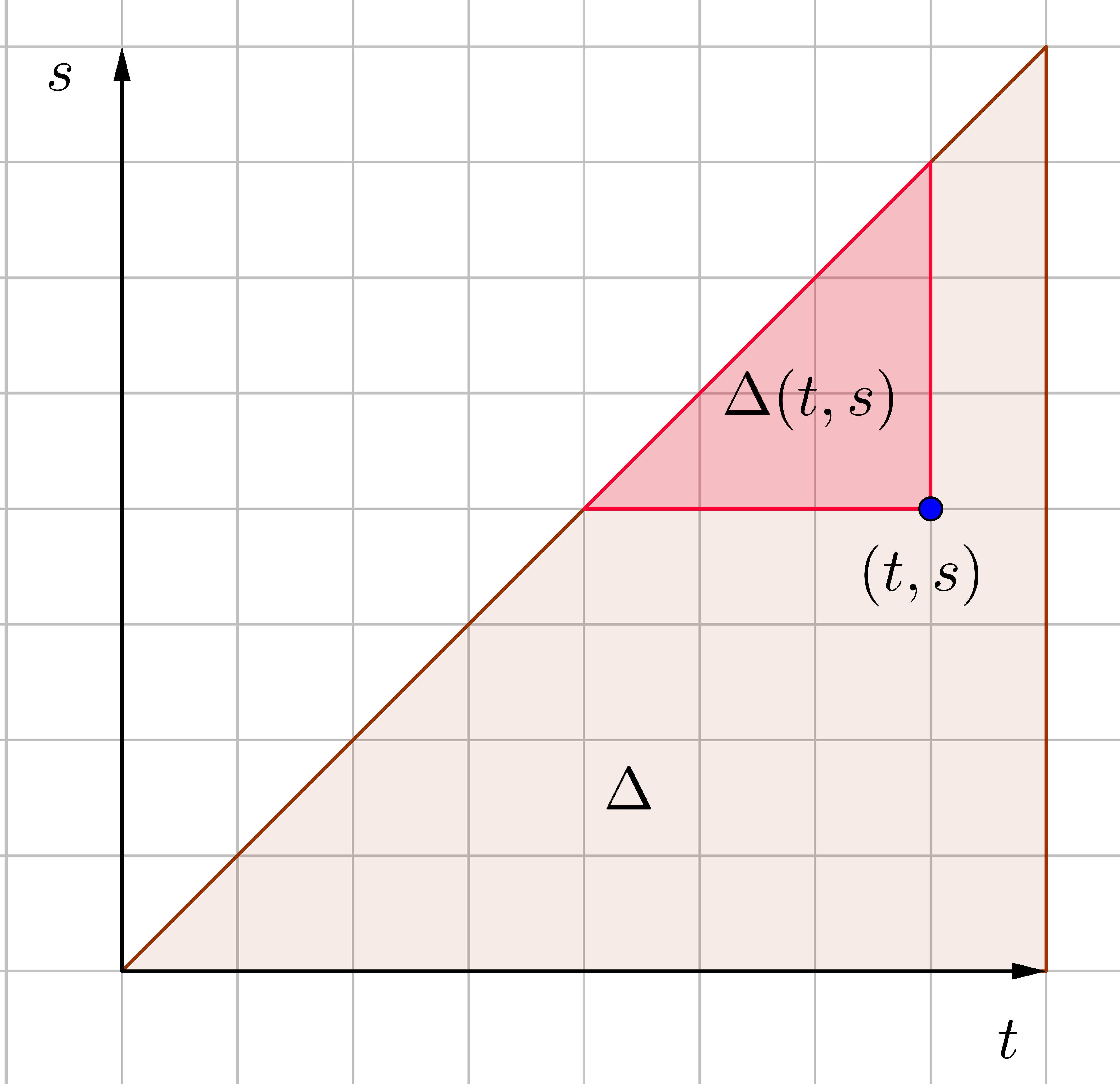}
\end{center}

As any upset in $\D$ (or in $\overline{\D}$) is the union of the 
upset triangles it contains, {\em any} upset $A \subset \D$ is an idempotent:
\\

\begin{center}
 \includegraphics[height=6cm]{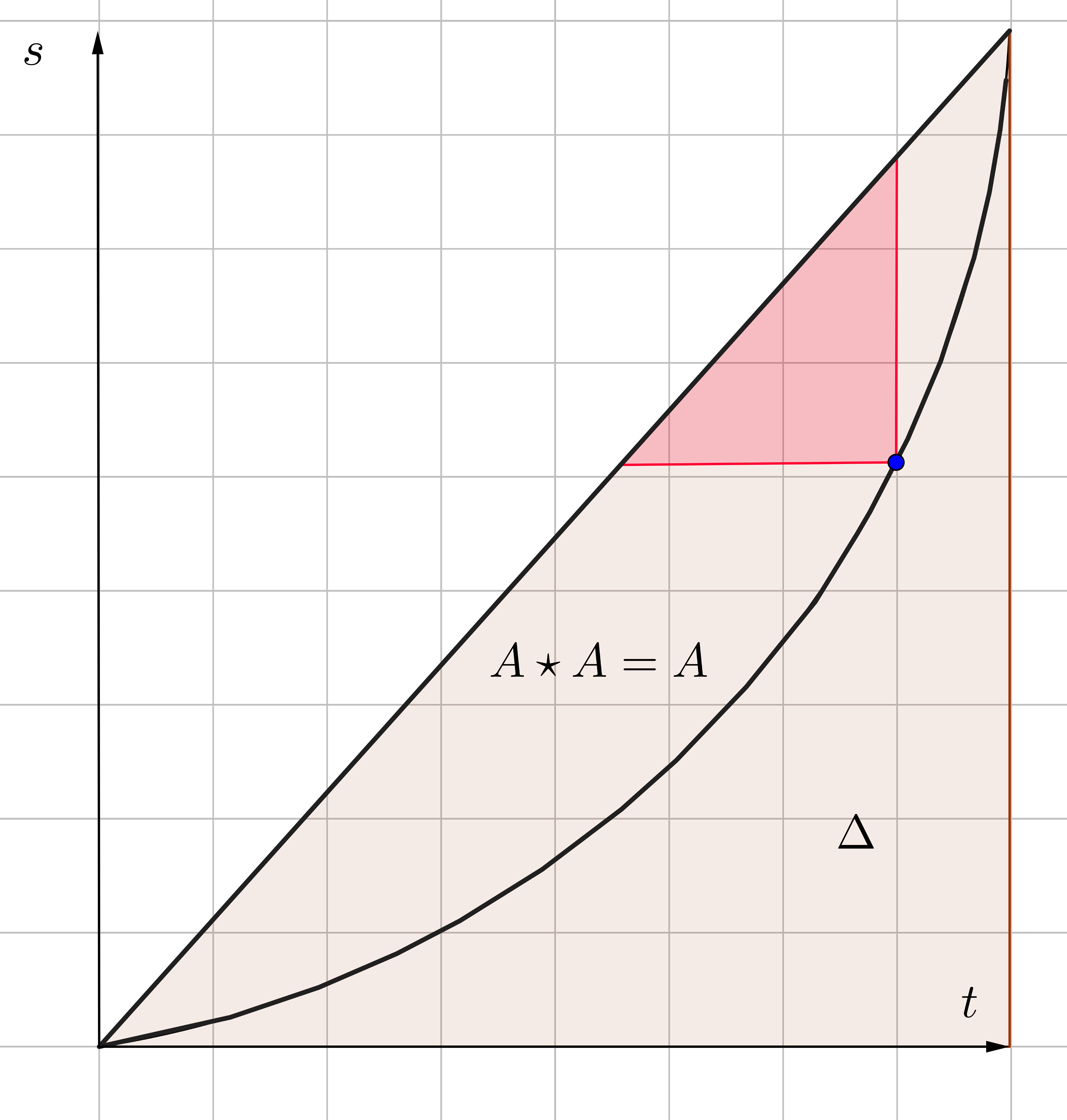}
\end{center}

Certain sub-algebras of $\Hom_{\D(s,t)}(E,E)$ will play a role in the 
next chapter.\\
 
\section{Pseudo-identity and pseudo-inverse} 
Let $E \to C, F \to C$ be Kolmogorov spaces over $C=]0,S]$. Recall that
restriction to the diagonal defines an isomorphism
\[ \Hom_{\overline{\D}}(E,F) \to \Hom_B(E,F) .\]
Using the composition with the restriction maps $e_{st}$ of $E$ and 
$f_{st}$ of $F$ and $u_t: E_t \to F_t$ we obtain for $s \le t$ maps 
\[ u_{st}:=f_{st}\circ t_t=u_s\circ e_{st}: E_t \to F_s ,\]
which are the components of the element in $\Hom_{\Delta}(E,F)$.

We can apply this in particular to the case $E=F$, $u=Id$ and obtain            
a canonical element
\[\iota_{\overline{\D}} \in \Hom_{\overline{\D}}(E,E) \]        
which has the restriction maps 
$$e_{st}:E_t \to E_s,\ t>s $$
as components.

If $A \subset \overline{\D}$ is any subset, there is a restriction map
\[\Hom_{\overline{\D}}(E,E) \to \Hom_A(E,E)\]

\begin{definition} The pseudo-identity\index{pseudo-identity} 
\[\iota_A \in \Hom_A(E,E)\] 
is the restriction of $\iota_{\overline \D} \in \Hom_{\overline \D}(E,E)$. 
\end{definition}



\begin{definition} Let $\p \in \Hom_A(E,F)$ and $\psi \in \Hom_B(F,E)$ be
two partial morphisms such that 
 $$\p \circ \psi= \iota_{A \star B} \in \Hom_{A \star B}(F,F),$$
then we call $\p$ a left pseudo-inverse to $\psi$ and
$\psi$ right pseudo-inverse to~$\p$.
\end{definition}

Of course, this is only interesting if $A \star B$ is non-empty. In general, it is not possible to have
$$A \star B=\overline{\D}.$$
But in situations where there is no loss of regularity this very well can
happen. Let us give a simple example. 
The map
$$\CM \to \CM, z \mapsto \a z, \ a \in \CM^* $$
induces a ring automorphism
$$\p:\CM\{ z \} \to \CM\{z \},\ f(z) \mapsto f(z/\a) $$
whose inverse is
$$\psi:\CM\{ z \} \to \CM\{z \},\ f(z) \mapsto f(\a z). $$
Let us take $E$ to be the Kolmogorov space
$$\Ot^b \to \Tt_0 $$
over neighbourhoods of the origin.

The map $\p$ can be seen as an element of $Hom_A(\Ot^b,\Ot^b)$ where
$$A=\{ (U,V) \in  \Tt_0^{op} \times  \Tt_0:V \subset \a U \} $$
whereas $\psi$ may be regarded as an element of 
$Hom_B(\Ot^b,\Ot^b)$ with
$$B=\{ (U,V) \in  \Tt_0^{op} \times  \Tt_0:V \subset \a^{-1} U \}. $$
One has
$$A \star B=B \star A=\overline{\D} $$
and 
$$\p \circ \psi=\psi \circ \p=\iota_{\overline{\D}} \in  \Hom_{\overline \D}(\Ot^b,\Ot^b).$$
Therefore in the Kolmogorov space setting, the maps are pseudo-inverses to each other and, passing to the direct limit these pseudo-inverses become inverses automorphisms of the ring $\CM\{ z \}$.

%
%
\chapter{Local morphisms}
In functional analysis a major role is played by the operators that
behave like differential operators, and which might be called 
{\em local operators.} This notion is rather vague in general, but 
in the context of Kolmogorov spaces we will define a precise notion 
of locality.
\section{The derivative}
The first important property of differential operators is that they only give rise to 
partial morphisms of Kolmogorov spaces. To explain this point, let us start with simple examples.

If we differentiate a $C^k$ function, we get a $C^{k-1}$ function. So
 we consider 
 $$E:=C^\bullet([0,1],\RM) \to \NM$$ as a Kolmogorov space
 with fibre $E_k:=C^k([0,1],\RM)$. The inclusions 
 $$C^{k+1}([0,1],\RM) \to C^k([0,1],\RM) $$
 give restriction mappings and therefore a Kolmogorov space structure over $(\NM,\geq)$.

The derivative gives a well-defined maps 
$$E_{k+1} \to E_k, f \mapsto f'$$ 
and thus we obtain a partial morphism 
$$D \in Hom_{\D}(E,E)$$ with
\[\D=\{ (k,n) \in \NM^{op} \times \NM\;\;|\;\;\; n < k\}\]
\begin{figure}[htb!]
 \includegraphics[width=9cm]{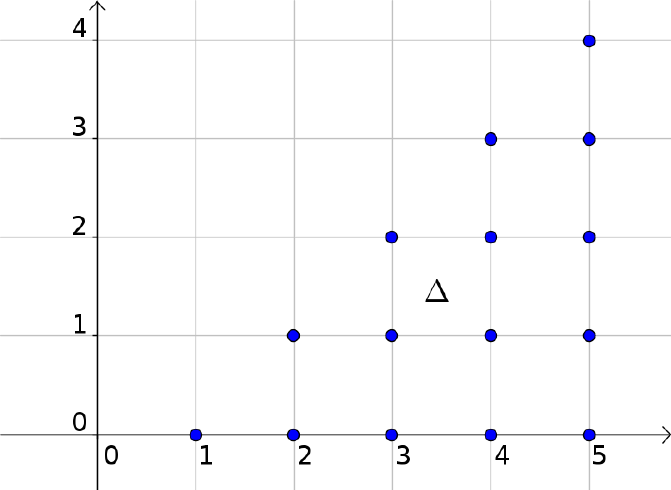}
\end{figure}

The situation becomes more interesting in the case of holomorphic functions.\\

The derivative gives a well-defined map on the algebra of convergent power series in one variable $z$:
 $$\CM\{ z \} \to \CM\{ z \}; \; f \mapsto f' .$$
The space $\CM\{ z \}$  arises as limit of the Kolmogorov space 
$$\Ot^c(D) \to D $$ 
with fibre
\[ \Ot^c(D)_t:=C(\overline{D_t})\cap \Ot(D_t) ,\]
where
$$D \to \RM_{>0} $$
is the unit disk. 

The derivative $ f \mapsto f'$ does not define a map
\[ \Ot^c(D)_t \to  \Ot^c(D)_{t},\ f \mapsto f'\]
and therefore it is not a morphism of Kolmogorov spaces.
However for all $s<t$, there {\em are} well defined maps:
$$\p_{st}: \Ot^c(D)_t \to  \Ot^c(D)_{s},\ f \mapsto f' $$
and these are compatible with the restriction maps. This can be interpreted 
as saying that the derivative is a partial morphism  $\p$ defined over the 
{\em open sub-diagonal}
\[  \D:=\{(s,t) \in ]0,S]\times ]0,S]\;|\; s < t\} .\]

The fact that this section $\p$ has no extension over the 
closed sub-diagonal $\overline{\Delta}$ is an expression of the 
'unbounded nature' of the derivative operator  
$$\p \in \Hom_\D(E,E).$$  
We claim that $\p$ is not bounded,
\[ \p \notin \Hom^b_{\D}(E,E) .\] 
To see this, one can look at the holomorphic function 
$$f(z)=(1-z)\log(1-z) $$ 
on the open disc $D_1$ of radius $1$. It is an element of $\Ot^c(D_1)=E_1$
and by restriction to smaller $t$ be obtain a bounded horizontal section of 
$E$ over $]0,1]$. Its derivative is
$$f'(z)=1+\log(1-z) \xrightarrow[z \mapsto 1]{} -\infty $$
which means that $\p(f)$ is an unbounded section of $E$ over $]0,1[$.

\section{The Cauchy-Nagumo inequality}

An unbounded operator like the derivative is of course more difficult 
to handle, but we will now see it defines a partial morphism {\em with a first order
pole along the diagonal}. The following fundamental fact shows what we need:
 
\begin{proposition}[{\em Cauchy-Nagumo inequalities}]
Let 
$$D \to \RM_{>0} $$
be the unit disk in $\CM$. Then the $k$-th derivative of $ f \in \Ot^c(D_t)$ satisfies the estimate:
 $$| f^{(k)} |_s\leq \frac{k!}{(t-s)^k}| f|_t $$
 for any $s<t$.
\end{proposition}
\begin{proof} The statement is an easy corollary of the Cauchy integral formula.
If $z \in D_s$ and $f \in \Ot^c(D_t)$, then one has
 $$f(z)=\frac{1}{2\pi i}\int_{\g_z} \frac{f(\xi)}{z-\xi} d\xi, $$
where $\g_z$ is the circle of radius $t-s$ centred at $z$. Differentiation
under the integral sign gives
\[f'(z)=\frac{1}{2 \pi i}\int_{\g_z} \frac{f(\xi)}{(z-\xi)^2} d\xi.\]
We parametrise $\g_z$ by:
$$\theta \mapsto \xi(\theta):=z+(t-s)e^{2\pi i \theta} $$
and thus
$$d\xi=2\pi i (t-s) e^{2\pi i \theta} d\theta$$
so that
$$ f'(z)=\frac{1}{t-s}\int_{0}^1 \frac{f(\xi(\theta))}{e^{2\pi i \theta}} d\theta.$$
and finally
$$| f'(z)| \leq \frac{1}{ t-s}\int_{0}^1 |f(\xi(\theta))| d\theta\leq \frac{1}{t-s}| f |_t ,\ \forall z \in D_s. $$
In other words:
$$| f'|_s \leq \frac{1}{t-s}| f |_t $$
for any $t>s$. For higher values of $k$ the proof is similar. 
\end{proof}

This result means that as we approach the boundary of the disc $D_t$, the value of $f'$ may increase, but not faster than the inverse of the distance to the boundary. In other words the derivative defines an almost complete morphism with a simple pole along the diagonal.
So if we rescale rescale the K2-space $\mathcal{H}om(E,E)$ with the function
$$\l(s)=(t-s), $$ 
then due to the Cauchy-Nagumo inequalities, 
the derivative morphism $\p$ is a bounded horizontal section
of $\mathcal{H}om(E,E)(\l)$ over $\D$:
\[ \p \in \G^{\infty}(\D,\mathcal{H}om(E,E)(\l)). \]

More generally, if we consider a differential operator of higher order
as an almost complete morphism of a Kolmogorov space, the above Cauchy-Nagumo 
inequality may be used to detect the {\em order} of the  differential operator 
in terms of the internal structure. 
\section{The division problem}
Another type of singularity appears when we consider division mappings. 
The simplest case is the division by $z$ of functions in a single variable.
If $f$ is a germ of a function holomorphic near $0$, we can always write
 $$f(z)=f(0)+zg(z) $$
The map which associates $g$ to $f$ is given by the $\CM$-linear map
$$\p:\CM\{ z \} \to \CM\{ z \}, f \mapsto \frac{f-f(0)}{z}.$$
We would like to understand this map in terms of Kolmogorov spaces.

So, we consider the unit disc $D=(D_s)$ over $B:=\RM_{>0}$ and the associated 
Kolmogorov space $E:=\Ot^c(D)$.
Clearly, for fixed $t$ we can consider the analogous map
\[\p_t^c: E_t \to E_t,\;\;f \mapsto \frac{f-f(0)}{z}\]
So unlike the case of the derivative, it is a complete morphism: the division maps are compatible with restrictions and we obtain a global horizontal section
\[\p^c \in \G^h(\overline{\D},\mathcal{H}om(E,E)))=\Hom_B(E,E)\]
But this map is again unbounded, i.e.
\[ \p^c \notin \G^{\infty}(\overline{\D},\mathcal{H}om(E,E)))=\Hom_B^b(E,E)\]
To see this, consider for instance sequence of linear polynomials
$$1+nz .$$
We consider these as a horizontal section of $E$:
$$\s_n:\RM_{>0} \to E,\ s \mapsto (s,1+nz).$$
Clearly one has
$$ \p^c(\s_n(s))=n$$
and
\[ |\s_n(s)|=\sup_{|z|\le s}|1+nz|=1+ns .\]
Hence, for $s\to 0$, we get
$$|\p^c(\s)|_s \geq \frac{|\p^c(\s(s))|}{|\s(s)|}=\frac{n}{1+ns} \sim n .$$
Thus the section $\p^c$ is unbounded.

Now rescale the K2-space $\mathcal{H}om(E,E)$ with the function
$$\l(s,t)=s. $$ 

Using the maximum principle, we get that
$$\left| \frac{f-f(0)}{z}\right|_t=t^{-1}|f-f(0)|_t$$
and therefore
$$\left| \frac{f-f(0)}{z}\right|_t \leq 2t^{-1}|f|_t .$$

The norm of this section is defined by
$$|u|^{(\l)}:=\sup_{f \in B_t} t |u(f)|_t \leq 2 $$
where $B_t$ denotes the unit ball in $E_t$.
This shows that after rescaling we get a bounded map.

 \section{Local operators}
We have seen the derivative defines a partial morphism with poles along $s=t$, whereas the division operator by $z$
has pole along $s=0$. If we now compose both operation, we get an operator:
$$\p:\CM\{ z \}  \to \CM\{ z \}, f \mapsto z^{-1}(f'(z)-f'(0)) $$
which lifts to a partial morphism with poles along $s=t$ and $s=0$, i.e.
the two sides of the cone $0 \le s \le t$.
 
This means that in the Kolmogorov space 
\[ E=\Ot^c(D) \to \RM_{>0}, \]
the map $\p$  defines an unbounded 
horizontal section 
\[ \p \in \Hom_\D(E,E)=\G^h(\D,\mathcal{H}om(E,E)).\] 
The norm of this section goes to infinity along the sides $s=0$ and $t=s$ of $\D$.  If we now rescale the K2-space 
$$\mathcal{H}om(E,E) \to B^{op} \times B $$ with the scaling function
 $$\l(t,s)=s(t-s), $$
the section $\p$ becomes bounded:
\[\p \in \G^{\infty}(\D,\mathcal{H}om(E,E)(\l)) .\]

Recall that any map 
 $$\p_{t,s}: E_t \to F_s,\ E=F=\Ot^c(D) $$
extends horizontally to the downset of $(t,s)$. So we 
get morphisms
 $$\p_{r,u}: E_u \to E_r,\;\;u \geq t,\;\;s \geq r.$$ 
This was the basis of our considerations on 
the external Hom \ref{D::external Hom-space}.  

The situation here is in some sense reversed, a map 
$$\p_{t,s}: E_t \to F_s$$ that arise from the derivative or division map 
has an {\em opposite behaviour}: it extend to the
part of the {\em upset} of $(t,s)$ below the diagonal. As a result, a single
$\p_{t,s}$ defines a horizontal section over the downset of the upset of $(s,t)$!\\
\vskip0.3cm
\begin{figure}[htb!]
 \includegraphics[width=0.9\linewidth]{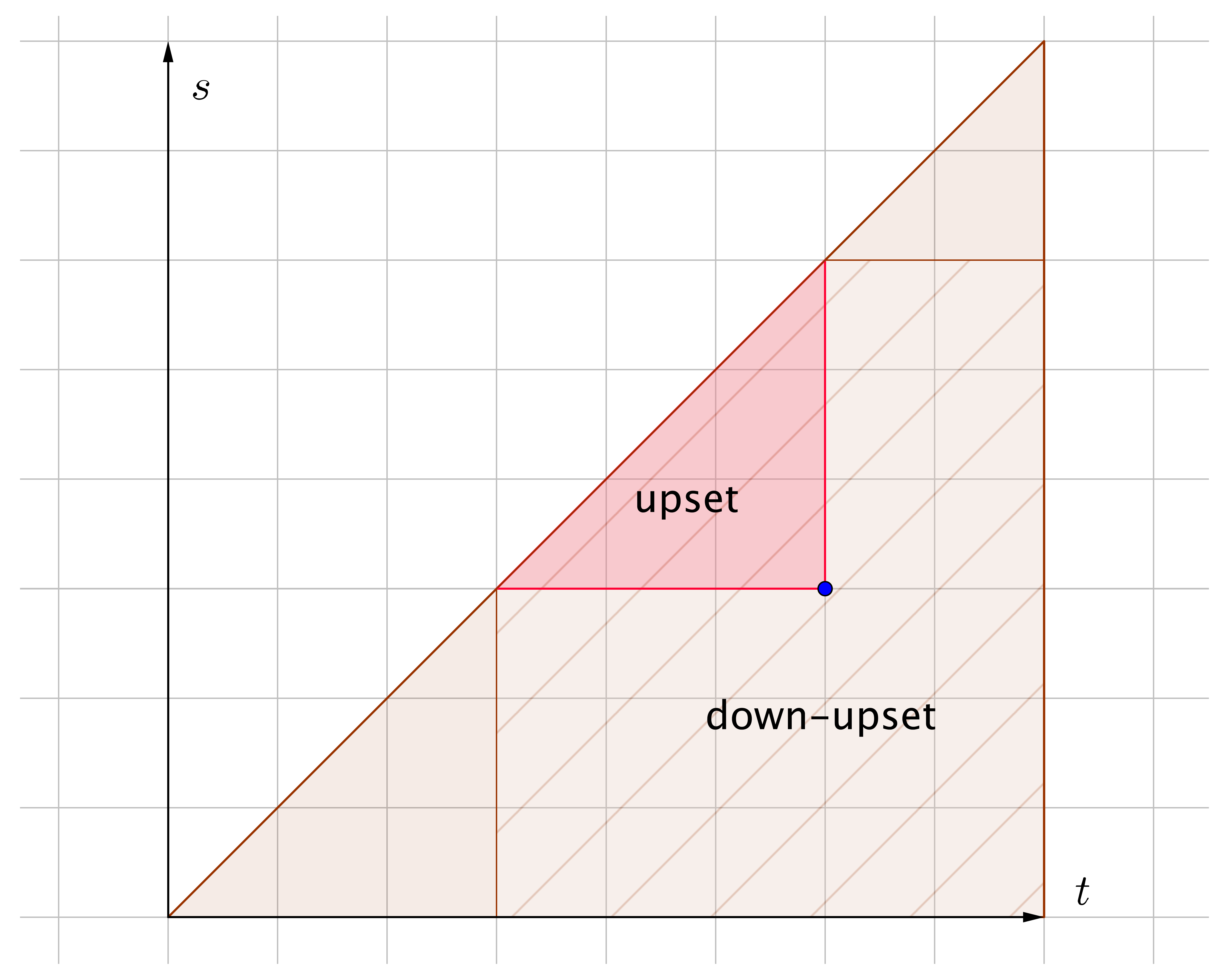}
\end{figure}
\vskip0.3cm

Let us turn this property into a concept. Given two $S$-Kolmogorov spaces
 $$E \to ]0,S],\ F \to ] 0,S], $$
we consider the space $\mathcal{H}om(E,F)(\l)$ with rescaling function
 $$\l(t,s)=s^k(t-s)^l, $$
where $k,l$ are arbitrary positive numbers.
For a point  $(t,s) \in \D$ we consider the Banach-space 
$$\Lt^{k,l}(E,F)_{t,s}:=\G^\infty(\D({t,s}),\mathcal{H}om(E,F)(\l)), $$
where $\D({t,s})$  is the upset of the point $(t,s)$. 

The norm of an element of $\Lt^{k,l}(E,F)_{t,s}$ is given by
\[ |u|:=\sup_{(t',s') \leq (t,s)} ((s')^k(t'-s')^l|u_{s't'}|) .\]

As the elements of $\Lt^{k,l}(E,F)_{t,s}$
are sections over $\D(t,s)$, if $(t',s') \in \D(t,s)$ then $\D(s',t') \subset \D(t,s)$, hence there is an obvious restriction mapping of norm $\le 1$:
$$\Lt^{k,l}(E,F)_{t,s} \to \Lt^{k,l}(E,F)_{t',s'} .$$
Thus we obtain in this way a $K2$-space but with opposite order on the base
$$\Lt^{k,l}(E,F) \to  \D^{op},$$
whose fibre above $(t,s)$ is $\Lt^{k,l}(E,F)_{t,s}$.
   
\begin{figure}[htb!]
 \includegraphics[width=0.8\linewidth]{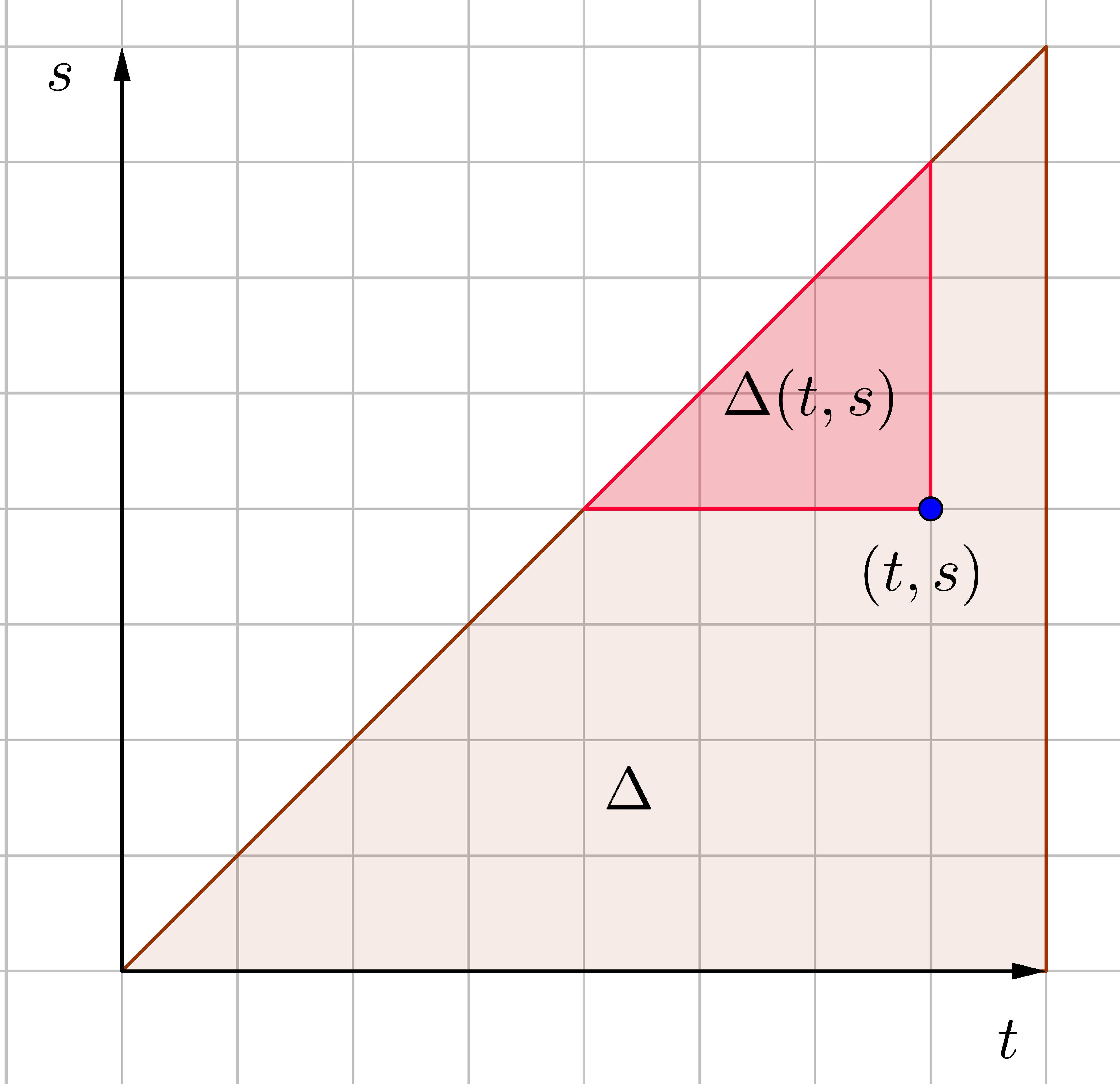}
\end{figure}

\begin{definition}\index{local operator}
Let $E, F $ be $S$-Kolmogorov spaces. Elements of the $K2$-space
$$\Lt^{k,l}(E,F) \to \D^{op} $$ will be called {\em local operators}.
\end{definition}

Frequently, the sections are non-singular along $s=0$, therefore we also use 
the notation 
$$\Lt^k(E,F):=\Lt^{0,k}(E,F)$$

\begin{example}
The  $k^{th}$ derivative 
\[E:=\Ot^c(D) \to \Ot^c(D),\;\;f \mapsto f^{(k)}\]
is an element of $\Lt^k(E,E)$. More generally, any differential operator $P$ of order $k$,
with holomorphic coefficients, defines an element of $\Lt^k(E,E)$.\\
\end{example}

\begin{proposition}
For $m \ge 0$ there are inclusions  
$$\Lt^{k,l}(E,F) \subset \Lt^{k+m,l+n}(E,F)$$ 
and if $i$ denotes the inclusion map one has
\[ |i(u)| \le S^{m+n}|u| .\]
\end{proposition}
\begin{proof} 
Fix $s<t$ and denote by $\|\cdot \|$ the operator norm in the Banach space 
$L(E_t,E_s)=Hom(E_t,F_s)$.
We can write 
$$\| u_{st} \| \leq \frac{|u|}{s^k(t-s)^l}=\frac{s^m(t-s)^n |u|}{s^{k+m}(t-s)^{l+n}}\le \frac{S^{m+n}|u|}{s^{k+m}(t-s)^{l+n}}. $$ 
\end{proof}

So the order of the pole along the sides of $\Delta$ gives a natural increasing filtration on the space of local operators:
\[ \Lt^{k,l}(E,F) \subset \Lt^{k+m,l+n}(E,F).\]
\section{Composition properties}
As has been remarked in chapter $8$, the upset triangle is an idempotent:
\[  \D(t,s) \star \D(t,s) =\D(t,s),\]
it is straightforward to compose almost complete morphisms: 
if $E,F,G$ are $S$-Kolmogorov spaces, we have a well 
defined bilinear composition map
\begin{align*}
\circ: \Hom_{\D(t,s)} (E,F)) \times \Hom_{\D(t,s)} (F,G)) &\to \Hom_{\D(t,s)}(E,G))\\
 (u,v)& \mapsto v \circ u \end{align*}
The sub-spaces of $(k, l)$-local operators
\[ \Lt^{k,l}(E,F)_{t,s} \subset Hom_{\D(t,s)} (E,F))\]
behave well under this composition:

\begin{proposition}
Let $E,F,G$ be $S$-Kolmogorov spaces.
If $u \in \Lt^{k_1,l_1}(E,F)$ and $v \in \Lt^{k_2,l_2}(F,G)$ then 
$$v\circ u \in \Lt^{k,l}(R,G)$$ with
$k=k_1+k_2,\ l=l_1+l_2 $. Moreover, one has the norm estimate: 
\[ |v \circ u| \le  \frac{l^l}{l_1^{l_1} l_2^{l_2}} |v||u|.\]
\end{proposition}
\begin{proof} 
Denote by $\| - \|$ the operator norm. For  $s < s' < t $ we have
\[ \|(v \circ u)_{st}\| \le \|v_{ss'}\|\,\|u_{s't}\|\]
 From the definition of locality we have:
\begin{align*}
\|v_{ss'}\| &\le \frac{|v|}{s^{k_2}(s'-s)^{l_2}},\\
\|u_{s't}\| &\le \frac{|u|}{(s')^{k_1}(s'-t)^{l_1}} < \frac{|u|}{s^{k_1}(s'-t)^{l_1}}.\end{align*}
We take the point $s'$ such that 
$$(s'-s)=\frac{l_2}{l}(t-s),\ (t-s')=\frac{l_1}{l}(t-s)$$
and  find the estimate
$$ \|v_{ss'}\|\,\|u_{s't}\| \leq \frac{l^l}{l_1^{l_1} l_2^{l_2}} \frac{|v| |u|}{s^k(t-s)^l}$$
thus
\[ |v \circ u| \le \frac{l^l}{l_1^{l_1} l_2^{l_2}} |v| |u|\]
\end{proof}
We remark that this estimate is slightly better then the one obtained by taking $s' =(s+t)/2$ the midpoint, 
from which one gets the simpler estimate $ |v \circ u| \le 2^{l} |v| |u| $.\\

The form of the above estimate leads to the idea to {\em recalibrate} the norm\index{calibrated norm} on $\Lt^{k,l}(E,F)$
and introduce
\begin{definition}
\[ ||u||:=\lambda_k |u|\]
where 
\[\lambda_k:=\left(\frac{e}{k}\right)^k\]
\end{definition}

This new norm now becomes {\em sub-multiplicative}:
\begin{corollary}
\[ ||v \circ u|| \le ||v||\, ||u||\]
\end{corollary}

For this to work one could replace $e$ in the above definition of $\lambda_k$
by any number $\alpha >0$, but we will see later that $\alpha=e$ is a particular 
convenient choice.  

\begin{corollary}\label{C::powerestimate} If $u_1,u_2,\ldots,u_n \in \Lt^1(E,E)$ then
\[ u_1 \circ u_2\circ \ldots \circ u_n \in \Lt^n(E,E)\] 
and we have the estimate:
\[ ||u_1 \circ u_2\circ \ldots \circ u_n|| \le ||u_1|| ||u_2||\ldots||u_n||, \]
or, in terms of the original norm:
\[ |u_1 \circ u_2 \circ \ldots \circ u_n| \le n^n |u_1| |u_2| \ldots |u_n| .\]
Especially for $u_1=u_2=\ldots=u_n=u$:
\[ ||u^n|| \le ||u||^n,\;\;\; |u^n| \le n^n |u|^n .\]
\end{corollary}

This is of course also easily proven directly by subdividing the interval $[s,t]$  in $n$ equal parts.\\

In reading these formulas one should be aware that for reasons of simplicity
we always write $|-|$ or $||-||$ for the norms, but which norm this in fact is, 
depends on the element to which we apply it. For example, 
if $u$ is 1-local, the  norm in the expression $|u|^n$ refers to the norm 
on $\Lt^1(E,E)$, whereas in $|u^n|$ we are using the norm on  
$\Lt^n(E,E)$~! 
 
\section{Partial differential operators and Huygens sets}
The Cauchy-Nagumo estimate links the norm of functions on discs to the size of these discs. 
In the proof we used the fact that the disc $D_t$ of radius $t$ is the union of the 
disc $D_s$ of radius $s<t$ and discs of radius $t-s$ centered at points $z \in D_s$. 
This is reminiscent of {\em Huygens principle} in geometrical optics and in fact the 
argument carries over to much greater generality.

\begin{definition} An $S$-Huygens open set\index{Huygens open set} $U$ in $\CM^n$ is a collection of 
strictly increasing relatively compact open sets 
$$U_s \subset \CM^n,\ s \in ]0,S]$$ such that for any $t>s$ and any $x \in U_s$ the polydisc 
$x+D_{t-s}$ is contained in $U_t$.
\end{definition}
\begin{figure}[htb!]
 \includegraphics[width=8cm]{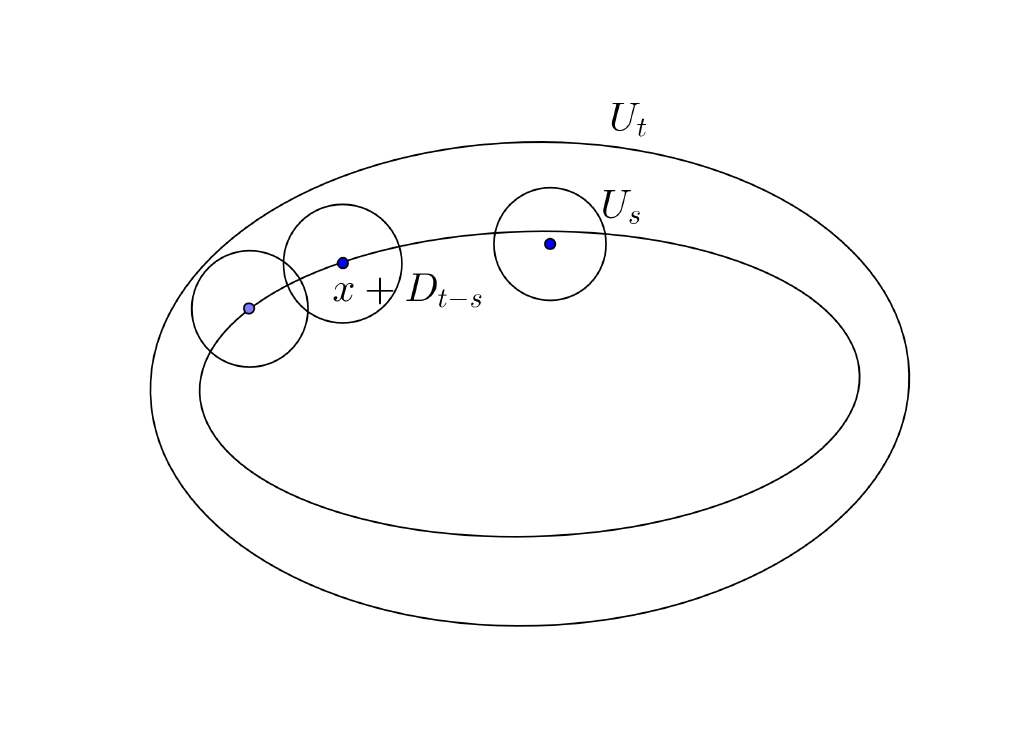}
\end{figure}

It is clear that there exist many $S$-Huygens sets. For example, beside polydiscs $D_s \subset \CM^n$,
the neighbourhoods of the real torus
$$U_s=\{ z \in (\CM^*)^n: 1-s < | z| < 1+s \} \subset (\CM^*)^n$$
form a Huygens sets.

For any $S$-Huygens set $U$ we have a corresponding $S$-Kolmogorov space $ \mathcal{O}^c(U) $. 
The Cauchy-Nagumo inequalities can be stated as follows:

\begin{proposition} If $U$ is an $S$-Huygens set then any differential operator $P$ of order $k$ defines
a $k$-local operator 
\[P \in \Lt^k (\mathcal{O}^c(U),\mathcal{O}^c(U)) .\] 
That is, there exist a constant $C >0$, such that 
\[ |Pf|_s \le \frac{C}{(t-s)^k} |f|_t .\]  
\end{proposition}
\section{The Cartan-Weierstrass theorem}
\label{S::Cartan-Weierstrass}
We have seen that the division operator with $z$ is local. This is a particular case of a general result proved by Cartan.
First let us recall what division means in an arbitrary number of variables.

We decompose the space $\CM^n$ as the product $\CM^{n-1} \times \CM$ with coordinates $x=(x_1,\dots, x_{n-1})$ and $y$.

An analytic series $g \in \CM\{ x,y \}$ of the form
$$g=y^d+a_1(x) y^{d-1}+\ldots+a_d(x) $$
with $a_i(0)=0$
is called a {\em Weierstrass polynomial} of degree $d$ .

The Weierstrass division theorem asserts that for any $f \in \CM\{ x,y \}$,
there exist unique $q,p \in \CM\{ x,y \}$ such that
$$f=q g+p.$$
where $p$ is a Weierstrass polynomial of degree less than $g$.\\
We will call the map
$$\dt_g:f \mapsto q  $$
the {\em division operator} and
$$  \rho_g:f \mapsto r  $$
the {\em remainder operator}. 
For instance, in the one variable case and $g:=y^d $, the division operator is
$$\dt_g:f \mapsto y^{-d}(f-\sum_{k=0}^{d-1}\frac{f^{(k)}(0)}{k!}y^k). $$

Assume now that the Weierstrass polynomial $g$ 
defines a holomorphic function on an open susbet $U \subset \CM^n$.  
As such, the function $g$ is a polynomial in $y$ depending on $x$ as parameters. 
At $x=0$, it reduces to $y^d$ and therefore has a zero of multiplicity $d$ at the origin.

We consider open polydiscs $\D_{\rho} \subset U$ and write the polyradius 
$\rho$ as
\[ \rho=(\rho',r),\;\;\rho'=(\rho_1,\rho_2,\ldots\rho_{n-1})\]
so that
\[ \Delta_{\rho}=\Delta_{\rho'} \times \Delta_r .\]
The linear projection
$$\Delta_{\rho} \to \Delta_{\rho'} $$
fibres the polydisc $\Delta_{\rho}$ by open discs. We say that the polydisc $\Delta_{\rho} \subset U$ is {\em adapted}\index{adapted polydisc} 
to $g$ if the hypersurface $V(g)$ does not intersect the boundary of the fibres  
$$V(g) \cap \left(\Delta_{\rho'} \times \partial \Delta_r \right) =\emptyset$$
 
 \begin{figure}[htb!]
 \includegraphics[width=0.9\linewidth]{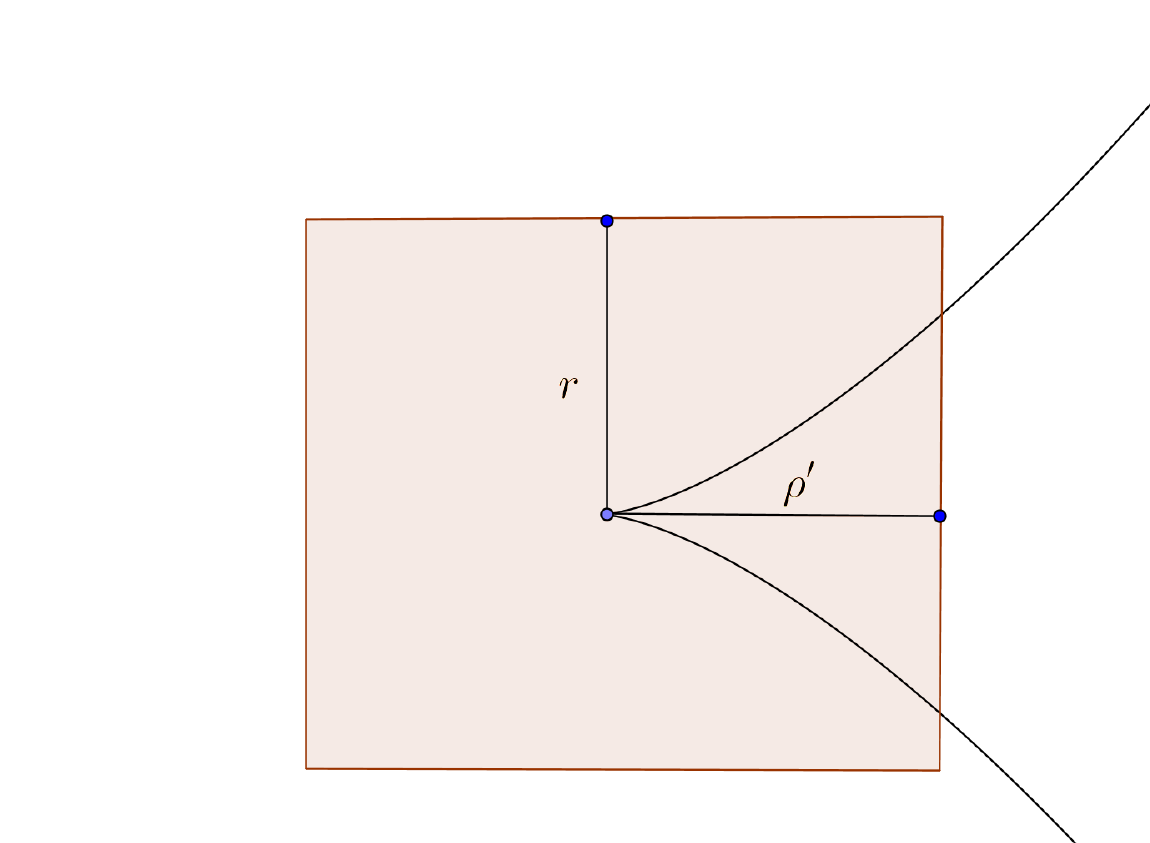}
\end{figure}

The existence of such a fundamental system of $0$-neighbourhood is elementary: taking the modulus of the Weierstrass polynomial
$$(x,y) \mapsto | g(x,y)|  $$
gives a continuous. For $x=0$, it does not vanish on the circles 
$$x=0,\ |y|=r .$$
Thus taking $\rho'$ sufficiently  guarantees that $V(g)$ does not intersect $\Delta_{\rho'} \times \partial \Delta_r $.
Moreover as the sum $N_x$   of the roots of the polynomial $g(x,-)$ counted with multiplicities is obtained
by integrating the logarithmitic derivative along $\partial \Delta_r$:
$$N_x=\frac{1}{2\sqrt{-1}\pi}\int_{\partial \Delta_r} \frac{\d_y g(x,y)}{g(x,y)}dy ,$$
it is a continuous function which therefore remains constant equal to $d$ as $x$ varies inside $\D_{\rho'}$. 
 
We may now formulate the Cartan-Weierstrass theorem: 
\begin{theorem} 
\label{T::hypersurface}
Consider  polydiscs $U_t:=\D_{\rho'(t)} \times \D_t$ adapted to a Weierstrass polynomial $g \in E:=\Ot^c(U)$ and defining
a fundamental system of $0$-neighbourhoods.
The division operator  $\dt_g \in \Hom(E,E)$ induces a $(d,0)$-local operator i.e.  $\dt_g \in \Lt^{d,0}(E,E)$
and more precisely:
$$ | \dt(g)|_t\leq \frac{1}{\e(t)} |f|_t $$
where 
$$\e(t):=\min\{|g(x,y)|:(x,y) \in \D_{\rho'} \times \d \D_t\} .$$
\end{theorem}
The second part is more precise since the equality
$$g(0,y)=y^d $$
and the fact that the polycylinders are adapted to $g$
imply there existence of a constant $C>0$ such that 
$$\e(t) \geq Ct^d$$
\begin{proof}
We assert that the function
\[ q(x,y) =\frac{1}{2\pi i}\int_{|\xi|=t} \frac{f(x,\xi)}{g(x,\xi)}\frac{d\xi}{(\xi-y)}\]
equals $\dt_g(f)$.

Define
\[ p:=f-q g.\]
Using the Cauchy formula, we get that:
\[p=\frac{1}{2\pi i}\int_{|\xi|=t} \frac{f(x,\xi)}{g(x,\xi)}\frac{g(x,\xi)-g(x,y)}{\xi-y} d\xi\]
As $g$ is a polynomial in $y$, the function
\[ \frac{g(x,\xi)-g(x,y)}{\xi-y} \]
is a polynomial of degree less than $d$ in the $y$-variable and therefore so is $p$. This proves
the assertion
$$q=\dt_g(f) .$$
From the first integral formula, we deduce the estimate:
$$| q |_t \leq \frac{1}{\e}\sup_{|y|=t}|f(x,y)| \leq \frac{1}{\e}|f|_t   $$
as well as
$$|p|_t \leq \frac{|g|_t}{\e}|f|_t. $$
This proves the theorem.
\end{proof}
 
Cartan used the theorem in order to prove, the following fundamental result:\\
\begin{theorem} Consider the ring $R=\Ot_{\CM^d,0}$
 and a map of free $R$-modules:
 $$A:R^m \to R^n. $$
There exists a fundamental system $\D$ of neighbourhoods of the origin, a number $k$ 
and a map $$B:R^n \to R^m $$
such that $B \in \Lt^{k,0}(\left(\Ot^c(\D)\right)^n, \left(\Ot^c(\D)\right)^m)$
 satisfies
 $$ABA=A .$$ 
 \end{theorem}
 A proof is given in the Appendix.
\section{Further properties} 

{\bf \em Locality and order filtration:} Just as derivations decrease the 
order of polynomials, local morphisms share a particular behaviour with 
respect to order filtrations of Kolmogorov spaces.
Consider for instance our $S$-Kolmogorov space $E=\Ot^c(D)$, where $D=(D_s)$
is the disc over $]0,S]$. An element $f \in E_t^{(l)}$ is a function 
holomorphic in the disc of radius 
$t$ with vanishing first $l$-derivatives at the origin.
For $k \le l$, the $k$-the derivative operator 
\[ u(f)=f^{(k)} \]
maps $f$ to a function, defined on $D_s$, $s<t$, whose first  $l-k$
derivatives vanish at the origin, so it belongs to $E^{(l-k)}$. 
This example is a particular case of a more general statement.\\

It is convenient to introduce the $S$-Kolmogorov space
\[ L^{k,l}(E,F)=p_*\Lt^{k,l}(E,F)\]
where $p:\D \to ]0,S],\;\;(t,s) \to t$ is the projection to the first 
coordinate, so we have
\[ L^{k,l}(E,F)_t=\Lt^{k,l}(E,F)_{t,0}\]
and 
\[ L^{k}(E,F)_t=\Lt^{0,k}(E,F)_{t,0}\]
We call these the $K1$-Kolmogorov spaces of local operators.
\index{$K1$-spaces of local operators}

\begin{proposition}
A $k$-local map $u \in L^k(E,F)$ of Kolmogorov spaces maps the term $E^{(l)}$ of the order filtration to $F^{(l-k)}$ for any $l \geq k$.
\end{proposition}
\begin{proof}
For $x \in E^{(l)}_t$, we have an inequality of the form $|e_{ts} x| \le R s^l$ for all $s \le t$.
we apply the estimate 
$$| u_{ts}(x) | \leq \frac{C}{(t-s)^k} | x | $$
with $t=2s$ and get that
$$| u_{ts}(x) | \leq \frac{C}{s^k} | x | \leq \frac{C R}{s^k} t^l=\frac{C R}{2^l} s^{l-k}. $$
This proves the proposition.
\end{proof}

{\bf \em Bi-Lipschitz nature of locality:} As in the definition one uses $t-s$, it seems that the notion of locality depends critically on the linear structure of the interval 
$]0,S]$. It turns out that this is not really the case.

If $u: E \to F$ is a partial morphism between $S$-Kolmogorov spaces, 
and $\phi:]0,T] \to ]0,S]$ a order preserving map, we obtain a linear map
\[ \phi^*u: \phi^*E \to \phi^* F;\;(\phi^*u)_{t' t}:=u_{\phi(t'),\phi(t)}\]

\begin{proposition}
Assume that $\phi$ has an inverse with Lipschitz constant $L$.
Then  $$u \in L^k(E,F) \implies \phi^*u \in L^k(\phi^*E,\phi^*F)$$
and 
\[ |\phi^*u|_t \le |u|_t L^k\]
\end{proposition}
\begin{proof}
If $s,t\in T$ with $s < t$ we get from the $k$-locality of $u$:
\[ |u (x)|_{\phi(s)} \le \frac{C}{(\phi(s)-\phi(t))^k} |x|_{\phi(t)}\]
Furthermore, from the Lipschitz constant of the inverse of $\phi$ we get
\[ (s-t) \le L (\phi(s)-\phi(t))\]
So \[ |u (x)|_{\phi(s)} \le \frac{C L^k}{(s-t)^k} |x|_{\phi(t)}\]
which means
\[|\phi^* u| \le CL^k\]
\end{proof}

\begin{corollary}
If $\phi$ is bi-Lipschitz, then the canonical map   
\[\phi^*L^k(E,F) \to L^k(\phi^*E,\phi^*F)\]
is $0$-local isomorphism.
\end{corollary}

\chapter{Local equivalence}
\section{Introduction}

Associated to a relatively compact open subset $U \to B$ of $\CM^n$ over 
a base $B$, we considered various Kolmogorov spaces  such as $\Ot^h(U)$, 
$\Ot^b(U)$, and $\Ot^c(U)$.

In all these cases, differential operators define almost complete morphisms 
\[P^a: \Ot^a(U) \to \Ot^a(U),\;\;\;f \mapsto Pf ,\;\;\;a \in \{b,c,h\}.\] 
Moreover, we proved that if $U$ is a Huygens set, then the partial differential operator $P^c$ is local. Now a natural question arises: are the corresponding
operators $P^b$ and $P^h$ also local?\\ 
To answer this question, one would like to compare these Kolmogorov spaces. 
This will lead us to the notion of {\em local equivalence}. 
If Kolmogorov spaces happen to be in correspondence via local maps, then 
automatically any local operator for one Kolmogorov space will be local for 
the other one. This means that in applications, in order to prove that a 
certain map is local,  we can choose the most adequate Kolmogorov space 
structure. This flexibility is very useful, as we do not have to fix a norm 
once and for all.  We can play with different norms at the same time and 
change structure sheaves if needed, which turns out to be a powerful tool. 

\section{Local equivalence of structure sheaves $\Ot^b$ and $\Ot^c$}
If $U \to B $ is a relatively compact open subset of $\CM^n$ over a base $B$, 
then, as any continuous function over a compact set is bounded, there is
a natural complete morphism
$$ I:\Ot^c(U) \to \Ot^b(U). $$

Conversely, a function $f \in \Ot^b(U)_t $ is by definition holomorphic in $U_t$. Therefore given any $s<t$, its restriction to $\overline{U}_s$ is continuous and thus belongs to $\Ot^c(U)_s$. 
This shows that the natural maps
$$ \Ot^b(U)_t \to \Ot^c(U)_s,\ f \mapsto f_{\mid U_s},\;\;\;s < t, $$
define an almost complete morphism
$$J \in \Hom_\D(\Ot^b(U), \Ot^c(U)). $$
However there are no natural maps $\Ot^b(U)_t \to \Ot^c(U)_t$ and the 
almost complete morphism $J$ does not extend to a complete morphism over 
$\overline{\D}$. 

The morphism $J$ is nevertheless $0$-local since
$$ | J(f) |_s=| f |_s \leq |f|_t$$
and it is pseudo-inverse to $I$:
\[ I \circ J =\iota_{\D},\;\;\;J \circ I =\iota_{\D} . \]

\section{Local equivalence of structure sheaves $\Ot^h$ and $\Ot^c$} 
As any continuous function over a compact subset of $\CM^n$ is 
integrable, there is also a natural complete morphism
$$ I:\Ot^c(U) \to \Ot^h(U). $$

Conversely, a function $f \in \Ot^h(U)_t $
is by definition holomorphic in $U_t$. Therefore given any $s<t$, its restriction to $\overline{U}_s$ is continuous and thus belongs to $\Ot^c(U_s)$. 
This shows that the natural maps
$$ \Ot^h(U)_t \to \Ot^c(U)_s,\ f \mapsto f_{\mid U_s},\;\;\;s < t $$
define an  almost complete morphism
$$J \in \Hom_\D(\Ot^h(U), \Ot^c(U)). $$
which is pseudo-inverse to $I$:
\[ I \circ J =\iota_{\D},\;\;\;J \circ I =\iota_{\D} .
\] 
Again there are no natural maps $\Ot^h(U)_t \to \Ot^c(U)_t$ and this 
almost complete morphism $J$ does not extend to a complete morphism over 
$\overline{\D}$. We will show that it is in fact local over Huygens sets. 
 

\begin{proposition}
\label{P::foncteur} If  $U \to B=\RM_{>0}$ is a relatively compact Huygens subset of $\CM^n$, then the  morphism $J$ is $n$-local.
\end{proposition}
\begin{proof}
Let $s<t$ and consider $f \in \Ot^h(U)_t$. Then the Taylor expansion 
of $f$ at a point $w \in U_s$ reads:
$$f(z)=\sum_{J \in \NM^n} a_J (z-w)^J,\ a_J \in \CM, $$
where we use multi-index notation. As $U$ is a Huygens set, the polydisc 
$D_w$ centred at $w$ with radius $\s=t-s$ is contained in $U_t$. We then have
$$\int_{D_w} | f(z)|^2 dV=\sum_{J \in \NM^n} C(J) |a_J|^2 \s^{2|J|+2n},\;\;C(I)=\prod_{k=1}^n \frac{\pi}{j_k+1} .$$
So we obtain
\[C(0) |a_0|^2 \s^{2n} \le \int_{D_w} | f(z)|^2 dV \leq  \int_{U_t} | f(z)|^2 dV=| f |_t^2 . \]
This shows that
$$ |f(w) |=|a_0| \leq  \frac{c}{\s^n}\left( \int_{\D_w} | f(z)|^2 dV \right)^{1/2} \leq   \frac{c}{(t-s)^n} | f |_t$$
for any $w \in U_s$ and $c:=\sqrt{\frac{1}{C(0)}}$.
Thus we see that
$$| Jf |_s \leq   \frac{c}{(t-s)^n} | f |_t.$$
This  proves the proposition.
\end{proof}


As a corollary we see that differential operators indeed are also local on $\Ot^h(U)$: we can 
write the operator
\[  \Pt:\Ot^h(U) \to \Ot^h(U)\]
as composition of local operators $ J \circ P \circ I$.\\
 
The above results give rise to the following concept:

\begin{definition}\index{local equivalence}
Two Kolmogorov spaces $E$ and $F$ are called {\em local equivalent} if there
are local morphisms
\[I :E \to F,\;\;\;J:F \to E\]
such that
\[ J \circ I=\iota_{\D},\;\;\;I \circ J=\iota_{\D} .\]
\end{definition}
 
So on Huygens sets $\Ot^c(U)$ and $\Ot^h(U)$ are local equivalent Kolmogorov 
spaces. This chain of ideas can be pushed much further. If $\Ft$ is a coherent 
sheaf on $\CM^n$, and $U$ an $S$-Huygens set which is $\Ft$-open. There there are Kolmogorov spaces $ \Ft^c(U)$ and $\Ft^h(U)$.
There are similar local maps $I$ and $J$ which shows that these Kolmogorov spaces are local equivalent. Their construction is natural, compatible with restriction and morphism of coherent sheaves.
So one may say that $\Ot^b$, $\Ot^c$ and $\Ot^h$ determine, via the maps $I$ and $J$ 
local equivalent {\em functors} on Huygens sets
$$\Ft \mapsto \Ft^b,\ \Ft \mapsto \Ft^c\;\;\;\textup{and}\;\;\; \Ft \mapsto \Ft^h.$$
 
\section{Local equivalence of $\ell^p(\l)$-spaces}
We study the issue of local equivalence for Kolmogorov-sequence spaces introduced
above. The estimate
$$\sup_{i}|x_i| \leq \left( \sum_{i \geq 0} |x_i|^p \right)^{1/p} $$
shows that the injection
$$I:\ell^p(\l) \to \ell^\infty(\l) $$
is an ordinary morphism of Kolmogorov spaces. On the other hand, for any $s<t$ the inclusion
$$ \ell^\infty(\l)_t \to \ell^p(\l)_s $$
induces an almost complete morphism of Kolmogorov spaces.\\

Given two sequences of increasing positive functions $\l,\mu$ define condition $(\L_p)$\index{condition $(\L_p)$} by:
 $$ \L_p=\L_p(\l,\mu):\ \  \exists C,\a>0,   \sum_{i \geq 0} \frac{\mu_i(s)^p}{\l_i(t)^p} \leq \frac{C}{(t-s)^\a}.$$


\begin{proposition}
If the condition $(\L_p(\l, \mu))$ holds with constants $C$ and $\alpha$, 
then there is an almost complete morphism
 $$J:\ell^\infty(\l) \to \ell^p(\mu) $$
induced by the inclusions
 $$\ell^\infty(\l)_t \subset \ell^p(\mu)_s,\ t>s $$
 which is $\a/p$-local with norm at most $C^{1/p}$.
\end{proposition}
\begin{proof}
Take
$$x=(x_0,x_1,\dots,x_i,\dots) \in \ell^\infty(\l)_t $$
then
\begin{align*}
| J(x)|_s^p &=\sum_{i \geq 0}|x_i|^p \mu_i(s)^p \\
 &= \sum_{i \geq 0}|x_i|^p\l_i(t)^p \frac{\mu_i(s)^p}{\l_i(t)^p}
 &\leq |x|_t^p \sum_{i \geq 0} \frac{\mu_i(s)^p}{\l_i(t)^p}
 & \leq \frac{C}{(t-s)^\a} |x|_t^p.
\end{align*}
\end{proof}

As an example, consider the weight sequences $\l=\mu=(s^i)$ on 
any interval $]0,S]$. These satisfies the condition $(\L_p((s^i),(s^i))$. 
Indeed: 
\begin{align*}
  \sum_{i \geq 0} \frac{\l_i(s)^p}{\l_i(t)^p}&= \frac{1}{1-(s/t)p}\\
  &=\frac{t^p}{t^p-s^p}\\
  &= \frac{t^p}{(t-s)(t^{p-1}+\dots+s^{p-1})}\\
  &=\frac{t}{(t-s)(1+\dots)} \leq \frac{S}{(t-s)} .
\end{align*}

Therefore the $S$-Kolmogorov sequence space $\ell^p((s^i))$ is 
local equivalent to $\ell^{\infty}((s^i))$, and hence all these 
spaces are local equivalent to each other.
 
\section{Local convolution operators}

Using local equivalences, one may sometimes deduce locality for certain maps
and even sometimes prove locality simultaneously for different Kolmogorov
spaces. As an example relevant for KAM theory, we study the case of a 
{\em convolution operator}\index{convolution operator} which is 
caracteristic of this type of simultaneous proof.

Consider the unit polycylinder $D \to \RM_{>0}$ in $\CM^n$. As an analytic 
function is determined by its power series expansion in one point, the maps
\[\Ot^a(D)_t \to \CM\{z_1,z_2,\ldots,z_n\} \subset \CM[[z_1,z_2,\ldots,z_n]], \;\;a \in \{b,c,h\}\]
that associate to a function its power series expansion at $0$, is injective. 
So we can think of the space $\Ot^a(D)_t$ as a sub-space of the ring of formal
power series. 

Recall that for two formal power series 
\[f=\sum_{I \in \NM^n} a_I z^I,\;\;\; g=\sum_{I \in \NM^n} b_I z^I \in \CM[[z_1,z_2,\ldots,z_n]] ,\]
the {\em Hadamard-product}\index{Hadamard product} or {\em convolution}\index{convolution} $\star$ is defined by taking the co\-ef\-fi\-cient\-wise product:
\[ f \star g=\sum_{I \in \NM^n} a_I b_I z^I .\]

If the convolution with $f \in \CM[[z_1,z_2,\ldots,z_n]]$ preserves the sub-space $\Ot^a(D)$, it defines a {\em convolution operator}\index{convolution operator}
\[ f \star: \Ot^a(D) \to  \Ot^a(D),\ g \mapsto f \star g \]

\begin{proposition}
\label{P::Diophante}
Consider a formal power series
$$f=\sum_{I \in \NM^n} a_I z^I \in \CM[[z_1,z_2,\ldots,z_n]].$$
If there exists an $k \in \NM$ such that $a_I=O(| I |^k)$
then the map 
\[ \Ot^h(D) \to  \Ot^h(D),\;\;\;\ g \mapsto f \star g\]
is local.
\end{proposition}
\begin{proof}
Recall that the monomials $z^I$ form an orthogonal basis in $\Ot^h(D_s)$ and
\[ |z^I|_s^2=C(I) s^{2|I|+2n},\;\;\; C(I)=\prod_{k=1}^n \frac{\pi}{|I_k|+1}\]
By Cauchy-Nagumo identities the differential operators
$$\del^c_{st}:\Ot^c(D)_t \to \Ot^c(D)_s,\;\;\; g \mapsto \sum_{i=1}^n z_i\d_{z_i}g$$
define a local map. As $\Ot^c(D)$ and $\Ot^h(D)$ are local equivalent, the 
corresponding operator
$$\del^h:\Ot^h(D) \to \Ot^h(D), g \mapsto \sum_{i=1}^n z_i\d_{z_i}g$$
is also local.

If $g=\sum_{I \in \NM^n} b_I z^I  \in \Ot^h(D_s)$ we find
$$ | \left(\del^h\right)^k g|_s^2=\sum_{I \in \NM^n}|I|^k | b_I|^2 C(I)s^{2|I|+2}$$
with 
$$ |I|=\sum_{k=1}^n i_k,\ I=(i_1,\dots,i_n).$$
The assumption on the $a_I$'s shows that there exists a constant $K$ such that:
$$  |f \star| \leq K | \left(\del^h\right)^k| .$$
As $\del^h$ is local this shows that $f \star $ is also $\Ot^h$-local.

\end{proof}

Note that if $f \star$ is $\Ot^h$-local, then by local equivalence, the 
corresponding almost complete maps
 \begin{align*}
                \Ot^b(D) \to  \Ot^b(D),&\ g \mapsto f \star g\\
                 \Ot^c(D) \to  \Ot^c(D),&\ g \mapsto f \star g
              \end{align*}
are also local.              

\chapter{Operator calculus in Kolmogorov spaces}

Given a bounded linear endomorphism $u:E \to E$ of a Banach space and a 
convergent power series $f \in \CM\{ z\}$, we may define 
\[f(u): E \to E, \] 
provided that the norm of $u$ is smaller than $R(f)$, the convergence radius 
of the power series $f$. Moreover, if $f \in \RM_+\{ z\}$, 
then one has the estimate
$$\| f(u) \| \leq f(\| u \|) $$
for the operator norms. We want to generalise these fundamental facts to the 
more general context of Kolmogorov spaces.
\section{The problem}
Most of the interesting linear operators are not bounded and the 
theory of  Banach space does not allow to take the exponentials of them. 
As a simple example, let us look again at the derivative: 
\[ \frac{d}{dx} : f \mapsto f'\] 
As the exponential series is
\[ e^{\frac{d}{dx}}=1+\frac{d}{dx}+\frac{1}{2!}\frac{d^2}{dx^2}+\frac{1}{3!}\frac{d^3}{dx^3}+\ldots\]
we see that it can not be defined in a sensible way on the Banach space 
$C^k([-1,1],\RM)$, as we have only $k$ derivatives. It consists of unbounded operators and it is relatively unclear how to define the infinite sum of them.

In the Fr\'echet space $C^\infty([-1,1],\RM)$, the terms in the esponential
series are well-defined, but the situation is in a sense even worse.   
The bounded subsets 
$$B_{\e}=\{f: \sup_{j \le i}  \| f^{(j)}\| < \e_i  \}  $$
indexed by sequences $\e=(\e_0,\dots,\e_n,\dots)$ form a fundamental system of bounded sets. (Here $\| \cdot \|$ denotes the supremum norm over the interval $[-1,1]$). In particular, the derivative just shifts the sequence $\e$ and is therefore a bounded map. But its exponential, which should be equal to the
transformation 
$$f(x) \mapsto f(x+1) $$
is ill-defined: the image of $1/(2x-3)$ would be $1/(2x-1)$ which is of course not $C^\infty$ since it is not even defined at $x=1/2$. So in some sense it is again unbounded, although all terms of the defining series are bounded 
operators. 

In fact in a Fr\'echet space, a simple map like
$$exp:C^0(]0,1],\RM) \to C^0(]0,1],\RM), f \mapsto e^f $$
is again pathological.  The directional derivative at $f$ of  the function $exp$ is the multiplication by $e^f$, it is therefore an isomorphism.
In the Banach space $C^0([0,1],\RM)$, one would conclude, by the
implicit function theorem, that the map is locally surjective. In the Fr\'echet space $C^0(]0,1],\RM)$,
the situation is very different: the image of $exp$ does not contain any open subset!

This is due to fact that the topology on $C^0(]0,1],\RM)$  is defined by the fundamental basis of $0$-neighbourhoods
$$V_{\e,\dt}=\{f: \sup_{x \in [\dt,1]} | f(x) | < \e  \}. $$ 
In such a neighbourhood, we have no control
of the function in the interval $]0,\dt[$. In particular, it contains functions with negative
values inside $]0,\dt[$.
\vskip0.3cm
\begin{figure}[htb!]
 \includegraphics[width=7.5cm]{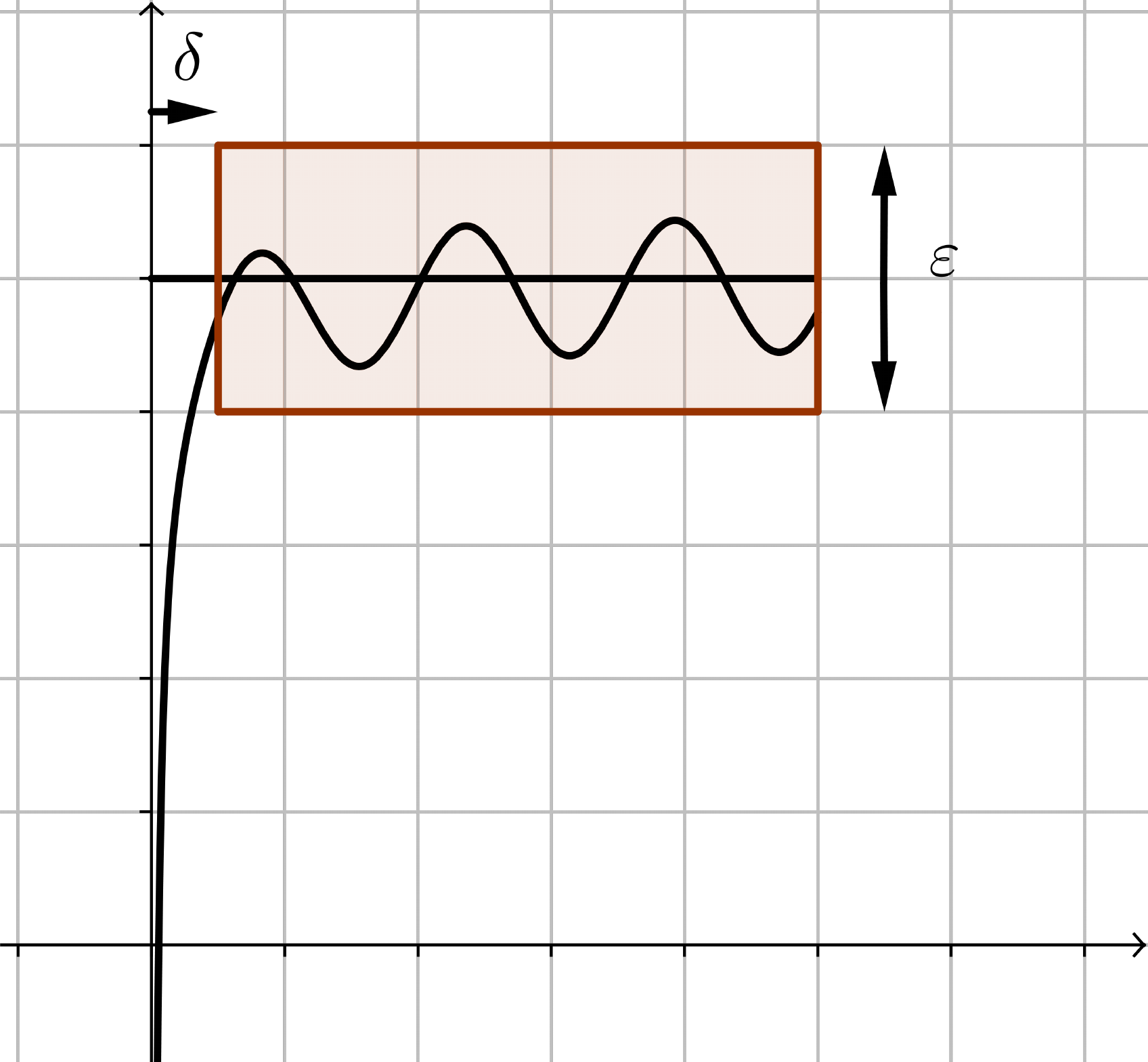}
\end{figure}
\vskip0.3cm

 As the exponential $e^f$ is a positive function,
this means that, although  the derivative is an isomorphism,  the image does not contain any open subset. We will now see that the situation for Kolmogorov spaces is much simpler. For instance, the exponential of a 
1-local map is, in general, well-defined and this provides a direct generalisation of the flow of a vector field.\\  
 
\section{The Borel map}  
As mentioned above, if $f$ is a convergent power series and $r <R(f)$,
the radius of convergence of $f$, then for any Banach space $E$, there 
is a map
$$rB_E \to Hom(E,E),\ u \mapsto f(u) ,$$
where $rB_E \subset E$ denotes the ball of radius $r$ in $E$.\\

For Kolmogorov spaces there is similar map, which however is 
richer from a geometrical viewpoint.\\

Recall that the {\em Borel transform}\index{Borel transform} of a formal power series is defined by
$$B:\CM[[z]] \to \CM[[z]],\ \sum_{n \geq 0} a_n z^n \mapsto \sum_{n \geq 0} \frac{a_n}{n!} z^n. $$
If $f$ is a convergent power series, $f \in \CM\{z\}$, then $Bf$ is not only convergent, but 
defines even an entire function on the complex plane. For example for the series
\[f(z)=\frac{1}{1-z}=1+z+\ldots+z^n+\ldots\]
has
\[Bf(z)=e^z=1+\frac{1}{1}z+\frac{1}{2!}z^2+\ldots+\frac{1}{n!}z^n+\ldots \]
as its Borel transform. As the Borel transform improves convergence of a series drastically,
it is very popular in the theory of resummation of divergent series, but this issue is not of 
importance here.\\

There is a simple result in the theory of Kolmogorov spaces that allows to define for any $1$-local operator $u \in \Lt^1(E,E)$ a morphism $\Bt f(u) \in {\mathcal H}om(E,E)$ as the limit of an infinite sequence. 
Via this map the upset of a point is mapped to a definition domain $A(u)$ and therefore to $u$ is associated also a partial morphism 
$$Bf(u) \in \Hom_{A(u)}(E,E). $$
%

\begin{definition}
\label{D::tetrahedron}
The tetrahedron\footnote{In some example, it is rather a cone but we will call it nevertheless
a tetrahedron.} of radius $R$\index{tetrahedron $\Xt(R)$} in a Kolmogorov space
$\Et \to \D$, denoted $\Xt_\Et(R)$, is defined as
\[ \Xt_\Et(R):=\{(t,s,u) \in \Et\;|\;\; ||u|| < R (t-s) \}\]
\end{definition}
So the set $\Xt_\Et(R)$ is just the preimage of a three-dimensional tetrahedron
$$\{ (t,s,x) \in \D \times \RM_{\ge 0}: x <R(t-s) \} $$
under the norm map 
\[ \nu: \Et \to \D \times \RM_{\ge 0} ,\;\;(t,s,u) \mapsto (t,s,||u||).\]
Note that the tetrahedron $\Xt_\Et(R)$ really is the {\em ball of radius $R$} centred at the origin for the norm rescaled by the factor 
$$\l(t,s)=(t-s)^{-1}. $$
We can consider curvilinear tetrahedra with arbitrary degrees along the sides of the three-dimensional prisma, but these will play no role in the sequel. 
Therefore we restrict our consideration here to this case.

Like for the unit ball in Banach spaces, we omit the subscript when no confusion is possible and moreover,
when $R=1$ we write simply $\Xt$ instead of $\Xt(1)$.
\vskip0.3cm 
\begin{figure}[h!]
\includegraphics[width=1\linewidth]{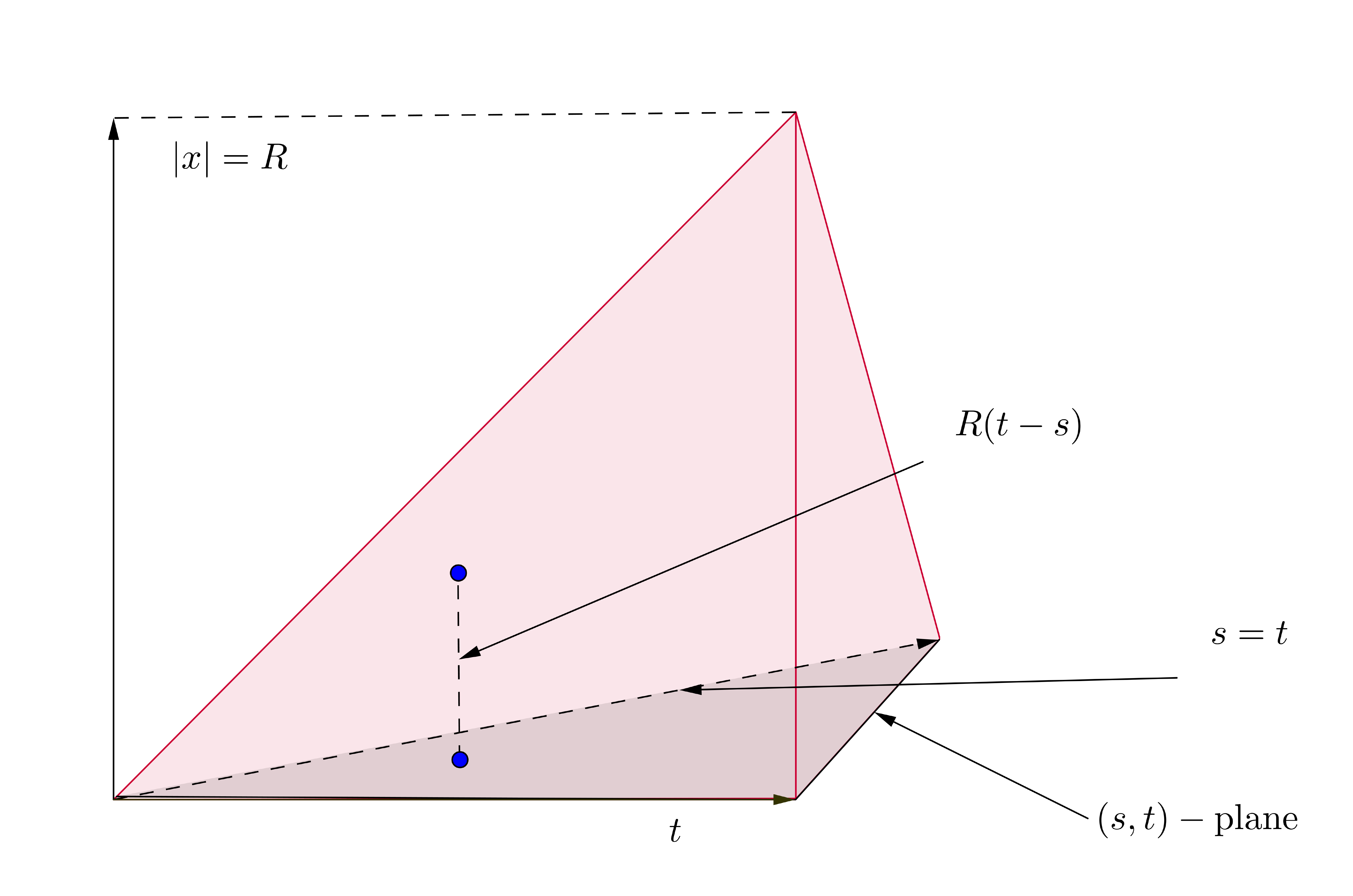}
\end{figure}
\vskip0.3cm
Given a power series
$$ f=\sum_{n \geq 0} a_n z^n$$
we use the notation
$$ |f|=\sum_{n \geq 0} |a_n| z^n.$$
\newpage
\begin{theorem}\label{T::Borel} 
Let $E$ be an $S$-Kolmogorov space and $u=(t,s,v) \in \Lt^1(E,E)$ a $1$-local 
morphism and 
$$f=\sum_{n \geq 0} a_n z^n \in \CM\{z\}$$ a power series with $R$ as radius 
of convergence. Then there is a well-defined map of Banach spaces over $\D \subset ]0,S]^2$, that we call the
{\em Borel map}\index{Borel map}:
$$\Bt f:\Xt(R) \to \mathcal{H}{om}(E,E),\ (t,s,v) \mapsto (t,s,\sum_{n=0} \frac{a_n}{n!} v^n) $$
and one has the estimate
$$\|\Bt f(u)\| \leq |f|\left(\frac{||u||}{t-s}\right) .$$
\end{theorem}
\begin{proof}
For $u \in \Lt^1(E,E)$, we have $u^n \in \Lt^n(E,E)$ and 
there is the power-estimate \ref{C::powerestimate}:
\[ \|u^n\| \leq \|u\|^n \]
There is an obvious map 
$$j_n: \Lt^n(E,E) \to \mathcal{H}{om}(E,E) $$
of Banach-spaces over $\Delta$, for which we have the bound
$$\frac{n^n}{e^n(t-s)^n}$$
on the norm. The choice of the constant $e$ gives a simple estimate 
for this norm:
$$\| j_n\|=\frac{n^n}{e^n(t-s)^n} \leq \frac{n!}{(t-s)^n} .$$
We obtain:
\[  \|\sum_{n=0}^{\infty} \frac{a_n}{n!} j_n(u^n)\| \leq \ \sum_{n=0}^{\infty} |a_n| \left(
\frac{||u||}{t-s}\right)^n= |f|\left(\frac{||u||}{t-s}\right) .\]
This proves the theorem.
\end{proof} 

Note that if $(t,s,u) \in \Lt^1(E,E)$, then $u$ is in fact a horizontal section over
the triangle $\D(t,s)$. So for each $(t',s') \in \D(s,t)$ we also have the restriction 
of $u$ to $\D(s',t')$

Using the above theorem, we get a horizontal section 
$$Bf(u) \in \Hom_{A(u)}(E,E)$$
with components
$$(t',s',\Bt f(u)_{t',s'}) \in  \mathcal{H}{om}(E,E)$$
The section $Bf(u)$ has definition set
\[A(u):=A(\|u\|,R):=\{(t',s') \in \D(t,s)\;|\;\;\|u\| <  R (t'-s') \}\]
 
Note that the data of $u$ over its upset determines the components of its Borel image 
\[\Bt f(u)_{t',s'}  \in \mathcal{H}{om}(E,E)_{t',s'}=Hom(E_{t'},F_{s'}) \]
for any point $(t',s')$ in the upset of $(t,s)$ sufficiently close to $(t,s)$. 
 
However in general the section $Bf(u)$ {\em is not defined in the entire upset} of 
$(t,s)$: as we  approach to the diagonal, the operator norm at
$(t',s')$ becomes larger and larger, and at some point, the series defining 
$Bf(u)$ no longer converges.
\section{The exponential mapping}
We now apply the previous considerations to the exponential series
\[ e^z=1+\frac{z}{1!}+\frac{z^2}{2!}+\frac{z^3}{3!}+\ldots\]
This series is the Borel transform of the geometrical series  
\[\frac{1}{1-z}= 1+z+z^2+z^3+\ldots ,\]
which has radius of convergence $R=1$. Using Theorem~\ref{T::Borel}, we deduce the following
statement:
 
\begin{proposition} \label{P::exponential} Let $E$ be an $S$-Kolmogorov space.
Then for any $(t,s,u) \in \Xt \subset \Lt^1(E,E)$ the exponential series converges to a well-defined section
\[ e^u \in \Hom_{A(u)}(E,E) ,\] 
where 
\[ A(u):=\{(t,s) \in \D \;|\;\;\|u\| < t-s \} .\]
Furthermore, there is an estimate for the norm-function of the section
$e^u$ on $A(u)$:
\[ |e^u| \le \frac{1}{1-\left(\|u\|/(t-s) \right)},\;\; \forall \;(t,s) \in A(u) .\]
\end{proposition}

In this generality the set $A(u)$ may very well be quite small. But there are
natural conditions that produce exponentials defined on large sets. What
is needed is that the norm of $u$, when restricted to $]0,t]$ becomes
small sufficiently fast.

We consider, as in chapter 9, the $K1$-spaces of local operators.
These were defined as the direct image of the
$K2$-space of local operators using the projection $p:(t,s) \mapsto t$:
$$L^k(E,E)=p_*\Lt^k(E,E) \to \RM_{\geq 0} ,$$ 
with fibre at $\tau$:
$$L^k(E,E)_\tau=\Lt^k(E,E)_{\tau,0}$$ 
and the norm is defined by
$$\| u \|=\sup \left(\frac{e}{k} \right)^k (t-s)^k\| u_{ts} \| . $$
where the sup is taken over  $\D(\tau,0)$.
Take for instance a $1$-local map of order $k>0$:
$$(\tau,u) \in  L^{1}(E,E)^{(k)}, $$
then
\[ \|u(t)\|= \a \,t^k+o(t^k)\]
for some $\a \in \RM_{\geq 0}$. In such a case, the definition set  $A(u)$ is 
non-empty and contains the origin in its closure.
 
If $k=1$ then, at the origin, the boundary of the set $A(u)$ is tangent to  the line $s=(1-\a)t$:\\
\begin{center}
\includegraphics[width=5cm]{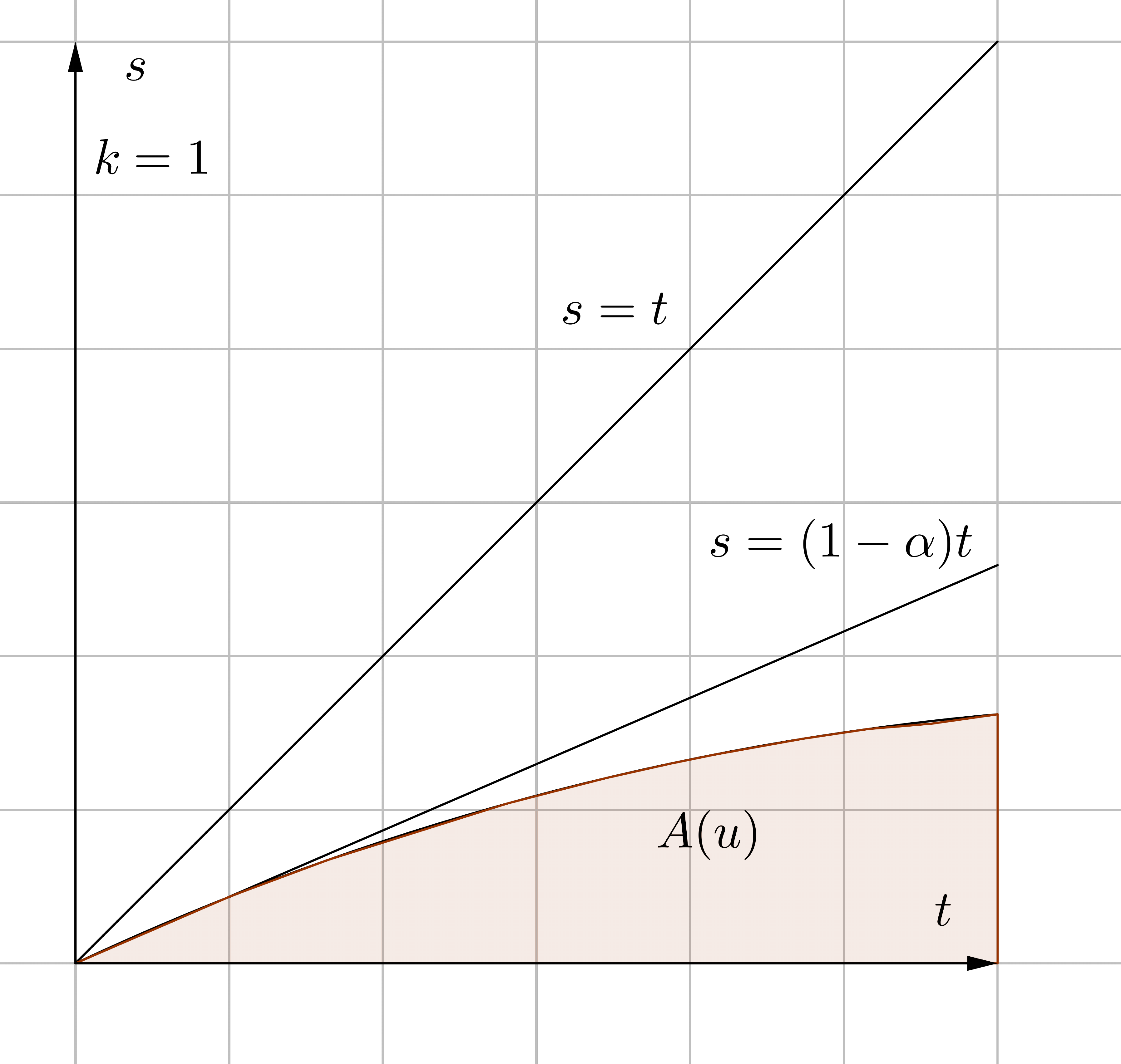}
\hspace{1cm}
\includegraphics[width=5cm]{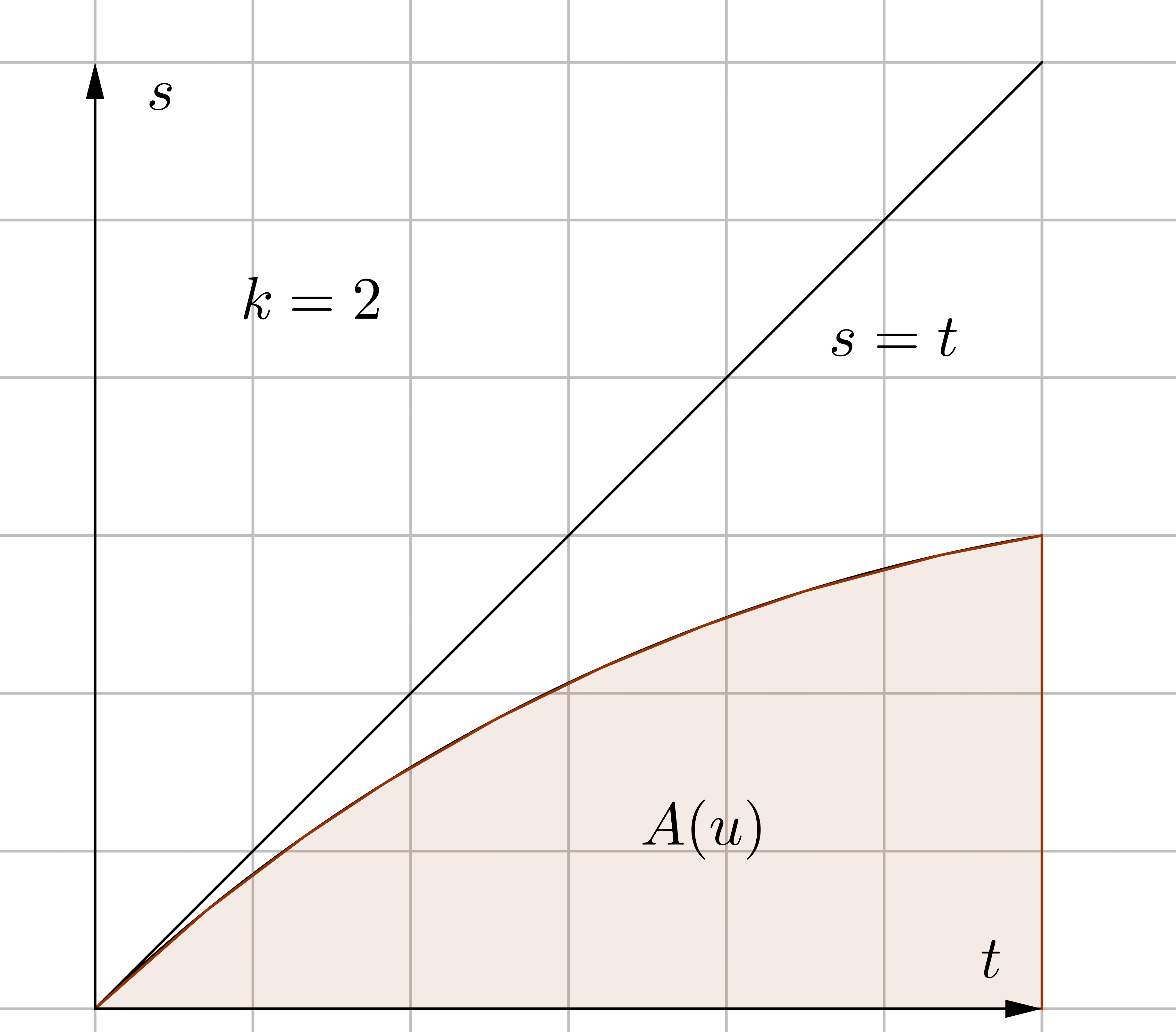}
\end{center}
\vskip0.3cm

If $k > 1$ then, at the origin, the boundary of the set $A(u)$ is tangent to the line $s=t$  at the origin.\\
%
\section{Some basic examples} 

We consider our standard Kolmogorov space $E:=\Ot^c(D)$, 
where $D_t \in \CM$ is the disc of radius $t$. 

We take $(\tau,u) \in L^1(E,E)$ with
\[    u:= \lambda \frac{d}{dz}  \]
From the Cauchy-Nagumo estimate it follows that
\[  \left|\frac{d}{dz}\right|_t=1, \]
so $|u|_\tau=|\lambda|$ and for the rescaled norm $||u||_t =e|\lambda|$. 
The exponential of $u$ is translation by $\lambda$, which is a partial morphism
\[e^u :E \to E, f(z) \mapsto f(z+\lambda)\]
The theorem gives 
\[ A(u):=\{ (t,s)\;|\;\; s \le t -e|\lambda| \}\] 
as definition set. 

In fact the translation can be 
defined on the bigger set, defined by the condition
\[  s \le t-|\lambda| .\]

\begin {center}
 \includegraphics[width=5cm]{Fig/definition_set_3.pdf}
\hspace{1cm}
\includegraphics[width=5cm]{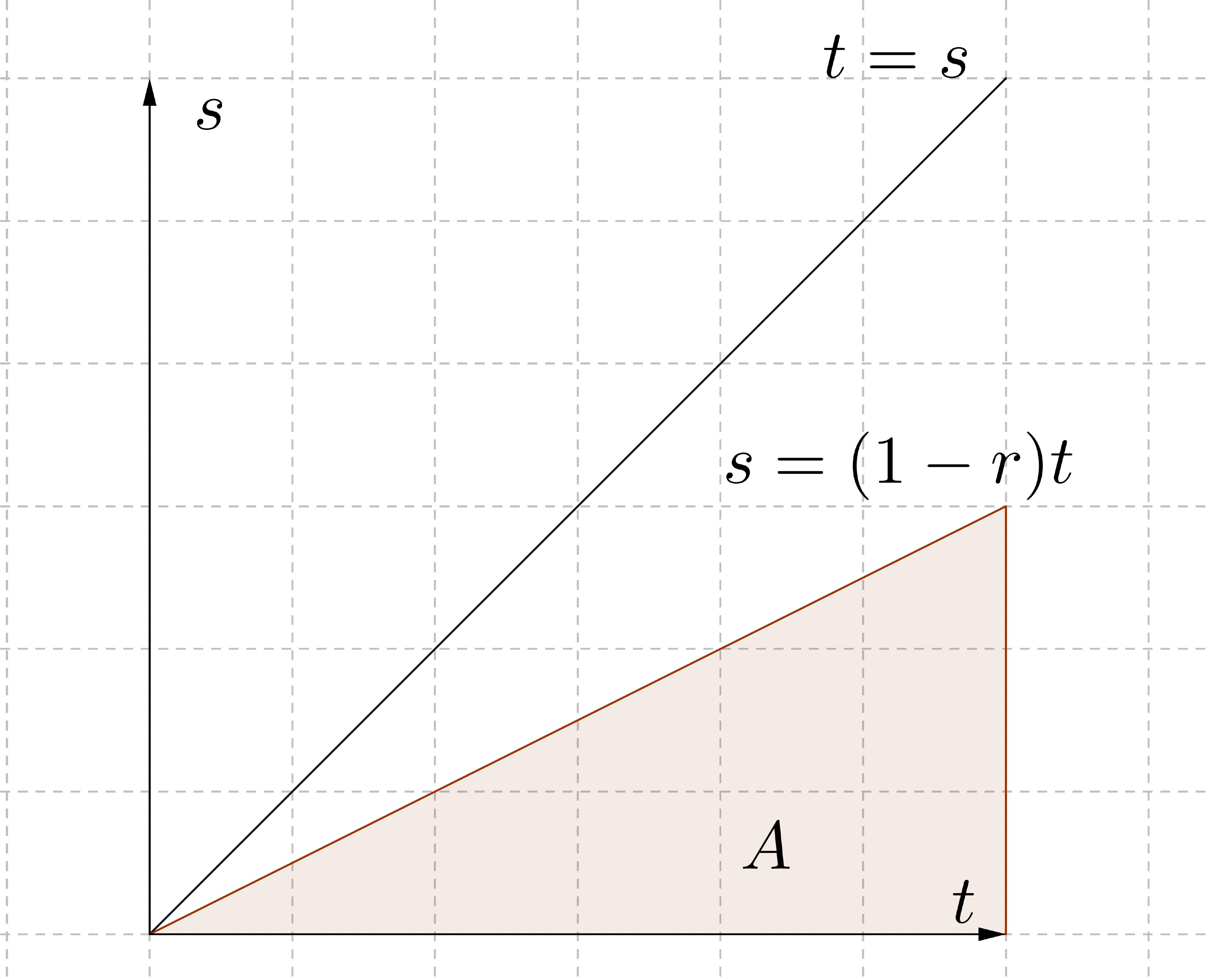}
\end{center}


This shows that the estimate obtained above in our theorem \ref{T::Borel}
is not optimal. This is to be expected,  as we are using only the simplest 
possible estimates.

Let us now take a local map of order one $u \in L^1(E,E)^{(1)}$ , we
have an estimate of the form:
\[ ||u(t)|| \le  \a  t.\] 
So for $\a < 1$ the definition set $A(u)$  contains the triangle
\[  s \le (1- \a)t \leq t-||u(t)||\]
 

Assume now that  
$$u:= \l z\d_z \in L^1(E,E),$$
where $\lambda \in \CM$. One has
\[|u(t)|=|\l| |z\d_z|_t =|\l| t,\;\;||u(t)||=|\l|e\, t \]
Thus the exponential of $u$ is the partial morphism
$$e^{z\d_z}:E \to E, f(z) \mapsto f(e^\l z) $$
and its maximal definition set is
$$\{ (s,t)\;|\;\;\; e^\l s \leq t \}, $$
but as before, our theorem gives a smaller definition set 
\[ A(u)=\{ (s,t)\;|\;\;\; s \le (1-|\l|e)t\}. \]

\section{Composition of exponentials} 
The composition of exponentials is straightforward in terms of
convolutions of their definition sets.

\begin{theorem} 
\label{T::compexponential1}
Let $E$  be an $S$-Kolmogorov space, and 
\[(\tau,u_i) \subset  L^1(E,E)\] 
a sequence of $1$-local morphisms. Assume that we have
estimates
$$|| u_i ||_t < \a_i t,\;\;i=0,1,2,\ldots  $$
and that  $\s:=\sum_{i=0}^{\infty} \a_i<+\infty$.
Then the sequence
$$g_n:=e^{u_n}e^{u_{n-1}}\cdots e^{u_0}  $$ converges to a partial morphism $g \in \Hom_A(E,E)$ where
\[A:=\{(s,t) \in ]0,\tau]^2\;|\;s <\prod_{i \geq 0}(1-\a_i)t \}.\]
Furthermore, one has the estimate for the norm function 
$$| g | < \frac{1}{1-\s\nu}  ,$$
with $\nu:= t/(t-s).$
 \end{theorem}
 \begin{proof}
Given an estimate of the form
$$\|u_i\|_t \le \a_i\,t $$
the definition set $A(u_i)$ contains the simplex
$$A_i:= \{(s,t) \in ]0,\tau]\times ]0,\tau]\;|\;s<(1-\a_i)t \}$$
and on $A$ one has
$$|e^{u_i}| \leq \frac{1}{1-\a_i \nu} $$
with $\nu=t/(t-s)$. 

Let us now consider the composition of two exponentials $e^{u_i} e^{u_{i+1}}$, assuming that
$$\|u_i\|_t \le \a_i$$
As 
$$\frac{1}{1-x} \times \frac{1}{1-y} < \frac{1}{1-(x+y)} $$
for $x,y \in ]0,1[$, we get that the composition $e^{u_{i+1}} e^{u_i}$ defines a morphism over
$$ A_{i+1} \star A_i =\{(s,t) \in ]0,\tau]\times ]0,\tau]\;|\;s<(1-\a_i)(1-\a_{i+1})t \}$$
and one has the estimate
$$|e^{u_{i+1}} e^{u_i}| \leq \frac{1}{1-(\a_i+\a_{i+1}) \nu} $$
By a straighforward induction, we get the estimate
$$|e^{u_n}e^{u_{n-1}}\cdots e^{u_0}  | \leq \frac{1}{1-(\sum_{i=0}^n\a_i) \nu} $$
over the set
$$ A_{n}\star \cdots \star A_0=\{(s,t) \in ]0,\tau]^2\;|\;s <\prod_{i = 0}^n(1-\a_i)t \} \supset A.$$

\end{proof}
Note that the sequence $e^{-u_0}e^{-u_1}\cdots e^{-u_n}$ converges to a pseudo inverse $h$ of $g$:
$$hg=gh=\iota_E \in  \Hom_{A \star A}(E,E).$$

\section{Formulation with sequences}




We will later use the following variant of Theorem \ref{T::compexponential1} on the composition of exponentials, now considered as a sequence of elements of
$L^1(E,E)$.\\

\begin{theorem} 
\label{T::compexponential2}
Let $E$  be an $S$-Kolmogorov space and $(t:=t_0,t_1,t_2,\ldots) $ a 
decreasing sequence converging to $s>0$. For any sequence
$(t_n,u_n) \in L^1(E,E)$  such that
\begin{enumerate}[{\rm i)}] 
\item $\| u_n \| \leq  t_n-t_{n+1}$
\item $\s:=\sum_{n \geq 0} || u_n ||/(t_n-t_{n+1}) < +\infty$ 
\end{enumerate}
the sequence 
\[g_0=e^{u_0},\;\;g_1=e^{u_1}e^{u_0},\]
$$g_n:= e^{u_n}e^{u_{n-1}}\cdots e^{u_0}  $$ 
converges to an element $ g \in \Hom(E_{t_0},E_s)$. 
Furthermore, we have the estimate:
$$ | g | < \frac{1}{1-\s/(t-s)}  .$$
 \end{theorem}
\begin{proof}
We have seen that $e^{u_i}$ exists as a section and thus defines elements 
of the Banach space $Hom(E_{t_i},E_{t_{i+1}})$ as long as $\|u_i\| \le t_i-t_{i+1}$, 
which holds by the first assumption. As a consequence
the compositions
\[e^{u_0},\;\;e^{u_1}e^{u_0},\;\;\ldots,e^{u_n}\ldots e^{u_1} e^{u_0},\ldots\]
are well defined. Furthermore, the Borel estimate gives
$$|e^{u_i}| \leq \frac{1}{1-\nu_i},\ \nu_i:=\| u_i\|/(t_i-t_{i+1}) .$$
As 
$$\frac{1}{1-x} \times \frac{1}{1-y} < \frac{1}{1-(x+y)} $$
for $x,y \in ]0,1[$, we get for the composition $e^{u_{i+1}} e^{u_i}$
$$|e^{u_{i+1}} e^{u_i}| \leq \frac{1}{1-(\nu_i+\nu_{i+1})} .$$
By a straighforward induction (and the fact that restrictions have norm
$\le 1$), we obtain the estimate
$$| g_n  | \leq \frac{1}{1-(\sum_{i=0}^n\nu_i)} .$$
Therefore
$$| g_{n+1}-g_n  | \leq 
\frac{| e^{u_{n+1}}-\Id |}{1-(\sum_{i=0}^n\nu_i)}    $$
Using again the Borel estimate 
$$ |e^{u_{n+1}}-1| \le \frac{\nu_{n+1}}{1-\nu_{n+1}} $$ 
we get 
$$| g_{n+1}-g_n  | \leq 
\frac{ \nu_{n+1}}{1-(\sum_{i=0}^{n+1}\nu_i)}  $$
From this it follows that the sequence $g_n$ converges in the Banach
space $Hom(E_t,E_s)$ with operator norm.
\end{proof}
  
\chapter{The fixed point theorem}
In Part 1, we saw that Kolmogorov's invariant torus theorem is a result on 
normal forms, which can be studied in the more general framework of group 
actions in infinite dimensional spaces. One could try to develop implicit 
function theorems to mimic the finite dimensional case. This approach however 
leads to inextricable difficulties. For this reason we developed the 
analysis of Kolmogorov spaces, which allows us to use a simpler method.
 
We study the Lie-iteration directly in  appropriate Kolmogorov spaces and 
develop a {\em fixed point theory} in Kolmogorov spaces. The resulting 
{\em fixed point theorems} are a little bit like to one for Banach spaces, but slightly more involved.

In the next chapter we will then apply the fixed point theorems directly to the
Lie-iteration, which leads the general normal form theorems. We first focus on fixed 
point theory.

%
\section{The Banach fixed point theorem}
The fixed point theorem for a complete metric space has been popular 
since its formulation by Banach in $1922$. It belongs to the most basic 
results in functional analysis and has many applications.

A {\em contraction with factor $C$} in a complete metric 
space $(X,d)$ is map $f:X \to X$ with the property that 
\[ d(f(x),f(y)) \le C d(x,y) .\]
If $C<1$, then starting from any point $x_0 \in X$, we obtain by iteration 
a sequence
\[ x_1=f(x_0),\;\;x_2=f(x_1),\ldots,x_{n+1}=f(x_n)\]
that by the above inequality is seen to be a Cauchy sequence, whose limit
is the unique fixed point $x_{\infty}$.
Clearly,
\[d(x_{n},x_{\infty}) \le C^n d(x_0,x_{\infty})\]
The simplicity of the proof of this fixed point theorem stems in part from the
fact the a priori infinite dimensional iteration 
$$ x_{n+1}=f(x_n),\;\; x_n \in X$$
is 'governed' by the very simple one-dimensional dynamical system
$$y_{n+1}=C y_n ,\ y_n=\| x_n\| \in \RM_{>0}.$$

We will formulate a similar iteration scheme in the context of 
relative Banach spaces. There are two important differences: first, we 
will need the fixed point to be {\em critical}, in order to overcome 
the 'detoriation of the estimate', which in the applications we have in 
mind are caused by the appearence of small denominators and the shrinking 
of domains of definition and more generally by the presence of singularities. 
Second, the iteration is no longer 'governed' by 
a one-dimensional dynamical system. Rather we have to deal with iterations 
in two or three dimensions, which fortunately are  almost as simple as the 
above one-dimensional system.
\section{Convergence in relative Banach spaces}
Although our main concern are Kolmogorov spaces, the maps we wish to iterate are
usually only mappings of relative Banach space mappings. 
We have seen in \ref{SS::definition} that a Kolmogorov space carries a canonical
topology. For a relative Banach space
$$E \to B $$
there is no obvious sense in which a sequence of elements can be said to converge. 
If the base $B$ has a topological structure, we may nevertheless define convergence 
to zero, by mimicking the definition from Kolmogorov spaces. So from now on and for 
the rest of this chapter, we assume that the base $B$ of our Banach spaces carries 
a topology.

\begin{definition}
A sequence of elements $x_n=(b_n,v_n)$ in a relative Banach space $E \to B$
is said to converge to $0$, if the sequence $b_n$ converges to $b \in B$
and the sequence of norms of the $v_n$ converge to $0$ in $\RM_{ \ge 0}$:
\[ \lim_{n \to \infty} |v_n|_{b_n} = 0 .\]
\end{definition}

\section{Ordered diagrams}
Our first task is to introduce an appropriate language to explain what it 
means to be 'governed' by a finite dimensional dynamical system.

In algebra one deals with {\em equality} of expressions. The equality of
two different compositions of maps leads to the notion of 
{\em commutative diagram}, that pervades a large part of mathematics. 

The core of analysis is to a large part the manipulation of {\em inequalities}.
Therefore, we introduce a corresponding notion of certain non-commutative diagrams that 
we call {\em ordered diagrams}. 

\begin{definition}\index{ordered diagram}
Let $B$ be an ordered set and $X,Y,Z$ arbitrary sets. A diagram
$$\xymatrix{X \ar[r]^-\phi \ar[d]_{p_1}  \ar@{}[rd]|{\geq} &  Y \ar[d]^{p_2}\\
Z \ar[r]_{\p} & B
}$$
is called {\em ordered} if
$$ \p \circ p_1 \geq p_2 \circ \phi .$$
\end{definition}

One can similarly define ordered diagrams for $\leq,\ <,\ >,=$, where the 
last case corresponds to commutativity. Many analytic notions can be 
formulated in terms of such ordered diagrams. 

For example, let $E$ be a Banach space and consider the Banach space $L(E,E)$  
of continuous linear operators of $E$ with the operator norm. Take a fixed 
vector $x \in E$ and consider the evaluation map
$$\phi:L(E,E) \to E,\;\;u \mapsto u(x) .$$
By definition of the operator norm 
$$\|-\|:L(E,E) \to \RM_{\ge 0},$$ 
the diagram
$$\xymatrix{L(E,E) \ar[r]^\phi \ar[d]_{\| \cdot \|} \ar@{}[rd]|{\geq}&  E \ar[d]^{| \cdot |}\\
\RM \ar[r]_{| x | \cdot} & \RM
}$$
is ordered.

The contraction property of a map $f:X \to X$ can also be formulated with the
help of an ordered diagram. Assume for example that  $X \subset E$ is a subset of a Banach space, with $0$ as fixed point. Then we get an ordered diagram
$$\xymatrix{X \ar[r]^f \ar[d]_{\| \cdot\|} \ar@{}[rd]|{\geq}  &  X \ar[d]^{\| \cdot\|}\\
\RM_{\ge 0} \ar[r]_{\l\cdot} & \RM_{\ge 0}
}$$

The proof of convergence to the fixed point can be split in two distinct conceptual steps:

\begin{enumerate}
\item The contraction property (with factor $C \ge 0$) leads to an ordered diagram.
\item For $C <1$, the dynamical system $\RM \to \RM,x \mapsto C x$ has the 
origin as unique fixed point.
\end{enumerate}

We will be dealing with relative Banach space and the iteration of relative 
maps, by which we mean the following.

Let $p: X \to B$ be a set over $B$. We will consider {\em iterations over $B$}\index{Iteration over a base}, by which we mean a commutative diagram of the form
$$\xymatrix{X \ar[r]^-F \ar[d]_{p}  \ar@{}[rd]|{=} &  X \ar[d]^{p}\\
B \ar[r]_{\phi} & B 
} .$$ 

So both $F$ and $\phi$ can be iterated. If $X$ and $B$ are metric spaces, we can
speak about convergence of the sequence of iterates. If the map $p$ is 
continuous and the sequence of iterates 
\[x_0,\;\;x_1=F(x_0),\;\;x_2=F(x_1), \ldots,\;\;x_{n+1}=F(x_n)\] converges in $X$, 
then clearly the sequence of projections $b_n=p(x_n)$, which is the
sequence of iterates 
\[b_0,\;\;b_1=\phi(b_0),\;\;b_2=\phi(b_1), \ldots, \;\;b_{n+1}=\phi(b_n),\] 
converges in $B$, and of course the converse needs not hold in general.
But there is an important situation where some sort of converse holds.
  
Let $X \to B$ is a subset of a relative Banach space $p: E \to B$, containing 
the $0$-section. Assume we have an estimate of the form 
\[ |F(x)| \le g(p(x),|x|) \]
We can ``iterate the inequality'', meaning that
we can look at the iteration of the associated map
\[ f: B \times \RM_{>0} \to B \times \RM_{>0},\;\;(b,x) \mapsto f(b,x):=(\phi(b),g(b,x)) .\]
The norm map 
\[ \nu: X  \to B \times \RM_{\ge 0},\;\;\,x \mapsto (p(x),|x|)\]
restricted to $X$ can be used to control the convergence in $X$.
The occurring maps and the estimate on $F(x)$ can be neatly summarised in an
ordered diagram
$$\xymatrix{X \ar[r]^-F \ar[d]_{\nu}  \ar@{}[rd]|{\geq} &  X \ar[d]^{\nu}\\
B \times \RM_{\ge 0} \ar[r]_{f} & B \times \RM_{\ge 0}
}$$ 
where $|(b,x)|:=|x|$, that is, {\em only the last component counts in the
comparison}. Then the convergence of the $f$-iterates implies the convergence 
in norm of $F$-iterates, more precisely:

\begin{proposition}\label{P::iterates}
Let $E \to B$ be a Banach space over $B$ and $X \to B$ a subset of $E$ 
containing the $0$-section. Assume a maps $F$ defines
an iteration over $B$
\[\xymatrix{X \ar[r]^-F \ar[d]_{p}  \ar@{}[rd]|{=} &  X \ar[d]^{p}\\
B \ar[r]_{\phi} & B }\] 
and assume moreover that we have an ordered diagram
$$\xymatrix{X \ar[r]^-F \ar[d]_{\nu}  \ar@{}[rd]|{\geq} &  X \ar[d]^{\nu}\\
B \times \RM_{\ge 0} \ar[r]_{f} & B \times \RM_{\ge 0}
}$$ 
where $\nu$ denotes the norm map.
Then the sequence of $F$-iterates
$$x_0,\;\;x_1=F(x_0),\;\;x_2=F(x_1),\;\;\ldots,\;\;x_{n+1}=F(x_n),\ x_0 \in X$$
converges to $0 \in E_b$,  if the sequence of $f$-iterates
$$y_{n+1}=f(y_n),\  y_0=\nu(x_0) \in B \times \RM_{\ge 0} $$ 
converges to $(b,0) \in B \times \RM_{\ge 0}$.
\end{proposition}
\begin{proof}
Put $y_n:=(b_n,z_n)$. By assumption the sequence $(b_n)$ converges to a limit $b \in B$. From the diagram, we get that
 $$|x_{n+1}|=\nu \circ F(x_n) \leq f \circ \nu(x_n)=z_n \xrightarrow[n \to +\infty ]{} 0  $$
\end{proof}

\section{Critical fixed points}
For relative Banach spaces there is a version of the contraction theorem
in the case of {\em quadratic critical points}. This was an essential 
discovery of Kolmogorov: the quadratic nature of an iteration compensates 
the presence of small denominators. In order to appreciate this point, 
let us first consider the  prototype of a linear map in a Banach space:
$$\RM \to \RM,\;\;\; x \mapsto C x ,$$
which is a contraction when the multiplicator $C <1$.
In a relative Banach space over a base $B$, the constant $C=C(b)$ 
will depend on $b \in B$, which during the iteration is changing.
So it may happen that for some $b$ the multiplicator $C(b) \ge 1$, 
and we loose control over the convergence.
Typically, we are dealing with local maps over the open sub-diagonal
$\D=\{ t > s > 0\}$ of the form
$$ x \mapsto \frac{C}{s^k(t-s)^l} x ,$$
and during the iteration $s$ or $t-s$ can  become small. So that
the multiplicator will become arbitrary large, complicating the issue
of convergence. It is unclear how to develop fixed point theory for such maps.

However, if we follow Kolmogorov and iterate maps of the form
$$ x \mapsto \frac{C}{s^k(t-s)^l} x^2,$$
then the extreme rapidity of convergence of the iterates of the 
quadratic map $x \mapsto x^2$ can compensate for the polar behaviour of
the coefficient
\[\frac{C}{s^k(t-s)^l} .\]
In the iteration process we are about to describe, the points $t$ and $s$ 
will approach each other according to a geometrical sequence, hence the
polar blow-up is easily compensated by the quadraticity of the map.
Of course, the same works more generally for maps that behave like
$$ x \mapsto \frac{C}{s^k(t-s)^l} x^d $$
with $d>1$.
There is an important subtlety to be observed here: during the iteration, 
the parameters $(s,t)$ converge to a point on the diagonal $s=t$, for 
which the map is not defined. The fixed point lies in the {\em closure} 
of the domain of definition of the map we iterate, so we are not dealing 
with a fixed point in the traditional sense of the word.





Let us define more precisely the non-linear maps we will consider.

\begin{definition}\index{ramification order}
Let $E$ and $F$ be relative Banach spaces over $\D$ and 
$X \to \D$ a subset of $E$.
We say that a map 
$$f:X \to F $$ 
has {\em ramification order} $(k,l,m) \in \RM^3$, if for some constant $C >0$, 
there is an ordered diagram 
$$\xymatrix{X \ar[r]^f \ar[d]_{\nu} \ar@{}[rd]|{\geq}  &  F \ar[d]^{\| \cdot\|}\\
 \D \times \RM_{\geq 0} \ar[r]_{\p} & \RM_{\ge 0}
}$$
with $\p(t,s,x)=C s^{k}(t-s)^{l}x^{m}$.
\end{definition}
So there are three basic invariants of a function $f$ between Banach 
spaces over $\D$, corresponding to the vanishing orders along the edges $t=0$, $s=t$ and $|x|=0$.
In applications $k,l$ are negative numbers and $m=2$. 
Such maps appear naturally, as the following class of examples shows.\\
Start with the usual Kolmogorov space 
$E:=\Ot^c(D)$ where $D$  the unit disc over $\RM_{>0}$.
Consider the K2-spaces
$$\Et=\Lt^1(E,E),\;\;\; \Ft=\mathcal{H}om(E,E) .$$
over the open sub-diagonal $\D$. 
For an analytic series $f$ of convergence radius $R$, we have the Borel map \ref{T::Borel}
$$\Bt f: \Et \supset \Xt(R)  \to \Ft,\;\;\; u \mapsto Bf(u)$$
for which we have the estimate:
$$\| Bf(u) \| \leq |f|\left(\frac{||u||}{t-s}\right) .$$
If $f$ lies in the $d$-th power of the maximal ideal and $0 < r< R$ chosen,
there exists a constant $C >0$, such that
$$|f(z)|\leq |f|(|z|) \leq C|z|^d,\;\;\;\textup{for}\;\; |z| \le r .$$
So if we restrict the map $\Bt f$ to the tetrahedron 
$$\Xt(r)=\{(t,s,u)\;|\;|u| \le r(t-s) \}$$ 
we deduce an estimate of the form
$$\| Bf(u) \| \leq C\left(\frac{||u||^d}{(t-s)^d}\right), \;\;\;\textup{for}\;\;\; u \in \Xt(r) .$$
So $Bf$ can be seen as a non-linear map with $(0,-d,d)$ as ramification orders.

We now need to understand what will replace the notion of contraction. 
This is slightly more involved than the simple condition  ``$C<1$'' of the 
Banach fixed point theorem, but nevertheless remains elementary.
  
\section{A simple iteration}
The iteration scheme we will use in Kolmogorov spaces, is governed by a simple 
three-dimensional dynamical systems that we will now study in some detail. 
For simplicity, we start with the case where the singularity is quadratic
and there is no pole along $s=0$.\\

We consider the {\em prisma}\index{prisma} 
\[ P:=\D \times \RM_{>0}:=\{(t,s,x)\in \RM_{ > 0}^2 \times \RM_{ \ge 0}\; |\; t >s \} \subset \RM^3.\]
as a subset of the K2 space
$$ \mathcal{P}: \D \times \RM \to \D,\ (t,s,x) \mapsto (t,s).$$
We fix numbers $C, k,\l \in \RM_{>0}$ with $\l < 1$ and define a map
\[ f: P \to  P,\;\;\; (t,s,x) \mapsto (s,s-\l (t-s), C\frac{x^2}{(t-s)^k} ) .\]
We are interested in the sequence of iterates
\[x, \;\;f(x),\;\;f(f(x)),\ldots \] 
As the variable $x$ does not appear in the first two coordinates of the map, we are dealing with a relative iteration, i.e. there is a commutative diagram
$$\xymatrix{ P \ar[r]^-{f} \ar[d]_{p} \ar@{}[rd]|{=} &  P \ar[d]^{p}\\
\D \ar[r]^\phi & \D} ,$$
where the map $p: (t,s,x) \mapsto (t,s)$ just forgets the last coordinate and
\[\phi(t,s)=(s,s-\l(t-s)) .\]

The dynamics of the map $\phi$ is very elementary:\\

\begin{center}
\includegraphics[width=0.6\linewidth]{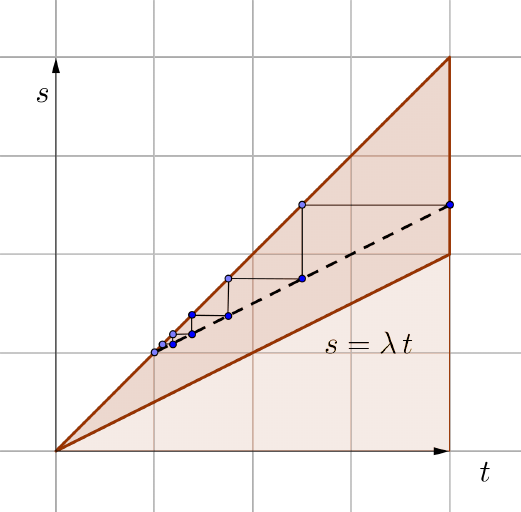}
\end{center}

\begin{proposition}\label{P::baseiteration}
For $(t_0,s_0) \in \D \cap \{ s>\l t\}$, the sequence of iterates
\[(t_0,s_0),\;\;(t_1,s_1):=\phi(t,s),\;\; (t_2,s_2):=\phi(\phi(t,s)),\;\ldots\]
converges to the point $(t_{\infty},t_{\infty}) \in \overline{\D}$,
where 
$$t_{\infty} :=\frac{s_0-\l t_0}{1-\l}$$
Explicietly:
\[ t_{n+1}=t_{\infty}+\l^n(t_0-t_{\infty})=s_n ,\]
\end{proposition}
We omit the easy proof. Note that the difference of the coordinates in each step gets multiplied by $\l$:
\[ s-(s-\l(t-s))=\l(t-s) .\]
The above picture and formulas show how we approach the diagonal during the 
iteration. 

Note that the limit point $(t_{\infty},t_\infty)$ does not belong $\D$, but
lies in its closure $\overline{\D}$; the third component of the map $f$
is not defined at the points of the diagonal.

\begin{definition} \index{tetrahedron $\Xt(R,k,\l)$} Let $E$ be a Banach space
over $\D$ and $X \subset E$ a subset. We define 
$$ \Xt_X(R,k,\l):= \{(t,s,x) \in X :  |x| <R \l^k(t-s)^k,\;\;s > \l t\} .$$ 
\end{definition}
For $k=1, \l=1$, we recover the definition of the tetrahedron given in \ref{D::tetrahedron}.
Quadratic iteration have the following property:
\begin{definition}\index{rapid convergence} 
We say that a sequence of of numbers $r_n \in \RM$ {\em converges rapidly to $0$},
if there exist $0 \le C <0$ and $\rho >0$, such that for all $n \in \NM$ one
has:
\[ |r_n| \le  C^{{\rho}^n}\]
We will say that a sequence $(x_n)$ of a relative Banach space converges {\em rapidly} to $0$
if the sequence $|x_n|$ converges rapidly to $0$.
\end{definition}
\begin{theorem} \label{T::D1} 
Fix  $C, k,\l \in \RM_{>0}$ with $\l  < 1$ and consider
the map
\[ f: P \to  P,\;\;\; (t,s,x) \mapsto (s,s-\l (t-s), C\frac{x^2}{(t-s)^k} ) . \]
Put 
\[ R:=C^{-1} .\]
Then the following holds: 

\begin{enumerate}[{\rm i)}]
\item The set $\Xt_P(R,k,\l)$ is $f$-invariant.
\item For any $(t_0,s_0,x_0) \in \Xt_P(R,k,\l)$, the third component of the sequence of  iterates 
$$(t_n,s_n,x_n)=f^{(n)}(t_0,s_0,x_0)$$ converges to zero rapidly and more precisely:
$$x_n = \left(R \l^k(t_0-s_0)^k \right)^{1-2^n} \l^{k n} x_0^{2^n}. $$
\end{enumerate}
\end{theorem}
\begin{proof}
Assume $(t,s,x) \in \Xt_P(R,k,\l)$ then, as $R=C^{-1}$, we have:
\begin{align*}
 x^2  &\leq R^2 \l^{2k}(t-s)^{2k},\\
\frac{C x^2}{(t-s)^2} & \leq  R \l^{2k}(t-s)^k .\\
\end{align*} 
The condition $s >\l t$ is obviously preserved by $f$,
thus  $f(t,s,x) \in \Xt_P(R,k,\l)$. This shows i).\\

The second statement ii) is an easy induction. The statement holds for $n=0$.
Define 
$$K=  R\l^k(t_0-s_0)^k .$$
As  $t_n-s_n=\l^n(t_0-s_0)$,  we get that if we assume that
$$x_n  =  K^{1-2^n} \l^{k n} x_0^{2^n},$$
then
\begin{align*}
x_n^2 &=   K^{2-2^{n+1}}\l^{2kn}x_0^{2^{n+1}}, \\
\frac{C x_n^2}{(t_n-s_n)^k} & = \frac{C \l^{-nk}}{(t_0-s_0)^k}K^{2-2^{n+1}}\l^{2kn}x_0^{2^{n+1}}\\
 & = K^{1-2^{n+1}}\l^{k(n+1)}x_0^{2^{n+1}}
 \end{align*}
 which shows that the equality is also right for $n+1$.
\end{proof}
 
The following picture shows a typical iteration in the variables $\s=(t-s)$ and $x$:\\

\begin{center}
\includegraphics[width=0.8\linewidth]{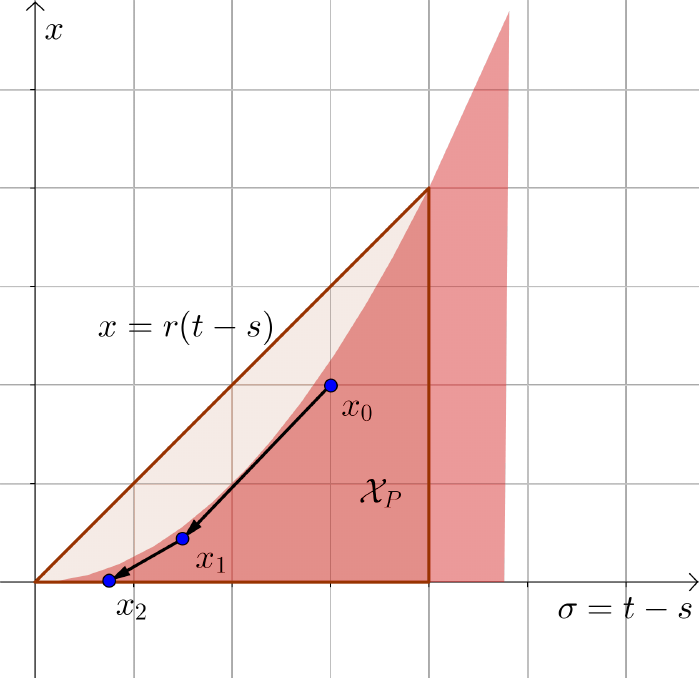}
\end{center}

The invariant set $\Xt_P(R,k,\l)$ is concave and as the components of the $f$-iterates are decreasing,
the intersection of $\Xt_P(R,k,\l)$ with 
$$\Xt_P(r)=\{ (t,s,x) \in P: x \leq r(t-s) \}  $$
is also $f$-invariant.\\
%

Using our simple observation \ref{P::iterates}, we now deduce a 
fixed point theorem for relative Banach spaces over $\D$.

Note however that our concept of convergence is not directly applicable for
Banach spaces $E \to \D$, as the limit point of the sequence on the base
belongs to the diagonal and is not in $\D$.
A more precise way of formulating the result would be to consider 
\[ \overline{\D}:=\{(t,s) \in ]0,S]^2\;|\;s \le t\} \supset \D:=\{(t,s) \in ]0,S]^2\;|\;s < t\} \]
and a Banach space $\overline{\Et} \to \overline{\D}$ that restricts to $\Et \to \D$, etc. As the matter is unimportant, we will suppress it from our notation.

As usual, we denote by
\[\nu: \Et \to \D \times \RM_{\ge 0},\ (t,s, v) \mapsto (t,s,|v|) \]
the norm map.
\begin{theorem}
\label{T::Fixedpoint} 
Let $\Et \to \D$ be a relative Banach space.
 Assume that for a given map
\[ F: \Xt_\Et(R,k,\l) \to \Et,\ (t,s,v) \mapsto (s,s-\l(t-s),g(t,s,v))\]
there is an ordered diagram
$$\xymatrix{\Xt_\Et(R,k,\l)  \ar[r]^-F \ar[d]_{\nu}   \ar@{}[rd]|{\geq} &  \Et \ar[d]^{\nu}\\
\Xt_P(R,k,\l)   \ar[r]_-{f} & P
}$$
with
\[ f: P \to  P, (t,s,x) \mapsto (s,s-\l (t-s), \frac{x^2}{R(t-s)^k} ) . \]
Then the set $\Xt_\Et(R,k,\l)$ is $F$-invariant and for $x_0 \in \Xt_\Et(R,k,\l)$
the sequence
\[x_{n+1}=F(x_{n})\]
converges rapidly to $0$.
\end{theorem}
\section{Quadratic iterations}
The above formulated fixed point theorem can be generalised to the case of maps with
ramification orders $-k,-l,d>1$, i.e. modelled on the map
\[ x \mapsto \frac{C}{s^k(t-s)^l}x^d .\]
We consider only the case $d=2$, which is the one relevant for our applications.. The case of general $d$ is straightforward and only notationally more difficult.\\

We consider as before the prisma 
\[ P:=\D \times \RM_{\ge 0}:=\{(t,s,x)\;| t >s, |x| \ge 0\}.\]
We fix four numbers $C, k, l \in \RM_{>0}$  and $0 <\l <1$ and 
define a map
\[ f: P \to  P, (t,s,x) \mapsto (s,s-\l (t-s), C\frac{x^2}{s^k(t-s)^l} ) .\]
At each step $s$ gets  multiplied by the factor
$$\rho(t,s):=\frac{s-\l(t-s)}{s}=1+\l-\l \frac{t}{s}$$
which monotonously increases and converges to $1$. 
This explains that ramification along $s=0$ is in fact harmless, although it
affects the estimates. We will need the following set:

\begin{definition}\index{set $\Xt(R,k,l,\l)$}  Let  $E$ be a Banach
space over $\D$ and $X \subset E$ a subset. We define 
\[ \Xt_X(R,k,l,\l):=\{(t,s,x) \in X\;|\;  |x| <R\rho^k(t,s)s^k\l^l(t-s)^l, s >\l t\}.\]
\end{definition}
Note that
$$ \Xt_X(R,0,l,\l)= \Xt_X(R,l,\l).$$
The following is a direct generalisation of $\ref{T::D1}$:

\begin{theorem}\label{T::D2} Fix  $R, k, l \in \RM_{>0}$ and $0 <\l  < 1$ and consider the map
\[ f: P \to  P, (t,s,x) \mapsto (s,s-\l (t-s), \frac{x^2}{R s^k(t-s)^l} ) . \]
Then the following holds : 

\begin{enumerate}[{\rm i)}]
\item The set  $\Xt_P(R,k,l,\l)$ is $f$-invariant.
\item For any $(t_0,s_0,x_0) \in \Xt_P(R,k,l,\l)$ the third component of the sequence of iterates 
$$(t_n,s_n,x_n)=f^{(n)}(t_0,s_0,x_0)$$ converges to zero rapidly and more precisely:
$$x_n = \left(  Rp_n^k s_0^k \l^l(t_0-s_0)^l \right)^{1-2^n} \l^{l n} x_0^{2^n}. $$
with $p_n=\prod_{i=0}^{n-1} \rho(t_i,s_i)$, $p_0=1$.
\end{enumerate}
\end{theorem}
Note that 
$$\prod_{i=0}^{n-1} \rho(t_i,s_i) < \rho(t_0,s_0)^n:=\rho_0^n $$
and therefore
$$x_n < \left(  R\rho_0^{k} s_0^k \l^l(t_0-s_0)^l \right)^{1-2^n} \rho_0^{k n} \l^{l n} x_0^{2^n}. $$
This explains that our equality implies rapid convergence to zero.
\begin{proof}
Assume $(t,s,x) \in \Xt_P(R,k,l,\l)$ then we have:
\begin{align*} 
x^2   &\leq R \rho^{k}(t,s)s^{k}\l^{l}(t-s)^{l},\\
  \left( \frac{x^2}{R s^{ k}(t-s)^{ l}} \right) &\leq  R\rho^{2k}(t,s)s^{ k}\l^{2 l}(t-s)^{l}
\end{align*}
 
With the notation $(t',s',x')=f(t,s,x)$, the previous inequality can be written as:
$$  x'  \leq  R\rho^{2 k }(t,s)s^{ k}\l^{2l}(t-s)^l.$$ 
Using $\rho(t,s) < \rho(t',s')$, $\rho(s,t)=s'$ and $t'-s'=\l(t-s)$ we 
obtain:
$$  x'  \leq  R\rho^k (t',s')(s')^k\l^l(t'-s')^l. $$
So we see that  $f(t,s,x) \in \Xt_P(R,k,l,\l)$. This shows i).\\

The second statement ii) is an easy induction. The statement holds for 
$n=0$. Define 
$$K_n=   Rp_n^k s_0^k \l^l(t_0-s_0)^l .$$
As $s_n=p_n s_0$ and $(t_n-s_n)=\l^n(t_0-s_0)$,  we obtain:
$$\frac{ \l^{2ln}}{R s_n^k(t_n-s_n)^l}  = \frac{ \l^{2ln}}{R\l^{nl} p_n^ks_0^k(t_0-s_0)^l} = K_n^{-1} \l^{ l(n+1)} $$
So if we assume that 
\[ x_n  =  K_n^{1-2^n} \l^{l n} x_0^{2^n},\]
then
\begin{align*}
x_n^2 &=   K_n^{2-2^{n+1}}\l^{2ln}x_0^{2^{n+1}} \\
\frac{ x_n^2}{R s_n^k(t_n-s_n)^l} & = K_n^{1-2^{n+1}}\l^{ l(n+1)}x_0^{2^{n+1}} .
 \end{align*}
which shows that the equality is also right for $n+1$.

\end{proof}

%
\section{Parametric version}
In order to prove the normal form theorem in the inhomogenous case, 
we need to study the case of a map depending on parameters. 
This means that we are now given two relative Banach spaces
$$\Et \to \D,\;\;\; \Ft \to \D ,$$
where the space $\Ft$ will serve as {\em parameter space}.
We consider a function on a subset of their direct sum 
$$X \to \D,\ X \subset \Et \oplus  \Ft  ,$$
If we assume uniform estimates in the parameter in the unit ball $B_\Ft$
of $\Ft$, we get a straightforward variant of the fixed point theorem. 
For simplicity we only consider the case $d=2$.
Our finite dimensional model is as follows. 
 
\begin{definition}\index{set $\Xt(R,k,l,\l,r)$}  Let  $E$ be a Banach
space over $\D$ and $X \subset E$ a subset. Omitting the dependence on $k,l,\l$, we define  
$$K(t,s):=  R\rho(t,s)^{k} s^k \l^l(t-s)^l  $$
and
\[ \Xt_X(R,k,l,\l,r):=\{(t,s,x) \in X\;|\;  |x| <K(t,s), \l t<s,\ \sum_{n \geq 0}  K(t,s)^{1-2^n}\rho(t,s)^{kn}\l^{ln}x^{2^n} \leq r\}.\]
\end{definition}
 
 \begin{theorem}\label{T::D2} Fix  $R, k, l \in \RM_{>0}$ and $0 <\l  < 1$ and consider the map
\[ f: P \times \RM_{\geq 0} \to  P \times \RM_{\geq 0}, (t,s,x,\a) \mapsto (s,s-\l (t-s), \frac{x^2}{R s^k(t-s)^l},x+\a) . \]
Then the following holds : 
\begin{enumerate}[{\rm i)}]
\item The iterates of the set  $\Xt_P(R,k,l,\l,r) \times \{ 0 \}$ are mapped inside $\Xt_P(R,k,l,\l,r) \times [ 0,r] $.
\item For any $(t_0,s_0,x_0) \in \Xt_P(R,k,l,\l)$ the third component of the sequence of iterates 
$$(t_n,s_n,x_n,\a_n)=f^{(n)}(t_0,s_0,x_0,\a_0)$$ converges to zero rapidly and more precisely:
$$x_n = \left(  Rp_n^k s_0^k\l^l(t_0-s_0)^l \right)^{1-2^n} \l^{l n} x_0^{2^n}. $$
with $p_n=\prod_{i=0}^{n-1} \rho(t_i,s_i)$, $p_0=1$.
\end{enumerate}
\end{theorem}
The proof given in the homogeneous applies mutatis mutandis, as well as its corollary
\begin{theorem}
\label{T::Fixedpoint_parametric} 
Let $\Et \to \D$, $\Ft \to \D$ be relative Banach spaces over the open sub-diagonal
$\D=\{ (t,s) \in ]0,S]^2: s<t \}$.
Assume that for a given map
\[ F:  \Xt_\Et(R,k,l,m,\l,r) \times rB_{\Ft} \to \Et \oplus \Ft,\ (t,s,v,\a) \mapsto (s,s-\l(t-s),g(t,s,v,\a))\]
there is an ordered diagram
$$\xymatrix{ \Xt_\Et(R,k,l,m,\l,r)  \times rB_{\Ft} \ar[r]^-F \ar[d]_{\nu_\Et}   \ar@{}[rd]|{\geq} &  \Et \ar[d]^{\nu_\Et}\\
\Xt_P(R,k,l,m,\l,r)   \times \RM_{\geq 0}  \ar[r]_-{f} & P \times \RM_{\geq 0} 
}$$
with
\[ f: \Xt_P(R,k,l,m,\l) \to  P, (t,s,x) \mapsto (s,s-\l (t-s),\frac{x^m}{Rs^k(t-s)^l} ) . \]
Then  $\Xt_\Et (R,k,l,m,\l)  \times \{ 0 \}$ is mapped inside $\Xt_\Et(R,k,l,\l,r) \times [ 0,r] $
and the iterates  
\[(x_{n},\a_{n})=F^{(n)}(x_{0},0)\]
converges rapidly to $0$.
\end{theorem}
 \chapter{The normal form theorem}
In this chapter we will give a prove of a general normal form theorem. It 
is based on the Lie-iteration of chapter 5, but we now let the iteration
take place in a Kolmogorov space. Using the results of the previous chapters,
we will see that a single simple condition ensures convergence of the process.

\section{Anatomy of the Morse lemma} 
Before we formulate the general statements, we look in some  detail at the 
complex one-variable {\em Morse lemma} in the KAM context. It states
$$f_0(z)=\frac{1}{2}z^2+o(z^2) $$
there exists an biholomorphic germ 
$$\p:(\CM,0) \to (\CM,0) $$
such that 
$$f_0 \circ \p(z)=\frac{1}{2} z^2. $$
The result is of course straightforward, since
$$f_0(z)=\frac{1}{2}\left(z\sqrt{1+o(1)}\right)^2 .$$ 
It is  nevertheless instructive to see how this basic result 
fits inside the general theory of normal forms. 
We take the particular example
$$f_0(z)=\frac{1}{2} z^2+z^3 .$$
and want to use the Lie-iteration to bring this function back to 
its normal form
$$a(z):=\frac{1}{2} z^2 .$$

\noindent \subsection*{Lie iteration on power series}
We begin by studying the Lie-iteration on the space $R=\CM[[z]]$ of formal 
power series and define the ingredients for the Lie-iteration.
We let
\[ \alg:=Der(R,R) = R \d_z\]
the $R$-module of derivations. 

There is an infinitesimal action of $\alg$ on $R$, given by 
\[ Der(R,R) \times R \to R,\;\; (v,f) \mapsto v(f) ,\]
and the action on the normal form $a$ is simply multiplication by $z$:
$$\rho: Der(R,R) \to R,\ v(z)\d_z \mapsto v(z)\d_z a(z)=z v(z) .$$

As the map $\rho$ is not surjective, there does not exist a right inverse
to it, but the map  
$$j: R \to \Der(R,R) ,\ b(z) \mapsto \frac{1}{z}(b(z)-b(0))\d_z $$
is a right inverse to $\rho$, when restricted to the maximal ideal 
$$\Mt =\left\{ \sum_{n \geq 0} b_n z^n: b_0=0 \right\} .$$
This means that 
\[ \rho \circ j_{|\Mt} = Id_{| \Mt} .\]
As we will always apply the map $j$ to series with $b(0)=0$, the last term
in the formula could be omitted.

The Lie iteration produces a sequence of power series
\[ f_0:=f=a+b_0,\;\;f_1=a+b_1,\;\;f_2=a+b_2,\ldots\]
where the remainders $b_n$ determine elements $v_n \in Der(R,R)$ via the
formula
\[ v_n:=j(b_n),\]
which in turn determine the next term of the sequence
\[ f_1=e^{-v_0}(f_0),\;\;f_2=e^{-v_1}(f_1),\;\;\;\ldots,f_{n+1}=e^{-v_n}(f_n), \ldots\]

Written in terms of the remainders, we get:
\[ b_{n+1}=e^{-v_n}(a+b_n)-a,\ a =\frac{z^2}{2}\]
and in terms of the derivations:
$$v_{n+1}=j(e^{-v_n}(a+v_n(a))-a). $$
So the Lie iteration produces a sequence
\[v_0,\;\;v_1,\;\;v_2,\ldots, v_n,\ldots \]
of elements in $Der(R,R)$ by the recursion:
\vskip0.3cm
$$ \begin{array}{| c|}\hline \\ \qquad v_{n+1}=j(e^{-v_n}(\Id+v_n)(z^2/2)-z^2/2) \qquad \\ \ \\ \hline \end{array}$$
\vskip0.3cm
\newpage
At the first step, we get  .
\[f_0=\frac{z^2}{2}+z^3,\;\;b_0=z^3,\;\;v_0=j(z^3)=z^2\d_z .\]

The first terms of the sequence are seen to be:
\begin{align*}
 v_0&=z^2 \d_z\\
 v_1&=\frac{1}{2}(-3z^3+8z^4-15z^5+24z^6-35z^7+48z^8-63z^9+o(z^9))\d_z\\
 v_2&=\frac{1}{4}(-18z^5+108z^6-493z^7+2100z^8-8579z^9+o(z^9))\d_z\\
 v_3&=\frac{1}{4}(-243z^9+o(z^9))\d_z\\
 v_4&=\frac{-295245}{16}z^{17}+o(z^{17}))\d_z
\end{align*}
The corresponding functions are:
\begin{align*}
 f_0&=\frac{1}{2} z^2+z^3\\
 f_1&=\frac{1}{2}(z^2-3z^4+8z^5-15z^6+24z^7-35z^8+48z^9-63z^{10})+o(z^{10})\\
 f_2&=\frac{1}{4}(2z^2-18z^6+108z^7-493z^8+2100z^9-8579z^{10})+o(z^{10})\\
 f_3&=\frac{1}{2}z^2 - \frac{243}{4}z^{10}+o(z^{10})\\
 f_4&=\frac{1}{2} z^2- \frac{295245}{16} z^{18}+o(z^{18})
\end{align*}

From this example one can observe that the order of the vector fields $v_i$ is equal to $2^{i}+1$, and  the order of remainders $b_i=f_i-a$ is $2^i+2$. We see that even at the level of formal series, the iteration scheme is 
converging rapidly in the $\Mt$-adic topology of $R$!

The composition
\[ \ldots e^{-v_2}e^{-v_1}e^{-v_0}\]
applied to series $z=0+1.z+0.z^2+\ldots \in R$ is computed to be the power series
\[ \psi(z):=z-z^2+\frac{5}{2}z^3-8z^4+\frac{231}{8}z^5-\ldots\]
so it is the automorphism 
\[ z \mapsto \psi(z)\]
that maps $f_0$ to the normal form $a$. Indeed, 
$\psi$ is seen to be the {\em inverse power series} to
\[ z \mapsto \phi(z)=z\sqrt{1+2z}=z+z^2-\frac{1}{2}z^3+\frac{1}{2}z^4-\ldots\]
for which
\[a(\phi(z))=f_0(z),\;\;\; \frac{1}{2}(z\sqrt{1+2z})^2=\frac{1}{2}z^2+z^3\]

\noindent \subsection*{The Lie iteration and the fixed point theorem}
The above procedure involves infinitely many coordinate transformations and  it is not clear that 
this infinite composition is an analytic coordinate transformation. Therefore we lift the  above calculation to the level of an appropriate Kolmogorov space and then deduce the result
 in $\CM\{z\}$

As usual we denote by 
$$D \to \RM_{>0}$$ 
the relative unit disc in $\CM$ with fibre
\[ D_s:=\{ z \in \CM\;|\;|z| < s\} , \]
and consider the Kolmogorov space  $E:=\Ot^c(D)$. 
The normal form  $a(z):=\frac{1}{2}z^2$ defines a section $a$ of $E$.
 We set 
\[\alg:= Der(E,E):=\Ot^c(D)\d_z \subset L^1(E,E) .\]
The infinitesimal action $\rho$ is defined as before
\[\rho: Der(E,E) \to E,\;\;\; v(z)\d_z \to v(z)\d_za(z)=z v(z)\]
and the map $j$ can be defined similarly:
$$j: E \to  Der(E,E),\ a(z) \mapsto \frac{1}{z}(b(z)-b(0))\d_z.$$
So the formulas remain the same, but the interpretation of the symbols is
slightly different. As $j$ involves a division, the map  $j$ is a local 
morphism  with simple pole along $s=0$:
$$j \in \Hom^{1,0}_{\D}(E,Der(E,E)).$$

To define the Lie iteration in this setting, we need an additional ingredient:
we have to choose a falling converging sequence 
$$  t_0 > t_1 > t_2 >\ldots >t_{\infty} >0$$
in such a way that that $u_n$ is holomorphic on the open disc of radius
$t_n$, continuous on the boundary, i.e.: 
\[ v_n \in \Ot(D)_{t_n}\]
We put $s_n:=t_{n+1}$ and generate the sequence $(t_n,s_n)$ by iterating the
map
\[ \phi: (t,s) \mapsto (s,t-\l (t-s)) \]
with $(t_0,s_0)$ as starting point. The Lie iteration produces a sequence
\[z^2\d_z=:v_0,\;\;v_1,\;\;v_2,\;\;\ldots, v_n,\ldots \]
of elements in $Der(E,E)$ by the recursion
$$ v_{n+1}=j(e^{-v_n}(\Id+v_n)(z^2/2)-z^2/2).$$
So we wish to iterate the following non-linear map of relative Banach spaces:
\[ F: Der(E,E) \to Der(E,E),\;\;\;(t,s,v) \mapsto (s,s-\l(t-s),j(g(v)a))\]
with
$$g(z)=e^{-z}(1+z)-1 \in \CM\{z\} .$$
The series $g$ is easily seen to be the Borel transform of 
\[f(z):=- \frac{z^2}{(1+z)^2} \in z^2 \CM\{z\}, \]
which has radius of convergence equal to $1$.
This holomorphic function admits a quadratic bound in any disc of 
radius $r<1$. For instance, if we take $r=1/2$ we have:
\[|f(z)| \leq |f|(|z|) \leq 4 |z|^2, \;\;\textup{for}\;\;|z| \leq \frac{1}{2} .\]
In order to apply the Borel bound, we have also to consider
the K2-space
\[ \Dt er(E,E):=p^*\alg \subset p^*L^1(E,E)=\Lt^1(E,E), \;\;p:(t,s) \mapsto t .\]
By \ref{T::Borel}, the map
$$Bf=g: v \mapsto e^{-v}(\Id+v)-\Id $$
defined on
$$\Xt\left(1/2\right):=\{u= (t,s,v) \in \Dt er(E,E) |\;\;  \|u \| < \frac{1}{2}(t-s)\} $$
satisfies
$$|g(u)| \leq |f|(\frac{\| u \|}{(t-s)}) \le 4 \frac{\| u \|^2}{(t-s)^2}.$$
In particular
$$|g(u)(a)| \leq 4 \frac{\| u \|^2}{(t-s)^2}|a(z)|_t=2 t^2 \frac{\| u \|^2}{(t-s)^2}$$
and composition with $j$ gives
$$|j g(u)(a)| \leq  2 t^2 \frac{\| u \|^2}{s(t-s)^2} .$$
It is basically this map that we have to iterate.

To apply the fixed point theorem,\ref{T::D2} and we have to choose 
parameters and make sure all conditions are satisfied. More precisely,
we have to find constants $R >0,0 <\l <1$ and initial data
\[u_0=(t_0,s_0,v_0) \in \Xt\left(1/2\right) \cap \Xt(R,1,2,\l)\]
where
\[\Xt(R,1,2,\l)=\{u=(t,s,v) \in \Dt er(E,E) :  \| u \| <R\rho(t,s)s\l^2(t-s)^2,\ s>\l t \}\]
and
$$\rho(t,s):=\frac{s-\l(t-s)}{s}=1+\l-\l \frac{t}{s}.$$

We pull-back the space $E$ via the inclusion
$$\phi:]0,t_0] \to \RM_{>0} .$$
Then choosing $C = 2t_0^2$ as constant we get that
$$|j g(u)(a)| \leq  C \frac{\| u \|^2}{s(t-s)^2} .$$
which means that the map induced by $j \circ g$ on $E[t_0]=\phi^*E$ is exactly of the type considered in the previous chapter, with ramification indices
$(-1,-2,2)$.

Let us now investigate in details how we can ascertain the condition that $u_0 \in \Xt(1/2)$.
As before we have
\[u_0=j(z^3)=z^2\d_z\]
and from the Cauchy-Nagumo inequalities, we find the estimate
$$| u_0(f) |_{s_0} \leq |z^2|_{s_0} \times |\d_z f|_{s_0} \leq \frac{s_0^2}{t_0-s_0}|f|_{t_0}, $$
thus for the renormalised norm of $u_0 \in \Lt^1(E,E)$ we have the inequality:
$$ \| u_0 \| \leq e\,s_0^2 .$$
So if we ensure that:
\begin{enumerate}[{\rm i)}]
\item $e s_0^2 \le \frac{1}{2}(t_0-s_0)$
\end{enumerate}
then $\frac{\| u_0\|}{t_0-s_0} \le \frac{1}{2}$, so that indeed $u_0\in \Xt(1/2)$.

Furthermore, we have to require our starting point $(t_0,s_0,v_0)$ 
to be in the set $\Xt(R,1,2,\l)$, $R=C^{-1}$. This can be achieved by 
choosing $t_0$ small enough but let us spell out the two additional
conditions. 

\begin{enumerate} 
\item[{\rm ii)}] $ C \| u_0 \| < \rho(t_0,s_0)s_0\l^2(t_0-s_0)^2$
\item[{\rm iii)}]$ s_0 >\l t_0.$ 
\end{enumerate}
which can be achieved for $t_0$ small enough.

Condition iii) is trivial to fulfill: we choose any number $\m \in (\l,1)$
and set $s_0=\m t_0$. 

Using $\|u_0\| \leq e s_0^2$, $C=2t_0^2$, we see that $(ii)$ is satisfied 
if
\[ 2\, e\, t_0^2s_0^2 < \rho(t_0,s_0) s_0 \l^2(t_0-s_0)^2, \]
which holds for $t_0$ small enough.  

Let us now estimate precisely how small $t_0$ should be. Using $\m$ with $0 < \l <\m <1$ and $s_0=\m t_0$, we can rewrite 
$i)$ and $ii)$ as

$$
\begin{array}{lrcl}
i)&e\,t_0 &\leq &\frac{1-\m}{2\m^2}\\[2mm]
ii)& e\,t_0 &<& \frac{(1+\l-\l/\m) \l^2 (1-\m)^2}{2 \m}\\[2mm]
\end{array}
$$
If we take $\m=1/2, \l=1/4$ the first two inequalities 
evaluate to
$$
\begin{array}{lrcl}
i) & e\,t_0 & \leq& 1\\ 
ii)& e\,t_0 & < & 3/256=0.01171875000...\\
\end{array}
$$
so we find that it is sufficient to take $t_0 < T_0$, where
$$ T_0 := \frac{3}{256 e}=0.00431108720123...$$
Applying the fixed point theorem, we get that the sequence 
$$u_n=(t_n,s_n,v_n) $$
converges rapidly to zero. 

The theorem on the composition of exponentials 
(Theorem~\ref{T::compexponential2}) shows that the morphisms
$$g_n:=e^{-u_n} \dots e^{-u_1}  e^{-u_0},\ t_\infty=\lim t_n $$
converge to an element 
$$g_{\infty} \in \Hom(E,E)_{t_\infty\,t_0}=\Hom(E_{t_0},E_{t_{\infty}}).$$
which conjugates $f_0$ to the normal form.  One has
\[ t_{\infty}=\frac{s_0-\l t_0}{1-\l}=\frac{\m-\l}{1-\l} t_0, \]
which with the given values give
\[ \frac{1}{3}T_0=1/256\, e = 0.001437029067\ldots > 0\]
as an estimate for the radius of convergence of the conjugating 
transformation. 

We conclude that there exists a holomorphic coordinate 
transformation that transforms $f_0= \frac{1}{2}z^2+z^3$ back to 
$a=\frac{1}{2}z^2$.

As one can see from this example, the arguments we used are of a very 
general nature and are not restricted to the particular case of the Morse 
lemma.
\newpage
\section{Homogeneous normal forms}

The computations and argument we used for the Morse lemma can be done in 
a more abstract framework. We start by formulating what we mean by an 
infinitesimal action  on an affine subspace in an $S$-Kolmogorov space $E$.

\begin{definition} A {\em singularity class} $(a,M)$ in an $S$-Kolmogorov 
space $E$ consists of\\
\begin{enumerate}[{\rm i)}]
\item  a horizontal section $a$ of $E$, $\;\;a \in \Gamma^{h}(]0,S],E)$.
\item an $S$-Kolmogorov subspace $M \subset E$.
\end{enumerate}
\end{definition}
Recall that such a horizontal section $a$ is determined by
a single vector $a_S \in E_S$ by restriction to $E_t, t\le S$. 
A singularity class $(a,M)$ defines a relative {\em affine subspace} 
$$a+M \subset E \to ]0,S]$$ 
with fibre
\[a_t +M_t \subset E_t\]

Furthermore, we consider a given S-Kolmogorov subspace
\[ \alg \subset L^1(E,E) \]
that will play the role of Lie-algebra. 

\begin{definition} A {\em Kolmogorov action} of $\alg$ on $a+M  \subset E$
is given by\\ 
\[ \rho: \alg \to M;\;\;u \mapsto u(a),\;\;\rho \in \Hom_{\D}^{0,1}(\alg,M)\]
such that for each $u \in \Xt_\alg$, its exponential maps $a_t+M_t$ to $a_s+M_s$.
\end{definition}
 
Our first general normal form theorem is the following:

\begin{theorem}\label{T::groupes} 
Let $E$ be an $S$-Kolmogorov space and let $\alg \subset L^1(E,E)$ define
a Kolmogorov-action on an affine subspace $a+M \subset E$. 
Assume that the map 
$$\rho:\alg \to M, u \mapsto u(a)$$ 
admits a local right quasi-inverse
$$j \in \Hom^{k,l}_{\D}(M,\alg). $$
Choose $0 < \l < \m <1$ and $0 <r <1$ and let 
\[ R:=2^{-k-2} \frac{(1-r)^2}{\|j\||a_S|}.\] 
Then for any section
\[ b \in \Gamma^{h}(]0,S],M)\]
with the property that 
$$j(b) \in  \Xt(r) \cap  \Xt(R,k,l,\l) $$
there exists an element 
\[ g \in Hom(E_S,E_t)\]
with the property that
\[ g(a_S+b_S) =a_{t} .\]
\end{theorem}
\begin{proof}
The map $g$ will be constructed using the Lie iteration scheme:
\begin{enumerate}[{\rm 1)}]
\item $b_{n+1}:=e^{-u_n}(a+b_n)-a$ ;
\item  $u_{n+1}:=j(b_{n+1}),$
\end{enumerate}
The sequence $(u_n)$ is obtained by iteration from the formula:
$$u_{n+1}=j(e^{-u_n}(\Id+u_n)(a)-a):=F(u_n). $$
We apply this in the context of Kolmogorov spaces. 
As before in \ref{P::baseiteration}, by iterating
\[ \phi: (t,s) \mapsto (s, s-\l(t-s)) \]
on an initial point
\[ (t_0,s_0)=(t_0,\m t_0) \]
we create a sequence $(t_0,s_0), (t_1,s_1), \ldots$
converging to $(t_{\infty},t_{\infty})$.

The basic step of the iteration is the composition of two 
maps. First, given $u_n \in \alg_t$ we have to form the
exponential expression 
$$e^{-u_n}(\Id+u_n)-\Id.$$ 
This produces a map from $M_t$ to some $M_m$, for some $m <t$. 
Then we apply $j$, which is a $(k,l)$-local map, so we can take the
component $j_{sm}$, which maps $M_m$ to $\alg_s$ for some $s < m$. 
Finally we will need appropriate estimates, and ascertain the procedure 
can be iterated.\\
The power-series
$$e^{-z}(1+z)-1 \in \CM\{z\}$$
is the Borel transform of 
\[f(z):=- \frac{z^2}{(1+z)^2} \in z^2 \CM\{z\} .\]
which has radius of convergence equal to $1$. 

We pull-back $\alg$ via the map
$$p:(t,s) \mapsto t $$
and get a K2-space
$$p^*\alg \subset \Lt^1(E,E) $$
of $1$-local operators.
 
For any $r <1$ and $u \in \Xt(r)$, we have the Borel estimate \ref{T::Borel}:
\[ |Bf(u)_{mt}| \le C \frac{||u||^2}{(m-t)^2},\;\;C:=\frac{1}{(1-r)^2}.\]
From the assumed locality of $j$ there is an estimate
\[ |j_{sm}| \le \frac{\|j\|}{s^k(m-s)^l}\]
If we take the $m=(s+t)/2$ we obtain the estimate
\[ |j(Bf(u)(a))_{st}| \le  \frac{||u||^2}{R s^k(t-s)^{l+2}}\]
where
\[ R:=  2^{-k-2} \frac{(1-r)^2}{\|j\| |a_S|}  . \]
(The choice $m=\frac{2}{k+2}t+\frac{k}{k+2}$ would lead to a slightly better
constant, but this is of no importance here.)
We obtain a map with ramification indices $(-k,-l,2)$:
\[F: \Xt(r)  \cap \Xt(R,k,l,2,\l)  \to p^*\alg,\ (t,s,u_{st}) \mapsto (s,s-\l(t-s),j(Bf(u)(a))_{st})\] 
Using this map, we have a way of coming from an element
$u_n \in p^*\alg_{t,s}$ to an element $u_{n+1} \in p^*\alg_{s,r}$. 

The fixed point theorem (Theorem~\ref{T::Fixedpoint}) shows that the map can 
be iterated and that the sequence of $(u_i)$ converges rapidly to $0$. The elements $u_0,u_1,u_2,\ldots$ then have exponentials
\[ e^{-u_i}: E_{t_i} \to E_{t_{i+1}}.\]
The theorem on the composition of exponentials (Theorem~\ref{T::compexponential2}) shows that the elements
$$g_n:=e^{-u_n} \dots e^{-u_1}  e^{-u_0} $$
converge to an element 
$$g_{\infty} \in \Hom(E,E)_{t_\infty\,t_0}=\Hom(E_{t_0},E_{t_{\infty}})$$
with $t_\infty=\lim t_n$. As $(u_n)$ converges to zero, we get that
\begin{align*}
g_n (a_{t_0})&=a_{t_n}+u_n(a_{t_n}) \\
\lim_{n \to +\infty} g_n (a_{t_0})&=a_{t_\infty}+\lim_{n \to +\infty} u_{n}(a_{t_n})\\
  &  = a_{t_\infty} 
 \end{align*}
This concludes the proof of the theorem.
\end{proof}
The theorem we proved is in fact more precise: 
 \begin{theorem}\label{T::groupes} 
Under the assumptions of the previous theorem, the Lie-iteration produces a sequence of elements 
$$u_0,u_1, u_2, \ldots,\;\;\;\; u_i=(t_i,v_i) \in \alg,$$
with exponentials
\[e^{u_i}: E_{t_i} \to E_{t_{i+1}}\]
such that the sequence of products 
\[ g_n:=e^{-u_{n}}\ldots e^{-u_1} e^{-u_0}:E_{t_0} \to E_{t_{\infty}}\]
converge to a well-defined element 
\[ g_{\infty} \in Hom(E_{t_0},E_{t_{\infty}})\]
with the property that
\[ g_{\infty}(a_S+b_S) =a_{t_{\infty}} .\]
\end{theorem}

\section{The normal form theorem (general case)}
In the parametric case, we consider families of local maps
$$j_\a  \in \Hom^{k,l}_{\D}(M,\alg) $$
depending on a parameter $\a$ which belongs to the unit ball  of some auxiliary Kolmogorov space $F$.

We say that a family is uniformly bounded if the norms
$\| j_\a \| $ are bounded for  $\a \in  B_F$. We then define
$$\| j \|:=\sup_{\a \in B_F} \| j_\a \| $$
\begin{theorem}\label{T::groupes} 
Let $E$ be an $S$-Kolmogorov space and $F \subset E$ a Kolmogorov subspace. Let $\alg \subset L^1(E,E)$ define
a Kolmogorov-action on an affine subspace $a+M \subset E$. 
Assume that the map 
$$\rho_\a:\alg \to M, u \mapsto u(a+\a)$$ 
admit uniformly bounded local right quasi-inverses
$$j_\a \in \Hom^{k,l}_{\D}(M,\alg),\ \a \in  B_F $$
Choose $0 < \l < \a <1$ and $0 <r <1$ and let 
\[ R:=2^{-k-3} \frac{(1-r)^2}{\|j\||a_S|}.\] 
Then for any subsection
\[ b \in \Gamma^{h}(]0,S],M)\]
with the property that 
$$j_\a(b) \in  \Xt(r) \cap  \Xt(R,k,l,\l,r,1) $$
there exists an element 
\[ g \in Hom(E_S,E_t)\]
with the property that
\[ g(a_S+b_S) -a_{t} \in F_t .\]
\end{theorem}
\begin{proof}
The map $g$ will be constructed using the Lie iteration:
\begin{enumerate}[1)]
\item $a_{n+1}:=a_n+\a_n$ ;
\item $b_{n+1}:=e^{-u_n}(a_n+b_n)-a_{n+1}$ ;
\item  $(\a_{n+1},u_{n+1}):=j(a_{n+1},b_{n+1}).$
\end{enumerate}
with $ b_0=b,\ u_0=j(b) $ and 
and by iterating
\[ \phi: (t,s) \mapsto (s, s-\l(t-s)) \]
from an initial point
\[ (t_0,s_0)=(t_0,\a t_0). \]
This iteration 
\[(a_{n+1},(\a_{n+1},u_{n+1}))=(a_n+\a_n,j(a_n+\a_n,Bf(u_n,a_n)+Bg(u_n,\a_n)))\]
is defined via the functions
$$Bf(u)= (e^{-u}(\Id+u)-\Id) (a),\ Bg(u,\a)=(e^{-u}-\Id) (\a)\;. $$
which are both a quadratic at $u=0,\ \a=0$. We put $v_n:=(\a_n,u_n)$.

We have seen that for any $r <1$ and $u \in \Xt(r)$, we have the Borel estimate \ref{T::Borel}:
\[ |Bf(u)_{mt}| \le C \frac{||u||^2}{(m-t)^2},\;\;C:=\frac{1}{(1-r)^2}.\]
The series
$$Bg(z)= e^{-z}-1 $$
is the  Borel transform of
$$g(z)=\frac{-z}{1+z} . $$
Therefore the Borel estimate gives:
\[ |Bg(u,\a)_{mt}| \le C' \frac{\|u\| \, \| \a \|}{(m-t)^2},\;\;C':=\frac{r}{(1-r)}<C.\]
As $j_\a$ is local, there is an estimate of the form
\[ |j_{\a,sm}| \le \frac{\|j\|}{s^k(m-s)^l}.\]
If we take the $m=(s+t)/2$ we obtain the estimate
\[ |j_\a(B(f+g)(u,\a)(a))_{st}| \le  \frac{||u||^2+\|u\| \, \| \a \|}{R s^k(t-s)^{l+2}}\]
where
\[ R:=  2^{-k-2} \frac{(1-r)^2}{\|j\| |a_S|}  . \]
Moreover as
$$||(u,\a)||:=\max(||(u)||,||(\a)||) $$
on the product space, we have
$$||u||^2+\|u\| \, \| \a \| \leq 2||(u,\a)||^2 $$
 The parametric fixed point theorem (Theorem~\ref{T::Fixedpoint_parametric}) shows that the map can 
be iterated and that the sequence of $(u,\a)$'s converges rapidly to $0$. The elements $u_0,u_1,u_2,\ldots$ then have exponentials
\[ e^{-u_i}: E_{t_i} \to E_{t_{i+1}}.\]
The theorem on the composition of exponentials (Theorem~\ref{T::compexponential2}) shows that the elements
$$g_n:=e^{-u_n} \dots e^{-u_1}  e^{-u_0} $$
converge to an element $g_{\infty}$
As $(u_n)$ converges to zero, we get that
\begin{align*}
g_n (a_{t_0})&=a_{t_n}+\sum_{i=0}^n\a_{t_i}+u_n(a_{t_n}) \\
\lim_{n \to +\infty} g_n (a_{t_0})&=a_{t_\infty}+\sum_{i \geq 0}\a_{t_i}+\lim_{n \to +\infty} u_{n}(a_{t_n})\\
  &  = a_{t_\infty} +\sum_{i \geq 0}\a_{t_i}
 \end{align*}
This concludes the proof of the theorem.
\end{proof}
%

\section{A heuristic for optimising the parameter choice}
The iteration scheme we used depends on the parameters
$r,\l,\a$ and the estimate for the final domain of 
convergence depends crucially on these choices. 

Let us return to the calculations on the Morse lemma. We saw that
with the choice
$$C=2t_0^2,\;\;r=\frac{1}{2},\;\;\l =\frac{1}{4},\;\;\;\m=\frac{1}{2},$$
we found convergence for the starting value
\[ t_0 < \frac{3}{256\,e} =0.004311087202\ldots\]
and thus we obtain $t_{\infty}=t_0/3 <0.001437029067\ldots$ as radius of 
convergence for the coordinate transformation. 
It is this value that  should be compared to the 'true' radius of 
convergence of the series $\psi(z)$, inverse to $z\sqrt{1+2z}$, which is 
computed to be
\[ \frac{\sqrt{3}}{9}=0.1924500898\ldots, \]
which is about $134$ times larger.

If we keep our choices for $C$ and $r$, the best choices for $\l$ and
$\m$ are determined by finding the maximum of the rational function
\[ F(\m,\l):=\frac{(1+\l-\l/\m)\l^2(1-\m)^2}{2\m}\cdot \frac{\m-\l}{(1-\l)}\]
that computes $e\, t_{\infty}$.
The level curves of $F$ suggest there is a single maximum in the region
$0 <\l <\m < 1$:
\begin{center}
\includegraphics[height=8cm]{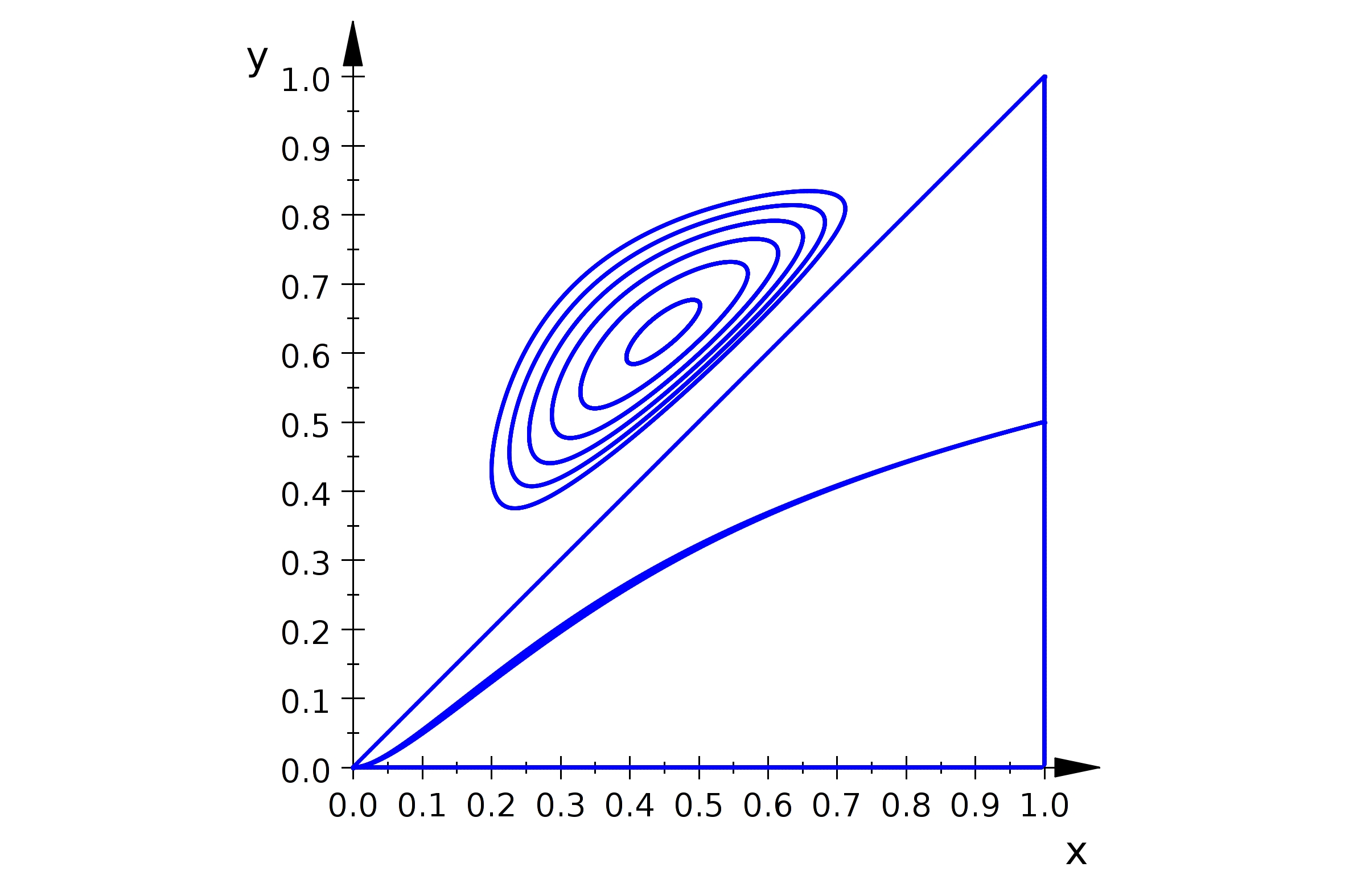}
($x=\l$, $y=\m$.)
\end{center}
A direct calculation shows that a
maximum is obtained for 
\[ \l=8\m^2+2\m-4=0.448612476...\]
where
\[ \m=0.6311094891...\]
is a root of the cubic equation $8\m^3-4\m^2-7\m+4=0$, which
leads to a slightly larger value
\[t_{\infty}=0.001949102953\ldots\]
for the radius of convergence.\\ 

However, we can also play with the parameter $r$. In general we have for
$f(z)=-z^2/(1+z)^2$ the inequality
\[ |f(z)|\le \frac{1}{(1-r)^2} |z|^2\;\;\;\textup{for}\;\;|z| \le r ,\]
which leads to
\[|j(g(u)(a))| \le \frac{t^2}{2}\frac{1}{(1-r)^2}\frac{\|u\|^2}{(t-s)^2}\]
as long as $\|u\| \le  r(t-s)$. The
inequalities now become
\[  e\, s_0^2 \le r(t_0-s_0)\]
\[ \frac{t_0^2}{2(1-r)^2}  e\, s_0^2 \le \rho(t_0,s_0) s_0 \l^2(t_0-s_0)^2\]
which using $\m$ can be written as:
\begin{align*}
e\,t_0 \le& r(1-\m)/\m^2\\
e\, t_0 \le& 2(1-r)^2 \rho(t_0,s_0)\l^2(1-\m)^2/\m .
\end{align*}
With the choice $r=1/2$ we have seen that the second inequalities is much
stronger then the first.\\

{\bf Heuristic.} {\em As we try to maximise the minimum of the two above
quantities, the maximum is realised when both inequalities coincide.}

This heuristic is supported by numerical calculations. 

Equating the right hand sides leads to a condition for $r$, depending
on $\l$ and $\m$ gives:
\[ \frac{r}{(1-r)^2} = 2 \rho \l^2 (1-\m)\m=:\nu(\l,\m).\]
We solve $r$ as function of $\m$ and $\l$ and find:
\[ r(\l,\m)=\frac{1+2\mu-\sqrt{1+4\mu}}{2\mu}\]
The level sets of the function
\[ r(\l,\m)\frac{(1-\m)}{\m^2}\frac{(\m-\l)}{(1-\l)} (=et_{\infty}) \]
look very much like what we had before.

\begin{center}
\includegraphics[height=8cm]{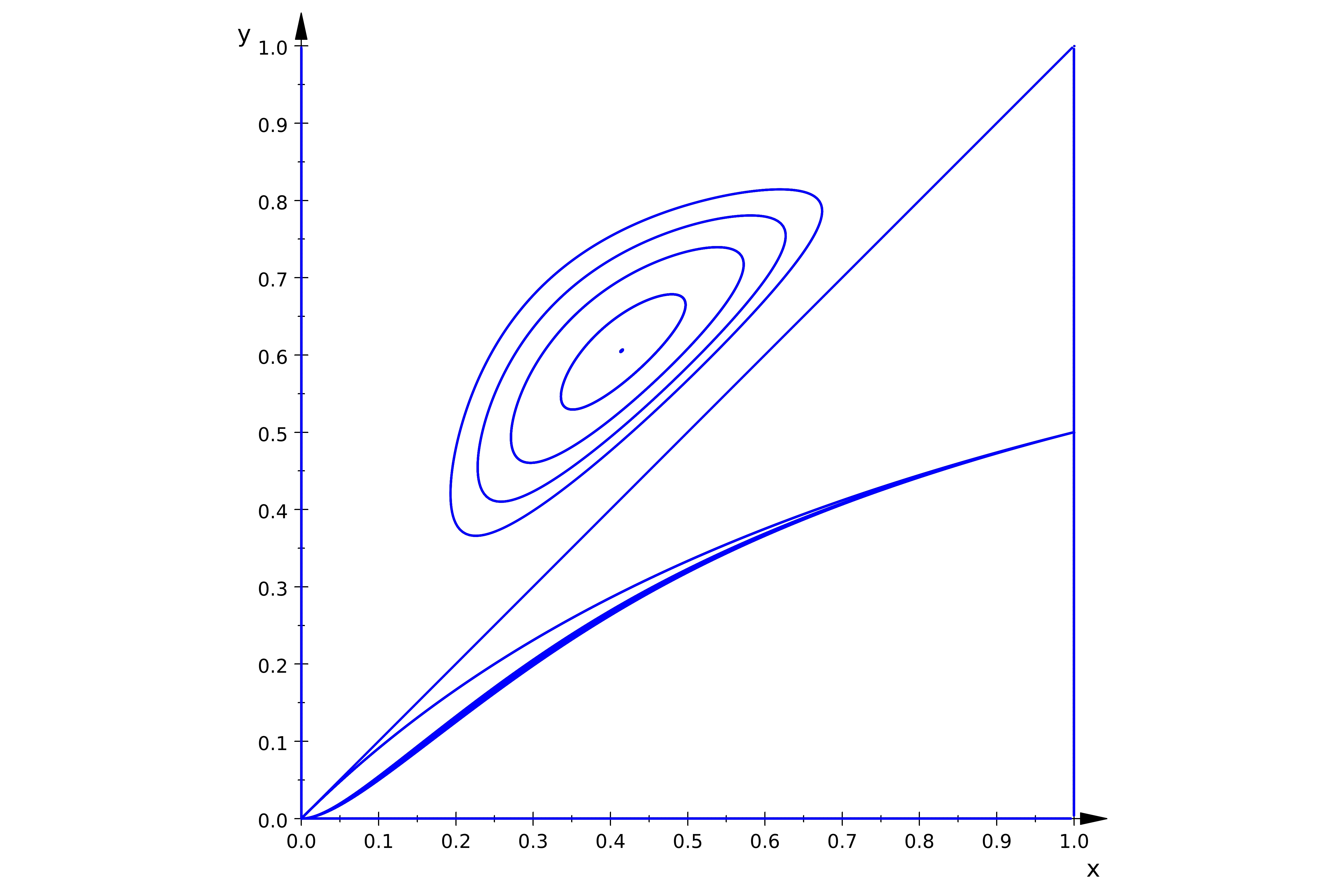}

($x=\l$, $y=\m$.)
\end{center}

It attains it maximum value
\[ e\,t_{\infty}=0.01883436563\ldots,\;\;\;t_{\infty}=0.006928775903\ldots\]
for the values 
\[\l=0.4145716992\ldots,\;\;\;\;\m=0.6054472202\dots\]
which turn out to be rather complicated algebraic numbers of degree~$14$.
We remark that the true radius of convergence is bigger by a factor
\[ Q:=0.00692877/0.192450=27.7754... \]

Finally, we can apply the same method to the function 
\[ \frac{1}{2}z^2+\beta z^n,\;\;\;\beta >0, n \ge 3 .\]
In this case we find 
$$\|u_0\| \le e \beta s_0^{n-1}$$
and the two inequalities reduce to
$$
\begin{array}{lrcl}
i) &e \beta t^{n-2} &\le& \frac{r(1-\mu)}{\mu^{n-1}}\\
ii)&e \beta t^{n-2} &< &2(1-r)^2\frac{\rho \l^2 (1-\m)^2}{\mu^{n-2}}\\
\end{array}
$$

If we equate these bound, we obtain exactly the same value for $r$:
\[r(\l,\m) =\frac{1+2\nu-\sqrt{1+4\nu}}{2\nu},\;\;\nu:=2\rho \l^2\mu(1-\mu),\]
and hence the following expression for the critical value for $t_{\infty}$:
\[t_{\infty}:=\Big(\frac{r(\l,\m)(1-\m)}{e\beta \mu}\Big)^{1/(n-2)}\frac{\m-\l}{\m (1-\l)}\]
This we have to compare with the radius of convergence of the coordinate 
transformation {\em inverse} to 
\[ z \mapsto z\sqrt{1+2\beta z^{n-2}},\]
which is computed to be
\[R(n,\beta):=\Big(\frac{1}{n \beta}\Big)^{1/(n-2)} \sqrt{1-\frac{2}{n}}.\]
The ratio between these quantities is
\[Q:=\frac{R(n,\beta)}{T_{\infty}(n,\l,\m)}=\frac{\m-\l}{\m(1-\l)}\frac{1}{\sqrt{1-2/n}} \Big(\frac{r(1-\m) n}{e \m}\Big)^{1/(n-2)}\]
For each values of $n$ one has the possibility of choosing appropriate
values $\l,\m$, to make this $Q$ as small as possible.\\

Below we summarise
the result of some numerical computations:\\
$$
\begin{array}{|c||c|c|c|c|c|c|c|c|c|c|}
\hline
n &3&4&5&6&7&8&9&10&20&50\\
\hline
\l & 0.414& 0.367& 0.338& 0.319& 0.301& 0.287&0.279& 0.259 &0.211 &0.154\\
\m & 0.605& 0.635& 0.655& 0.676& 0.687& 0.699&0.713& 0.719 &0.769 &0.847\\
Q  &27.775& 5.439& 3.153& 2.397& 2.033& 1.820&1.682& 1.584 &1.249 &1.099\\
\hline
\end{array}
$$
The optimal choices of $\l$, $\mu$ and $Q$ appear to depend 
monotonically on~$n$. As one sees from the data and the explicit formula, 
for $n \to \infty$ the quantity $Q$ tends to $1$.
So our method is asymptotically optimal, despite the fact that all estimates 
we used were rather crude. However it is probably,
that more refined methods lead to slightly better estimates for the radius 
of convergence.

\chapter*{Bibliographical notes on part II} 
The program of general normal form theory and group actions was considered in:\\

{\sc J. Moser}, {\em A rapidly convergent iteration method and non-linear partial differential equations II}, Ann. Scuola Norm Sup. Pisa, 20:3, 1966, 499-535.
   
The ideas of Moser led to the important works on implicit function theorems:\\

{\sc R.S Hamilton}, {\em The inverse function theorem of Nash and Moser}, Bull. Amer. Math. Soc., 7:1,1982, 65--222.\\  

{\sc F. Sergeraert}, {\em Un th\'eor\`eme de fonctions implicites sur certains espaces de Fr\'echet et quelques applications}, {Ann. Sci. \'Ecole Norm. Sup.}, 1972, 5:4,  p. 599--660.\\

{\sc E. Zehnder}, {\em Generalized implicit function theorems with applications to some small divisor problems I}, {Communications Pure Applied Mathematics}, 28, 1975, p. 91--140.\\
  
{J. F\'ejoz} discovered an elementary version of {Zehnder}s first implicit function theorem. The result was published later in:\\

{\sc J. F\'ejoz}, {\em Introduction to KAM theory, with a view to celestial mechanics in Variational methods in imaging and geometric control}, Radon Series on Comput. and Applied Math. 18, de Gruyter, 2016.\\
 
However implicit function theorems require a global linear inverse (or quasi-inverse) which
are seldom available, for instance in the presence of resonances.\\

An abstract theorem on infinite group dimensional group actions avoiding implicit 
function theorems was given in:\\

{\sc J. F\'ejoz et M. Garay}, {\em Un th\'eor\`eme sur les actions de groupes de dimension 
infinie}, {Comptes Rendus \`a l'Acad\'emie des Sciences} {\bf 348}, {7-8}, {427 - 430}, 2010.\\

Besides the problem of global inverse, it is clear now that the theorems could not be 
pushed further, because they presuppose locality for operators which are in general not local 
(see for instance Assumption H1 in Zehnder's paper). Therefore, applications of implicit 
function theorems require to parametrise the infinite dimensional groups in a non-standard 
way such as $\Id+v$ for diffeomorphisms, generating functions for symplectomorphisms, etc.
 
This explains that direct methods using power series turned out to be more powerful 
and sometimes simpler compared to the abstract approach, see for instance:\\

 {\sc R\"ussmann, H.},	 {\em Invariant tori in non-degenerate nearly integrable Hamiltonian systems}, {Regular and chaotic dynamics}, {2001}, {6:2}, p. 119-204.\\
  
Abstract KAM theory, including the theory of Kolmogorov spaces, was founded --although not in definitive form-- in:\\

{\sc M. Garay}, {\em An Abstract KAM Theorem}, Moscow Math. Journal, 14:4  p.745-772, 2014.\\

{\sc M. Garay}, {\em Degenerations of invariant Lagrangian manifolds}, {Journal of Singularities}, 8, p.{50-67}, {2014}.\\

It is only recently that the authors of this book discovered the general aspects of functorial analysis, came across the general definition of Kolmogorov spaces and discovered finite dimensional reduction using the norm map. 


\chapter*{Photo Credits}
{\em J. Moser} by Konrad Jacobs, Erlangen (1969).\\

 \end{document}